\documentclass[12pt]{article}
\usepackage{amsmath,amsthm,amssymb}

\usepackage{txfonts}	
\usepackage[vcentermath]{youngtab}
\usepackage{tabularx}
\usepackage{booktabs,array}
 \usepackage{mathtools}

\usepackage{color}
\definecolor{myColor}{gray}{0.975}

\newcommand{\create}[1]{\begin{#1}}
\newcommand{\delete}[1]{\end{#1}}

\newcommand{\Frac}[2]{\leavevmode\kern.10em\raise.50ex\hbox{\the\scriptfont0 $#1$}\kern+.20em
                         {\big /}\kern.20em\lower.80ex\hbox{\the\scriptfont0 $#2$}}

\newcommand{\Sum}[2]
{
	\lower.40ex\hbox{$
	\underset{#1}{\overset{#2}{\mathlarger{\mathlarger{\mathlarger \Sigma}}}}
	\hspace{.04in}$}
}
\newcommand{\cupEx}[2]
{
	\lower.40ex\hbox{$
	\underset{#1}{\overset{#2}{\mathlarger{\mathlarger{\mathlarger{ \cup}}}}}
	\hspace{.02in}$}
}
\newcommand{\capEx}[2]
{
	\lower.40ex\hbox{$
	\underset{#1}{\overset{#2}{\mathlarger{\mathlarger{\mathlarger{ \cap}}}}}
	\hspace{.02in}$}
}
\newcommand{\Prod}[2]
{
	\lower.40ex\hbox{$
	\underset{#1}{\overset{#2}{\mathlarger{\mathlarger{\mathlarger \Pi}}}}
	\hspace{.04in}$}
}
\newcommand{\Directsum}{\rotatebox[origin=c]{180}{\reflectbox{$\Pi$}}}
\newcommand{\Coprod}[2]
{
	\lower.40ex\hbox{$
	\underset{#1}{\overset{#2}{\mathlarger{\mathlarger{\mathlarger \Directsum}}}}
	\hspace{.04in}$}
}
\newcommand{\OPlus}[2]
{
	\lower.40ex\hbox{$
	\underset{#1}{\overset{#2}{\mathlarger{\mathlarger{\mathlarger \oplus}}}}
	\hspace{.02in}$}
}

\newcommand{\simto} {
	\rightarrow\hspace{-12.5pt}\raise4.pt\hbox{$\mathsmaller \sim$} \hspace{7.4pt}}
\newcommand{\simlongto} {\longrightarrow \hspace{-16.5pt}\raise4pt\hbox{$\sim$} \hspace{5.51pt}}
\newcommand{\simlongtoL}{
	\longrightarrow \hspace{-16.5pt}\raise4pt\hbox{$\sim$} \hspace{8.85pt}}

\newcommand{\coloniff}{\hspace{.1in}\raise0.06ex\hbox{:}\mspace{-10mu}\iff}

\newcommand{\sen}
{
	\noindent
	\hspace{-2em}
	\leavevmode
	\leaders\hrule depth-2.1pt height 3.3pt\hfill\kern0pt
	\hspace{-2em}
	\vspace{.5ex}
}

\newcommand{\senS}
{
	\noindent
	\hspace{-.6em}
	\leavevmode
	\leaders\hrule depth-2.1pt height 2.5pt\hfill\kern0pt
	\hspace{-.6em}
	\vspace{.5ex}
}

\newcommand{\id}{\text{\hspace{1pt}id}}
\newcommand{\Ker}{\text{Ker}\,}

\newcommand{\gl}{\text{gl}}

\newcommand{\hs}[1]{\hspace{#1}}
\newcommand{\vs}[1]{\vspace{#1}}

\newtheorem{theorem}{Theorem}[section]

\newtheorem{notation}[theorem]{Notation}

\newtheorem{proposition}[theorem]{Proposition}

\newtheorem{corollary}[theorem]{Corollary}

\newtheorem{lemma}[theorem]{Lemma}

\newtheorem{example}[theorem]{Example}

\theoremstyle{definition}
\newtheorem{definition}[theorem]{Definition}

\newtheorem*{remark}{Remark}

\numberwithin{equation}{section}

\usepackage{secdot}
\sectiondot{section}

\long\def\note#1\par{%
\leavevmode
\hangindent3em Note: #1}

\usepackage{mathtools}

\usepackage{amsthm,amssymb}
\usepackage{cancel}
\makeatletter
\def\@seccntformat#1{\csname the#1\endcsname.\hs{-.375em}\quad}
\makeatother
\usepackage{tikz}
\usetikzlibrary{matrix}

\begin{document}

\setlength{\abovedisplayskip}{1ex} 
\setlength{\belowdisplayskip}{1ex}

\renewcommand{\labelenumi}{(\rm \theenumi)}
\renewcommand{\labelenumii}{$(${\rm \alph{enumii}}$)$}

\newcommand*{\QEDA}{\hfill\ensuremath{\square}}

\newcommand{\Uq}{U\hs{-.15em}_q(\mathfrak{g})}
\newcommand{\Ut}{\tilde{U}\hs{-.15em}_q(\mathfrak{g})}
\newcommand{\Up}{U\hs{-.15em}_q^{\hs{.2em}+}
\hs{-.1em}(\mathfrak{g})}
\newcommand{\Um}{U\hs{-.15em}_q^{\hs{.2em}-}
\hs{-.1em}(\mathfrak{g})}
\newcommand{\Uz}{U\hs{-.15em}_q^{{\hs{.2em}0}}
\hs{-.0em}(\mathfrak{g})}

\newcommand{\Utp}{\tilde{U}\hs{-.15em}_q^{\hs{.2em}+}
\hs{-.1em}(\mathfrak{g})}
\newcommand{\Utm}{\tilde{U}\hs{-.15em}_q^{\hs{.2em}-}
\hs{-.1em}(\mathfrak{g})}
\newcommand{\Utz}{\tilde{U}\hs{-.15em}_q^{{\hs{.2em}0}}
\hs{-.0em}(\mathfrak{g})}

\newcommand{\Bb}{\hs{1pt}\overline{\hs{-1pt}B\hs{1pt}}\hs{-1pt}}

\newcommand{\NEWPAGE}{\newpage}
	\renewcommand{\NEWPAGE}{}

\newcommand{\VS}{\vs{4ex}}
	\renewcommand{\VS}{}
\newcommand{\cI}{{\mathcal I}}
\newcommand{\cA}{{\mathcal A}}
\newcommand{\cB}{{\mathcal B}}
\newcommand{\cC}{{\mathcal C}}
\newcommand{\cD}{{\mathcal D}}
\newcommand{\cF}{{\mathcal F}}
\newcommand{\cH}{{\mathcal H}}
\newcommand{\cK}{{\mathcal K}}
\newcommand{\cL}{{\mathcal L}}
\newcommand{\cM}{{\mathcal M}}
\newcommand{\cN}{{\mathcal N}}
\newcommand{\cO}{{\mathcal O}}
\newcommand{\cQ}{{\mathcal Q}}
\newcommand{\cS}{{\mathcal S}}
\newcommand{\cT}{{\mathcal T}}
\newcommand{\cV}{{\mathcal V}}
\newcommand{\cP}{{\mathcal P}}
\newcommand{\cR}{{\mathcal R}}
\newcommand{\cY}{{\mathcal Y}}
\newcommand{\fra}{\mathfrak a}
\newcommand{\frb}{\mathfrak b}
\newcommand{\frc}{\mathfrak c}
\newcommand{\frd}{\mathfrak d}
\newcommand{\fre}{\mathfrak e}
\newcommand{\frf}{\mathfrak f}
\newcommand{\frg}{\mathfrak g}
\newcommand{\frh}{\mathfrak h}
\newcommand{\fri}{\mathfrak i}
\newcommand{\frj}{\mathfrak j}
\newcommand{\frk}{\mathfrak k}
\newcommand{\frI}{\mathfrak I}
\newcommand{\fm}{\mathfrak m}
\newcommand{\frn}{\mathfrak n}
\newcommand{\frp}{\mathfrak p}
\newcommand{\fq}{\mathfrak q}
\newcommand{\frr}{\mathfrak r}
\newcommand{\frs}{\mathfrak s}
\newcommand{\frt}{\mathfrak t}
\newcommand{\fru}{\mathfrak u}
\newcommand{\frA}{\mathfrak A}
\newcommand{\frB}{\mathfrak B}
\newcommand{\frF}{\mathfrak F}
\newcommand{\frG}{\mathfrak G}
\newcommand{\frH}{\mathfrak H}
\newcommand{\frJ}{\mathfrak J}
\newcommand{\frN}{\mathfrak N}
\newcommand{\frP}{\mathfrak P}
\newcommand{\frQ}{\mathfrak Q}
\newcommand{\frT}{\mathfrak T}
\newcommand{\frU}{\mathfrak U}
\newcommand{\frV}{\mathfrak V}
\newcommand{\frX}{\mathfrak X}
\newcommand{\frY}{\mathfrak Y}
\newcommand{\frZ}{\mathfrak Z}
\newcommand{\rA}{\mathrm{A}}
\newcommand{\rC}{\mathrm{C}}
\newcommand{\rd}{\mathrm{d}}
\newcommand{\rB}{\mathrm{B}}
\newcommand{\rD}{\mathrm{D}}
\newcommand{\rE}{\mathrm{E}}
\newcommand{\rH}{\mathrm{H}}
\newcommand{\rK}{\mathrm{K}}
\newcommand{\rL}{\mathrm{L}}
\newcommand{\rM}{\mathrm{M}}
\newcommand{\rN}{\mathrm{N}}
\newcommand{\rR}{\mathrm{R}}
\newcommand{\rT}{\mathrm{T}}
\newcommand{\rZ}{\mathrm{Z}}
\newcommand{\bbA}{\mathbb A}
\newcommand{\bbB}{\mathbb B}
\newcommand{\bbC}{\mathbb C}
\newcommand{\bbG}{\mathbb G}
\newcommand{\bbF}{\mathbb F}
\newcommand{\bbH}{\mathbb H}
\newcommand{\bbK}{\mathbb K}
\newcommand{\bbP}{\mathbb P}
\newcommand{\bbN}{\mathbb N}
\newcommand{\bbQ}{\mathbb Q}
\newcommand{\bbR}{\mathbb R}
\newcommand{\bbV}{\mathbb V}
\newcommand{\bbX}{\mathbb X}
\newcommand{\bbY}{\mathbb Y}
\newcommand{\bbZ}{\mathbb Z}

\newcommand{\scI}{{\mathscr I}}
\newcommand{\scA}{{\mathscr A}}
\newcommand{\scB}{{\mathscr B}}
\newcommand{\scC}{{\mathscr C}}
\newcommand{\scD}{{\mathscr D}}
\newcommand{\scF}{{\mathscr F}}
\newcommand{\scH}{{\mathscr H}}
\newcommand{\scK}{{\mathscr K}}
\newcommand{\scL}{{\mathscr L}}
\newcommand{\scM}{{\mathscr M}}
\newcommand{\scN}{{\mathscr N}}
\newcommand{\scO}{{\mathscr O}}
\newcommand{\scQ}{{\mathscr Q}}
\newcommand{\scS}{{\mathscr S}}
\newcommand{\scT}{{\mathscr T}}
\newcommand{\scV}{{\mathscr V}}
\newcommand{\scP}{{\mathscr P}}
\newcommand{\scR}{{\mathscr R}}
\newcommand{\scY}{{\mathscr Y}}

\newcommand{\adj}{\operatorname{adj}}
\newcommand{\Ad}{\mathrm{Ad}}
\newcommand{\Ann}{\mathrm{Ann}}
\newcommand{\rcris}{\mathrm{cris}}
\newcommand{\ch}{\mathrm{ch}}
\newcommand{\coker}{\mathrm{coker}}
\newcommand{\diag}{\mathrm{diag}}
\newcommand{\Diff}{\mathrm{Diff}}
\newcommand{\Dist}{\mathrm{Dist}}
\newcommand{\rDR}{\mathrm{DR}}
\newcommand{\ev}{\mathrm{ev}}
\newcommand{\Ext}{\mathrm{Ext}}
\newcommand{\cExt}{\mathcal{E}xt}
\newcommand{\fin}{\mathrm{fin}}
\newcommand{\hd}{\mathrm{hd}}
\newcommand{\rht}{\mathrm{ht}}

\newcommand{\im}{\mathrm{im}}
\newcommand{\inc}{\mathrm{inc}}
\newcommand{\ind}{\mathrm{ind}}
\newcommand{\coind}{\mathrm{coind}}
\newcommand{\Lie}{\mathrm{Lie}}
\newcommand{\Max}{\mathrm{Max}}
\newcommand{\mult}{\mathrm{mult}}
\newcommand{\op}{\mathrm{op}}
\newcommand{\ord}{\mathrm{ord}}
\newcommand{\pt}{\mathrm{pt}}
\newcommand{\qt}{\mathrm{qt}}
\newcommand{\rad}{\mathrm{rad}}
\newcommand{\res}{\mathrm{res}}
\newcommand{\rgt}{\mathrm{rgt}}
\newcommand{\rk}{\mathrm{rk}}
\newcommand{\SL}{\mathrm{SL}}
\newcommand{\soc}{\mathrm{soc}}
\newcommand{\Spec}{\mathrm{Spec}}
\newcommand{\St}{\mathrm{St}}
\newcommand{\supp}{\mathrm{supp}}
\newcommand{\Tor}{\mathrm{Tor}}
\newcommand{\wt}{\mathrm{wt}}
\newcommand{\Ab}{\mathbf{Ab}}
\newcommand{\Alg}{\mathbf{Alg}}
\newcommand{\Grp}{\mathbf{Grp}}
\newcommand{\Mod}{\mathbf{Mod}}
\newcommand{\Sch}{\mathbf{Sch}}\newcommand{\bfmod}{{\bf mod}}
\newcommand{\Qc}{\mathbf{Qc}}
\newcommand{\Rng}{\mathbf{Rng}}
\newcommand{\Var}{\mathbf{Var}}
\newcommand{\gromega}{\langle\omega\rangle}
\newcommand{\lbr}{\begin{bmatrix}}
\newcommand{\rbr}{\end{bmatrix}}
\newcommand{\cd}{commutative diagram }
\newcommand{\SpS}{spectral sequence}
\newcommand{\seteq}{\mathbin{:=}}
\newcommand\C{\mathbb C}
\newcommand\hh{{\hat{H}}}
\newcommand\eh{{\hat{E}}}
\newcommand\F{\mathbb F}
\newcommand\fh{{\hat{F}}}
\def\ge{\frg}
\def\AA{{\cal A}}
\def\al{\alpha}
\def\aq{\cA_q(\frn)}
\def\tilaq{\widetilde{\cA_q}(\frn)}
\def\bfk{{\bf k}}
\def\ci(#1,#2){c_{#1}^{(#2)}}
\def\Ci(#1,#2){C_{#1}^{(#2)}}
\def\mpp(#1,#2,#3){#1^{(#2)}_{#3}}
\def\bCi(#1,#2){\ovl C_{#1}^{(#2)}}
\def\ch(#1,#2){c_{#2,#1}^{-h_{#1}}}
\def\cc(#1,#2){c_{#2,#1}}
\def\bfx{{\bf x}}
\def\bup{B^{\geq}}
\def\blow{\ovl B^{\leq}}
\def\bfi{{\mathbf i}}
\def\bfm{{\mathbf m}}
\def\bfh{{\mathbf h}}
\def\bfj{{\mathbf j}}
\def\bq{B_q(\ge)}
\def\bqm{B_q^-(\ge)}
\def\bqz{B_q^0(\ge)}
\def\bqp{B_q^+(\ge)}
\def\beneme{\begin{enumerate}}
\def\beq{\begin{equation}}
\def\beqn{\begin{eqnarray}}
\def\beqnn{\begin{eqnarray*}}
\def\bigsl{{\hbox{\fontD \char'54}}}
\def\bbra#1,#2,#3{\left\{\begin{array}{c}\hspace{-5pt}
#1;#2\\ \hspace{-5pt}#3\end{array}\hspace{-5pt}\right\}}
\def\cd{\cdots}
\def\ci(#1,#2){c_{#1}^{(#2)}}
\def\CC{\mathbb{C}}
\def\colb{\color{blue}}
\def\colr{\color{red}}
\def\ddd{\hbox{\germ D}}
\def\del{\delta}
\def\Del{\Delta}
\def\Delr{\Delta^{(r)}}
\def\Dell{\Delta^{(l)}}
\def\Delb{\Delta^{(b)}}
\def\Deli{\Delta^{(i)}}
\def\Delre{\Delta^{\rm re}}
\def\ei{e_i}
\def\eit{\tilde{e}_i}
\def\Eit{\widetilde{E}_i}
\def\eneme{\end{enumerate}}
\def\ep{\epsilon}
\def\eeq{\end{equation}}
\def\eeqn{\end{eqnarray}}
\def\eeqnn{\end{eqnarray*}}
\def\fit{\tilde{f}_i}
\def\Fit{\widetilde{F}_i}
\def\FF{{\rm F}}
\def\ft{\tilde{f}}
\def\fulp{\Phi_{\rm BK}}
\def\gau#1,#2{\left[\begin{array}{c}\hspace{-5pt}#1\\
\hspace{-5pt}#2\end{array}\hspace{-5pt}\right]}
\def\gl{\hbox{\germ gl}}
\def\hom{{\hbox{Hom}}}
\def\HOM{{\rm H\textsc{om}}}
\def\hfpp{\Phi^{(+)}}
\def\hfpm{\Phi^{(-)}}
\def\ify{\infty}
\def\id{{\rm id}}
\def\io{\iota}
\def\kp{k^{+}}
\def\km{k^{-}}
\def\llra{\relbar\joinrel\relbar\joinrel\relbar\joinrel\rightarrow}
\def\lan{\langle}
\def\lar{\longrightarrow}
\def\max{{\rm max}}
\def\ld{\ldots}
\def\lm{\lambda}
\def\Lm{\Lambda}
\def\mapright#1{\smash{\mathop{\longrightarrow}\limits^{#1}}}
\def\mapLR#1{\smash{\mathop{\longleftrightarrow}\limits^{#1}}}
\def\Mapright#1{\smash{\mathop{\Longrightarrow}\limits^{#1}}}
\def\mm{{\bf{\rm m}}}
\def\nd{\noindent}
\def\nn{\nonumber}
\def\nnn{\hbox{\germ n}}
\def\catob{{\cal O}(B)}
\def\oint{{\cal O}_{\rm int}(\ge)}
\def\ot{\otimes}
\def\op{\oplus}
\def\opi{\ovl\pi_{\lm}}
\def\ovl{\overline}
\def\plm{\Psi^{(\lm)}_{\io}}
\def\pmby{\pmb y}
\def\qq{\qquad}
\def\q{\quad}
\def\qed{\hfill\framebox[2mm]{}}
\def\QQ{\mathbb Q}
\def\qi{q_i}
\def\qii{q_i^{-1}}
\def\ra{\rightarrow}
\def\ran{\rangle}
\def\rgmod{\hbox{$R$-\hbox{gmod}}}
\def\rlm{r_{\lm}}
\def\ssl{\mathfrak sl}
\def\slh{\widehat{\ssl_2}}
\def\syl{\scriptstyle}
\def\ti{t_i}
\def\tii{t_i^{-1}}
\def\til{\tilde}
\def\tilrgmod{\widetilde{R}\hbox{-gmod}}
\def\tm{\times}
\def\tt{\frt}
\def\ua{U_{\AA}}
\def\ue{U_{\vep}}
\def\uq{U_q(\ge)}
\def\ufin{U^{\rm fin}_{\vep}}
\def\ufinp{(U^{\rm fin}_{\vep})^+}
\def\ufinm{(U^{\rm fin}_{\vep})^-}
\def\ufinz{(U^{\rm fin}_{\vep})^0}
\def\uqm{U^-_q(\ge)}
\def\uqp{U^+_q(\ge)}
\def\uqmq{{U^-_q(\ge)}_{\bf Q}}
\def\uqpm{U^{\pm}_q(\ge)}
\def\uqq{U_{\bf Q}^-(\ge)}
\def\uqz{U^-_{\bf Z}(\ge)}
\def\ures{U^{\rm res}_{\AA}}
\def\urese{U^{\rm res}_{\vep}}
\def\uresez{U^{\rm res}_{\vep,\ZZ}}
\def\util{\widetilde\uq}
\def\uup{U^{\geq}}
\def\ulow{U^{\leq}}
\def\TY(#1,#2,#3){#1^{(#2)}_{#3}}
\def\Vbox{V_{\fsquare(0.2cm,{})}}
\def\Lbox{L_{\fsquare(0.2cm,{})}}
\def\Bbox{B_{\fsquare(0.2cm,{})}}
\def\Bfy{B(\ify)}
\def\vep{\varepsilon}
\def\vp{\varphi}
\def\vpi{\varphi^{-1}}
\def\VV{{\cal V}}
\def\xii{\xi^{(i)}}
\def\Xiioi{\Xi_{\io}^{(i)}}
\def\xxi(#1,#2,#3){\displaystyle {}^{#1}\Xi^{(#2)}_{#3}}
\def\WW{{\cal W}}
\def\wtil{\widetilde}
\def\what{\widehat}

\def\Lm{\Lambda}
\def\lm{\lambda}
\def\vep{\varepsilon}
\def\vp{\varphi}
\def\cd{\cdots}
\def\ovl{\overline}
\def\ot{\otimes}
\newpage
\create{center}
{\Large
Decomposition Theorem for Product of Fundamental Crystals in Monomial Realization
of type $C_n$}
		
		\vspace{1ex}
{\large
Manal Alshuqayr\footnote{
Ministry of Education, Salahuddin 12434, Riyadh, Saudi Arabia
}} and  
{\large Toshiki Nakashima\footnote{ Division of Mathematics, 
Sophia University, Kioicho 7-1. Chiyoda-ku, Tokyo 102-8554, Japan,
T.N. is supported by 
Grant-in-Aid for Scientific Research (C) 20K03564, 
Japan Society for the Promotion of Science.}

}
\delete{center}

\thispagestyle{empty}

\begin{abstract}
We consider a product of fundamental crystals of type $C_n$ in monomial realization, 
where the product means a natural product of 
Laurent monomials, not a tensor product.  
Then we shall show that the product still holds a crystal structure and 
describe how it is decomposed into irreducible crystals, which is, in general, 
different from the decomposition rule for the tensor product of the fundamental crystals. 
\end{abstract}

\section {Introduction }

\noindent\indent
For a symmetrizable Kac-Moody algebra $\mathfrak{g}$, let $\Uq$ be the corresponding quantum group. 
It is known that any integrable highest weight $\Uq$-module $V$
 has a crystal base 
 $(L,B,wt,\{\varepsilon_i\}, \{\varphi_i\}, \{{\tilde e}_i\},\{{\tilde f}_i\})_{i\in I}$  
 \cite{M1,M2}, where $I$ is the index set of simple roots of $\mathfrak g$, 
 $B$ is called a crystal of $V$, 
 $\varepsilon_i, \varphi_i$ are integer valued-functions on $B$ and 
 $\tilde e_i$, $\tilde f_i$ are called the {\it Kashiwara operators} on $B$.          
Roughly, we can say that crystal base is a basis for a $\Uq$-module at $q=0$ and 
holds a remarkable property for the tensor product of modules, namely,
if $\Uq$-modules $M_i$ ($i=1,2$) have crystal bases $(L_i,B_i)$ ($i=1,2$), then 
the tensor product $M_1\otimes M_2$ holds 
a crystal base $(L_1\otimes L_2, B_1\otimes B_2)$. 
The explicit actions of the Kashiwara operators on tensor product of crystal bases
are given as follows (see Sect.2 for details): for $b_1\in B_1, \,b_2\in B_2$, we have
\begin{align} \label{tenor rule}
         \nonumber
        \tilde{e}_{i}( b_{1}\otimes \, b_{2} )&=\left\{ {\begin{array}{rcl}
         \tilde{ e}_{i}b_{1}\otimes \, b_{2}&\mbox{ if}&\varphi_{i}( 
         b_{1} )\geq \varepsilon_{i}( b_{2} ),\\ 
         b_{1}\otimes \tilde{e}_{i}b_{2}&\mbox{ if}& \varphi 
        _{i}( b_{1} )<\varepsilon_{i}( b_{2} ), 
            \end{array}} \right.\\
         \nonumber
        \tilde{f}_{i}( b_{1}\otimes \, b_{2} )&=\left\{ {\begin{array}{rcl}
         \tilde{f}_{i}b_{1}\otimes \, b_{2}&\mbox{ if}&\varphi_{i}( 
         b_{1} )>\varepsilon_{i}( b_{2} ),\\ 
         b_{1}\otimes \tilde{f}_{i}b_{2}&\mbox{ if}&\varphi 
       _{i}( b_{1} )\le \varepsilon_{i}( b_{2} ).
         \end{array}} \right.
           \end{align}

Indeed, using these rules, we consider a decomposition of tensor product of crystal bases into irreducible crystal bases (see e.g.,\cite{TN}), which  describes 
       the decomposition of tensor product of corresponding modules, that is, 
for any dominant weight $\lambda,\,\mu$, there exist   dominant weights
       $\lambda_{1},\mathellipsis ,\lambda_{k}$ such that 
\[    
B(\lambda)\otimes B(\mu)\cong B(\lambda_{1}) \oplus \mathellipsis\oplus B(\lambda_{k}),
\]
and then we obtain the decomposition 
\[
V(\lambda)\otimes V(\mu)\cong V(\lambda_{1}) \oplus \mathellipsis\oplus V(\lambda_{k}).
\]

      There are several types of realization for crystal bases, $e.g.$, tableaux realization, path realization, geometric realization, 
polyhedral realization, $etc$. 
Here we will treat the ``monomial realization" which is introduced by
Nakajima \cite{N} and refined by Kashiwara \cite{MK}. Let $\mathcal{M}$ be the set of Laurent monomials in the variables 
$Y_{i}(n)\, (i \in  I,\,n \in \mathbb{Z},$ see Sect.3): 
\[
\mathcal{M}\coloneqq\left\{{\left. \prod\limits_{i\in I,n\in \mathbb{ Z}} {Y_{i}( n )}^{{y}_{i}(n)} \,\right|}\, 
y_{i}(n)\in\mathbb{Z}\, {\hbox{vanish except finitely many }}(i,\, n)\right\}.
\]
Then we can define the crystal structures on this set of Laurent monomials. 
This ${\mathcal M}$ or its subcrystals are 
called a {\it monomial realization} of crystals.
In Sect.3, it is introduced that if a monomial 
$Y$ satisfies the highest weight condition, 
then the irreducible component in $\mathcal{M}$ including
$Y$ is isomorphic to the crystal $B(wt(Y))$ corresponding 
to the module $V(wt(Y))$ \cite{MK} as we mentioned above. In fact, 
by Corollary \ref{M crystal }
if $\mathfrak{g}$ is  semi-simple, then $\mathcal{M}$ is an infinite direct sum of irreducible crystal $B(\lambda)$'s, namely, 
 any irreducible component in $\mathcal{M}$ is isomorphic to the crystal
$B(\lambda)$ for some dominant weight $\lambda$.
Let us denote an irreducible component corresponding to a highest monomial $Y$ 
by $\mathcal{M}(Y)$. The tensor product $\mathcal{M}(Y)\otimes \mathcal{M}(Y')$ 
has a crystal structure. But it is not clear whether the set of products
\[
\mathcal{M}(Y)\cdot\mathcal{M}(Y') \coloneqq\left\{M_{1}\cdot M_{2} \vert\, M_{1}
\in \mathcal{M}(Y),
M_{2}\in \mathcal{M}(Y')\right\},
\]
holds a crystal structure, where $M_{1}\cdot M_{2}$ means a natural product of Laurent  monomials.

Let $\mathcal{M}(Y_{i}(m))$ $(i \in  I,\, m \in \mathbb{ Z})$ 
be one of monomial realizations for the fundamental crystal $B(\Lambda_{i})$
(in \cite{BCD} the explicit forms are given for classical type $\ge$), 
where note that the weight of $Y_i(m)$ is equal to 
the fundamental weight $\Lambda_i$ and 
$\tilde e_j Y_i(m)=0$ for any $j\in I$. 
In this article, for type $C_{n}$ we consider the following problems
(see \cite{AN} for type $A_n$):
\begin{enumerate}
\item[(1)] \label{crystal structure} 
Dose the product
       $\mathcal{M}(Y_{p}(m) )\cdot \mathcal{M}(Y_{q}( 1 ))$
       $(p,q\in I,\,m\in \mathbb{Z})$ hold a crystal structure?
\item[(2)] \label{describe the decomposition}
If the answer for \eqref{crystal structure} is affirmative, 
describe the decomposition of the crystal \\
$\mathcal{M}(Y_{p}(m) )\cdot \mathcal{M}(Y_{q}( 1 ))$ $(m\geq1$).
\end{enumerate} 
Indeed, the answer for \eqref{crystal structure} is  affirmative, which is shown in 
       Proposition \ref{crystal of product}, and more general results are given
 in Corollary \ref{general result}; 
 for any $m_1,m_2, \mathellipsis m_l\in
 \mathbb{ Z}$, $p_1,p_2,\mathellipsis p_l\in I\, (l>0)$,
\[
\mathcal{M}(Y_{p_1}( m_1) ) \cdot  
\mathcal{M}(Y_{p_2}(m_2) ) \cdot  \mathellipsis  \cdot  
\mathcal{M}(Y_{p_l}( m_l) ),
\]
holds a crystal structure. Here note that there might be some monomials $M_{1},M_{2}\in \mathcal{M}(Y)$ and $M'_{1},M'_{2}\in \mathcal{M}
      (Y')$ such that $M_{1}\neq M'_{1}$,  $M_{2}\neq M'_{2}$ and $M_{1}\cdot M_{2}= M'_{1}\cdot M'_{2}$. Therefore, it may happen on the number of monomials:
\[
\#(\mathcal{M}(Y)\cdot\mathcal{M}(Y'))<\#(\mathcal{M}(Y)\otimes\mathcal{M}(Y')),
\]
     and then one can deduce that the decomposition for $\mathcal{M}(Y)\cdot\mathcal{M}(Y')$ may differ from the one for the tensor product
$\mathcal{M}(Y)\otimes\mathcal{M}(Y')$. 
Indeed, the answer for (2) is  given in Theorem \ref{smy.deco.}.
  \begin{align}
                   \nonumber
                  \mathcal{M}(Y_{p}(m) ) \cdot\mathcal{M}(Y_{q}( 1 ))
                   &\cong   B(\Lambda_{p}+\Lambda_{q})\\
\nonumber
&\oplus\bigoplus \limits_{\substack{(a,c)\ne(p,q),(q,p)\\
0\le a\le p,\,\,a+p\le c+q,\,\,(p+q)-(a+c)\in 2\mathbb{ Z}_{>0.}\\a\le c\le n,\,\,a+q\le c+p,\,\,\frac{p+q+c-a}{2}\le n.} }B(\Lambda_{a}+\Lambda_{c})\cdot  \mathcal{L}(m\geq \frac{q-p-a-c+2}{2}+n)\\
              \nonumber
                &\oplus\bigoplus\limits_{\substack{(a,c)\ne(p,q),(q,p)\\
                {0\le a< p<c\le n},\,\,a+p\le c+q,\,\,p+q=a+c,\\
a+q\le c+p,\,\,\frac{p+q+c-a}{2}\le n.}}
B( \Lambda_{a} +\Lambda _{c} )\cdot \mathcal{L}(m\geq q-a+1),
\end{align}
where $\mathcal{L}(P)=1$ if $P$ is true and $\mathcal{L}(P)=0$ otherwise.
By Proposition \ref{Y forms}, forgetting the condition in $\mathcal L$, 
then we get the decomposition rule 
for the tensor product $\mathcal{M}(Y_{p}(m))\otimes\mathcal{M}(Y_{q}(1))$. 
Hence, by this theorem, if $m$ is sufficiently large, we get that 
          $\mathcal{M}(Y_{p}(m))\cdot\mathcal{M}(Y_{q}(1))\cong B(\Lambda_{p})\otimes
          B(\Lambda_{q})$.
However, if not we find that $\mathcal{M}(Y_{p}(m))\cdot\mathcal{M}(Y_{q}(1))$ is not 
necessarily isomorphic to $ B(\Lambda_{p})\otimes B(\Lambda_{q})$.
Indeed, if $m=1$, we know that   $\mathcal{M}(Y_{p}(m))\cdot\mathcal{M}(Y_{q}(1))\cong 
B(\Lambda_{p}+\Lambda_{q})\varsubsetneq B(\Lambda_{p})\otimes B(\Lambda_{q})$.

\noindent\indent
The organization of the article is as follows:
\newline
In Section 2, we will recall the theory of crystal bases, such as, 
the existence theorem, the tensor product theorem and crystal graphs.
In Section 3, Nakajima's monomial realization of crystals will be introduced
(\cite{MK,N}), 
 which is our main ingredient in the article.
 In Section 4, first, the decomposition theorem for a tensor product of 
 the fundamental crystals $B(\Lambda_p)\otimes B(\Lambda_q)$ will be given 
 in Lemma \ref{decomposition}.
The explicit form of  monomial realization for the fundamental crystal 
$B(\Lambda_{i})$ of type $C_{n}$  will be described following \cite{BCD}.
In Section 5,  we state the first main result, which is the answers for  question (1). 
Indeed, we show that the product for fundamental crystals in monomial realization of 
type $C_{n}$ holds a crystal structure and  
we will obtain several important results for 
monomial realization of fundamental crystal  $\mathcal{M}(Y_{p}(m))$.  
In particular, Theorem \ref{induction on n} is crucial to describe the decomposition of product 
$\mathcal{M}(Y_{p}(m))\cdot \mathcal{M}(Y_{q}(1))$ $(m\geq1)$. 
In the last section, we state the second main result, which is the answers for  question  (2). In fact, we obtain the explicit form of  the decomposition for product $\mathcal{M}(Y_{p}(m))\cdot \mathcal{M}(Y_{q}(1))$ $(m\geq1)$, 
which is given in Theorem \ref{forms thm} and Theorem \ref{smy.deco.}. 

\noindent\indent
J.Kamnitzer et al. treated the similar problem in \cite{KJ}, for simply-laced cases. Their
approach is completely differs from ours but the result for type $A_n$ might be the same. 
But their method may not be applied to
multi-laced cases, like as type $C_n$.

{\bf Acknowledgments}: We appreciate M.Nakasuji 
for her helpful comments and suggestions.

\section{Preliminaries}
	\noindent\indent
  Our basic setting is the same as the one in 
  \cite[Sect.2]{AN}. 
  \begin{notation}
  Let $A=(a_{ij})_{i,j\in I}$ be the Cartan matrix for a simple Lie algebra $\mathfrak{g}$ and let $\mathbb{Q}(q)$ be the rational function field in $q$, where $I=\{1,\mathellipsis ,n\}$. 
Let  $t$ be a Cartan subalgebra of $\mathfrak{g}$, 
        $\{\alpha_i\in t ^\ast\}_{i\in I}$ the 
       set of simple roots  and $\{h_i\in t\}_{i\in I}$ 
       the set of simple coroots. 
We take an inner product $(\,\, ,  \,\,)$ on $t^\ast$ such that $(\alpha_i,\alpha_i)\in \mathbb{Z}_{>0}$ and 
      $\langle h_i,  \lambda \rangle =2( \alpha_i, \lambda )/(\alpha_i,\alpha_i)$ for 
       $\lambda_i\in t ^\ast $. Let $\{\Lambda_i\}_{i\in I}$ be the dual basis of $\{h_i\}$ and set
      $P \coloneqq \oplus_i \mathbb{Z}\Lambda_i$ and  $P^\vee \coloneqq \oplus_i \mathbb{Z}h_i$. 
      Then the quantum group  $U_q(\mathfrak{g})$ is the algebra over  
      $\mathbb{Q}(q)$ 
      generated by $\{e_{i},\,f_{i},\,k_{i}^{\pm 1}\}$
      ( $i \in I=\{\,1,\cdots ,n\}\,$) satisfying  the usual
    relations (see e.g., \cite{AN}).
	For simple roots $\{\,\alpha_{i}\vert i=1,2,\mathellipsis ,n\}\,$,
	$a \in \mathbb{Z}$ and $i\in I$, 
	set $q_i\coloneqq q^{(\alpha_{i},\alpha_{i})}$, 
	$[a]_i\coloneqq \frac{q_i^{\,a} - q_i^{-a}}{q_i-q_i^{-1}}$ and 
$f_i^{(k)}\coloneqq
		\frac{f_i^k}
		{[k]_i!}
		\lower2pt\hbox{,}\,
e_i^{(k)}\coloneqq
		\frac{e_i^k}
		{[k]_i!}$, where
$[n]_{i}!\coloneqq [n]_i[n-1]_i\mathellipsis[2]_i[1]_i$, \end{notation}
	\iftrue
 
Let $M$ be a finite dimensional $U_q(\mathfrak{g})$-module and for 
$\lambda\in P$ set $M_\lambda=\{v\in M  \vert 
k_iv=q^{\langle h_i, \lambda\rangle }(\forall  i\in I)\}$, which is called a weight space 
of $M$ of weight $\lambda$. Then we know that  $M=\bigoplus_{\lambda \in P}M_{\lambda } $.

 \begin{definition}
                  Let $M=\bigoplus_{\lambda \in P}M_{\lambda } $ be a 
                 finite dimensional  $\Uq$-module.
                  For each $i\, \in \, I$, every weight vector $u\, \in \, M_{\lambda }\, 
                  (\lambda \in \, \, wt(M))$ is written in the form 
                        \begin{align} \label{form of u}
   u=u_{0}\, +\, f_{i}u_{1}\, + \mathellipsis+\, f_{i}^{(N)}u_{N}, 
                       \end{align}
               where $N \in \mathbb{Z} _{\geq 0 }$ and $ u_{k}\in M_{\lambda +k\alpha 
                  _{i}}\cap\Ker e_{i}$ for any $k=0,1,\mathellipsis ,N$.
                  Here, each $u_{k}$ in the expression is uniquely determined by $u$, and 
                  $u_{k}\ne 0$ only if $\lambda (h_{i})+ k\geq0$. 
                The\itshape{ Kashiwara operators}  $\tilde{e}_{i},\tilde{f}_{i}\in End_{\mathbb{Q}(q)}(M)( i\in I )$ are defined  for $u=\sum\limits_{k=0}^{N}f_{i}^{(k)}u_{k}\in M_{\lambda}$ by                       
                          \begin{align}
                  \tilde{e}_{i}u=\sum\limits_{k=1}^N f_{i}^{(k-1)} \, u_{k},\, \, \, \, 
                  \, \, \tilde{f}_{i}u=\sum\limits_{k=0}^N f_{i}^{(k+1)} \, u_{k}\, .
                         \end{align}
 \end{definition}

Let $\mathcal A$ be the subring of $\mathbb{Q}(q)$ defined by ${\mathcal A}
=\{f(q)/g(q)\,\vert\, f(q),g(q)
\in \mathbb{Q}[q], g(0)\neq 0\}$.

\begin{definition}
Let $M$ be a finite dimensional  $\Uq$-module and $L$ be a 
free ${\mathcal A}$-submodule. $ L $ is called a \itshape{crystal lattice} if 
$M\cong \mathbb{Q}(q)\otimes_{\mathcal A} L$,
$L=\bigoplus_{\lambda \in P}(L \cap M_{\lambda })$ and 
                 $\tilde{e}_{i}L\subset L,\, \tilde{f}_{i}L\subset L$
                   for all $ i\in I .$
               
 \end{definition}
     \begin{definition}\label{crystal basis}
               A pair $(L, B)$ is called a \textbf{\itshape{crystal base}} of a  finite dimensional  $\Uq$-module $M$ if it satisfies:  
          \begin{enumerate}
              \item $L$ is a crystal lattice of $M$,
              \item $B$ is a $\mathbb{Q}$-basis of $L/qL$,
               \item $B=\bigsqcup_{\lambda \in P }B_{\lambda},$
                 where $ B_{\lambda }= B\cap ( L_{\lambda }/q L_{\lambda }),$ 
                 \item $\tilde{e}_{i}\,B\subset B\cup \left\{ 0 \right\},\,
                 \tilde{f}_{i}B\subset B\cup \left\{ 0 \right\}$ for all $ i\in I,$
                 \item  for any $ b,\, b' \in B$ and $ i\in I,$ we have $\tilde{f}_{i}b=b'$
                 if and only if $ b=\tilde{e}_{i}b'$.
                  \end{enumerate}
 \end{definition}

          Let $P_{+}\coloneqq \{\sum\limits_{i=1}^{n}m_{i}\Lambda_{i}\, \vert\, m_{i}\in \mathbb{Z} 
           _{\geq 0 }
          (\forall i\in I)\}$ be the set of dominant weights. For any finite dimensional irreducible 
          $\Uq$-module $L$, there exists a unique dominant weight $\lambda\in P_{+}$ such 
         that $L \cong V(\lambda)$, where $V(\lambda)\coloneqq \Uq / \sum\limits_{i\in I} \Uq e_{i}
         +\sum\limits_{i\in I} \Uq(k_{i}-q_{i}^{\langle h_i,  \lambda \rangle})+\sum\limits_{i\in I} \Uq 
          f_{i}^{1+\langle h_i,  \lambda \rangle}$.
Here let $\pi _{\lambda}\coloneqq \Uq \rightarrow V(\lambda)$ be the natural projection and  set $u_{\lambda}\coloneqq \pi _{\lambda}(1)$, 
 which is the highest weight vector of $ V(\lambda)$.
Let us define for $\lambda\in P_{+}$, 
\[
L(\lambda)\coloneqq \sum\limits_{\substack{ i_{1},
       \mathellipsis ,i_{l}\in I\\ l\geq 0}}A \tilde{f}_{i_1}\mathellipsis\tilde{f}_{i_l}u_{\lambda},\quad
B(\lambda)\coloneqq \left\{\tilde{f}_{i_1}\mathellipsis\tilde{f}_{i_l}u_{\lambda}{\rm mod}\, qL(\lambda)\,\vert\, i_{1},\mathellipsis ,i_{l}\in I, l\geq 0\right\}\setminus\{0\}.
\]
\begin{theorem} [{\cite{M1,M2}}]
       The pair $(L(\lambda),B(\lambda))$ is a crystal base of $V(\lambda)$.
\end{theorem}
\begin{theorem}[{\cite{M1,M2}}]
          Let $M_{j}$ be a  finite dimensional $\Uq$-module and let $( L_{j},
          B_{j} )\, \, $be a crystal basis of $M_{j} \, (j\, =\, 1,\, 2)$. Set 
          $L:={L_{1}\otimes }_{\mathcal A}L_{2}$ and $B:=\, B_{1}\otimes B_{2}:=\{b_1\ot b_2\mid b_1\in B_1,\,b_2\in B_2\}$ .
Then $( L, B )$ is a crystal basis of ${M_{1 }\otimes 
        }_{\mathbb{Q}(q)}M_{2 }$, where the action of Kashiwara operators$\, 
        \tilde{e}_{i}$ and $\tilde{ f}_{i}$ on $B \, (i \in I)$ is given 
        by 
         \begin{align} 
         \nonumber
        \tilde{e}_{i}( b_{1}\otimes \, b_{2} )&=\left\{ {\begin{array}{rcl}
         \tilde{ e}_{i}b_{1}\otimes \, b_{2}&\mbox{ if}&\varphi_{i}( 
         b_{1} )\geq \varepsilon_{i}( b_{2} ),\\ 
         b_{1}\otimes \tilde{e}_{i}b_{2}&\mbox{ if}& \varphi 
        _{i}( b_{1} )<\varepsilon_{i}( b_{2} ), 
            \end{array}} \right.\\
         \nonumber
        \tilde{f}_{i}( b_{1}\otimes \, b_{2} )&=\left\{ {\begin{array}{rcl}
         \tilde{f}_{i}b_{1}\otimes \, b_{2}&\mbox{ if}&\varphi_{i}( 
         b_{1} )>\varepsilon_{i}( b_{2} ),\\ 
         b_{1}\otimes \tilde{f}_{i}b_{2}&\mbox{ if}&\varphi 
       _{i}( b_{1} )\le \varepsilon_{i}( b_{2} ), 
         \end{array}} \right.
           \end{align}
      and we have
         \begin{align}\label{epsilon def}
           \nonumber
         wt( b_{1}\otimes \, b_{2} )&=\, wt( b_{1} )+\, wt(\, 
         b_{2}),\\
        \varepsilon_{i}( b_{1}\otimes \, b_{2} )&=\max (\varepsilon 
       _{i}( b_{1} ){,\varepsilon }_{i}( b_{2} 
        )-\langle\hs{1pt} \ h_{i}, wt( b_{1} )\rangle),\\
           \nonumber
        \varphi_{i}( b_{1}\otimes \, b_{2} )&=\max (\varphi_{i}( 
        b_{2} ),\varphi_{i}( b_{1} 
        )+\langle\hs{1pt} \ h_{i}, wt( b_{2} )\rangle).
              \end{align} 
Here, 
we understand $b_{1}\otimes 0 =0\otimes  b_{2} =0$.
\end{theorem}
\begin{definition}[{ \cite{M}}]
Let $I$ be a finite index set and let $A=(a_{ij})_{i,j\in I}$ be 
a Cartan matrix of $\mathfrak g$ and $P$ a corresponding weight lattice.   
\itshape{Crystal} associated with the Cartan matrix $A$  is 
a set $B$ together with the maps 
$wt:B\to P$, $  \tilde{e}_{i},\,  \tilde{f}_{i}:B\to B\cup \{0\}$, 
and $\varepsilon _{i},\,\varphi_{i}:B\to\mathbb{Z}\cup\{-\infty\}(i\in I)$ 
satisfying the following: 
\begin{align*}
& \hbox{$\varphi_{i}(b)=\varepsilon_{i}(b)+\langle h_{i},\,wt(b)\rangle$ for all $i\in I$},\\
&  \hbox{$wt(  \tilde{e}_{i}b)=wt\, b+\alpha_{i}$ if $ \tilde{e}_{i}b\in B$, }
\q \hbox{$wt(  \tilde{f}_{i}b)=wt\, b-\alpha_{i}$ if $ \tilde{f}_{i}b\in B$},\\
&  \hbox{$\varepsilon_{i}(  \tilde{e}_{i}b)=\varepsilon_{i} ( b)-1,\,\varphi_{i}(  \tilde{e}_{i}b)=\varphi_{i}( b)+1$ if $\tilde{e}_{i}b\in B$,}\\
&  \hbox{$\varepsilon_{i}(  \tilde{f}_{i}b)=\varepsilon_{i} ( b)+1,\,\varphi_{i}(  \tilde{f}_{i}b)=\varphi_{i}( b)-1$ if $\tilde{f}_{i}b\in B$,}\\
&  \hbox{$ \tilde{f}_{i}b=b'$ if and only if $b=\tilde{e}_{i}b'$ for $b,\,b'\in B$, $i\in I$,}\\
&  \hbox{if $\varphi_{i}(b)=-\infty$ for $b\in B$, then $  \tilde{e}_{i}b=  \tilde{f}_{i}b=0$.}
\end{align*}
 \end{definition}
\begin{definition}
             Let $B$ be a crystal.
             Take $B$ as the set of vertices and define the $I$-colored 
             arrows on $B$ by: 
$b\buildrel i \over \longrightarrow b'\,$ if and only if
              $\tilde{f}_{i}b =b^{'}(i\in I).$
Then $B$ is given an $I$-colored oriented graph structure, 
which is called the {\it crystal graph} of $B$.
 \end{definition}

\section{Nakajima's monomial realization }
In this section, we recall the crystal structure on the set of monomials discovered by H. Nakajima \cite{N}. Our exposition follows that of M.Kashiwara \cite{MK}.
Let $\mathcal{M}=\mathcal{M}( \mathfrak{g})$ be the set of Laurent monomials in the variables $Y_{i}(m)\, (i \in  I, m \in \mathbb{ Z})$, called $Y$-variable: 
\[
\mathcal{M}( \mathfrak{g})\coloneqq\left\{{\left. \prod\limits_{i\in I,m\in \mathbb{ Z}} {Y_{i}( m )}^{{y}_{i}(m)} \,\right|}\, 
y_{i}(m)\in\mathbb{Z}\, {\hbox{vanish except finitely many }}(i,\, m)\right\}.
\]
We shall define a crystal structure of $\mathcal{M}$. 
Let $c={(c_{ij})}_{i\ne j\in I}$ be a set of integers such that 
$c_{ij}+\, c_{ji} = 1$ for $i\ne j$.
 We set 
\[
   A_{i}(m)\, =\, Y_{i}(m)Y_{i}(m\, +\, 1)\, \prod\limits_{j\ne i} {Y_{j}(m\, 
   +c_{ji})}^{\langle h_{j},\alpha_{i}\rangle\,}.
\]
For a monomial  $M\, =\, \prod\limits_{i\,\in I,m\,\in\,\mathbb{ Z}} {Y_{i}
( m )}^{y_{i}( m )}\in\mathcal{M} $, we define
            \begin{align}\label{varepsilon}
              \nonumber
            wt( M )&=\, \sum\nolimits_i ( \sum\nolimits_m {y_{i}(m)} )\Lambda 
             _{i}\, \, ,\\
             \varphi_{i}( M )&= \max\,\bigg\{\, \sum_{k\le m} {y_{i}(k)} ;\, 
              m\in \,\mathbb{ Z}\bigg\},\\
              \nonumber
             \varepsilon_{i}( M )&=\max\,\bigg\{\,-\sum_{k> m} {y_{i}(k)} ;\, 
              m\in \, \mathbb{ Z}\bigg\}.
                  \end{align} 
We also define 
\begin{align}
\nonumber
\tilde{ f}_{i}( M )&=\left\{ {\begin{array}{rcl}
                  0\quad\quad\quad&\mbox{ If}& \varphi_{i}( M )=0,\\ 
                   A_{i}{(n_{f})}^{-1}M&\mbox{  If}&\varphi 
                  _{i}( M )>0,
                          \end{array}} \right.
                         \qq 
                         \tilde{ e}_{i}( M )&=\left\{ {\begin{array}{rcl}
                      0\quad\quad&\mbox{ If}& \varepsilon_{i}( M )=0, \\ 
                 A_{i}(n_{e})M&\mbox{  If}&\varepsilon 
                 _{i}( M )>0, 
                               \end{array}} \right.
                           \end{align}
where 
\begin{align}\label{nf def}
n_{f}&=\min\left\{m\mid \varphi_{i}( M )\mathrm{=}\sum_{k\le m} {y_{i}(k)}\right\}=\min\left\{m\mid \varepsilon_{i}( M )\mathrm{=-}\sum_{k>m} {y_{i}(k)}\right\} \, ,\\
 n_{e}&=\max\left\{ m\mid \varphi_{i}( M 
                     )\mathrm{=}\sum_{k\le m} {y_{i}( k )} \right\}    
=\max\left\{m\mid \varepsilon_{i}( M )\mathrm{=-}\sum_{k>n} {y_{i}(k)} \right\}\,.
\end{align}

                           Note that $y_{i}(n_{f}) > 0,\,y_{i}(n_{f}+1)\le 0$ and
                      $  y_{i}(n_{e}+1)< 0,\,y_{i}(n_{e})\geq 0$. 
 
                    Let us denote by  $\mathcal{M}_{c}$ the crystal $\mathcal{M}$ associated with c.

   We shall call a crystal 
an $I$-crystal when we emphasize the index set $I$ of simple roots. 
For an $I$-crystal $B\, $and a subset $J$ of $ I$,  let us denote by $\Psi_{J} (B)$ the
            $J$-crystal $B$ with the induced crystal structure from $B$. 
\begin{definition}
            A crystal $B$ is called \itshape{ semi-normal} if for each $i\in I$  the 
           $\{i\}\,$-crystal $\Psi 
           _{\{i\}\,}(B)$  is a crystal associated with a finite dimensional  module. This is 
           equivalent to say that $\varepsilon_{i}(b)=max \{\,n \in \, N\, ;\, 
           \tilde{e}_{i}^{n}b \ne 0\}$ and $\varphi_{i}(b)=max\{\,n \in N\, 
           ;\, \tilde{f}_{i}^{n}b\ne 0\}$ for any $b\, \in \, B$ and $i\, \in I$.
         \end{definition}

 \begin{definition}
              A crystal $ B$ is called \itshape{ normal} if for any subset $J$  of  $I$ such that 
              $\{\,\alpha_{i}\, ;\, i\in J\}\,$ is a set of simple roots for finite-dimensional 
               semisimple Lie algebra $\mathfrak{g_{\mathrm{J}}}$, $\Psi_{J}(B)$ is the crystal associated with  a finite dimensional  $U_q(\mathfrak{g_{\mathrm{J}}})$-module, where $U_q(\mathfrak {g_{\mathrm{J}}})$ is  the associated quantum group with $\{\,\alpha_{i}\, ;\, i\, \in  \, J\}\,$.
 \end{definition}
\begin{proposition}[\cite{MK}]
 $\mathcal{M}_{c}$ is a semi-normal crystal. 
\end{proposition}

 For a family of integers ${(m_{i})}_{i\in I}$, let us set $c'\, =\, 
{({c'}_{ij})}_{i\ne j\in I}$, by ${c'}_{ij}=\, c_{ij}+m_{i}\, -m_{j}$. Then 
the map $Y_{i}( m )\mapsto \, Y_{i}( m+m_{i} )$ 
gives a crystal isomorphism  $\mathcal{M}_{c}\buildrel \sim\over\longrightarrow 
 \mathcal{M}_{c'}$.
Hence if the Dynkin diagram has no loop, 
the isomorphism class of the crystal $\mathcal{M}_{c}$ does not 
depend on the choice of $c$.

\begin{theorem}[{\cite{MK}}]
For a highest weight vector $M \in \mathcal{M}_{c}$, the 
connected component of  $\mathcal{M}_{c}$  containing $M$ is isomorphic to
$B(wt(M))$. 
\end{theorem}

\begin{theorem}[{\cite{MK}}] \label{normality} 
$\mathcal{M}_{c}$  is a normal crystal.
\end{theorem}

\begin{corollary}\label{M crystal }
\begin{enumerate}
\item For each $i \in  I$, $\mathcal{M}$  is isomorphic to 
a crystal graph of a $U_{q}(\mathfrak{sl}_2) $-module.
\item Let $M$ be a monomial with weight $\lambda $, 
such that $\tilde{ e}_{i}M= 0$ for all $ i \in  I$, and let $\mathcal{M}(\lambda )\, $be the connected component of $\mathcal{M}$ containing $M$. 
Then there exists a crystal isomorphism
\[
\mathcal{M}( \lambda ) \buildrel\sim\over\longrightarrow\,
B( \lambda )\qquad\quad
(M\longmapsto u_{\lambda }).
\]
\item
 If $\mathfrak{g}$ is a semi-simple Lie algebra, then $\mathcal{M}_{c}$ is an infinite direct sum of irreducible crystal $B(\lambda)$'s, namely, 
 any connected component in $\mathcal{M}_{c}$ is isomorphic to the crystal
$B(\lambda)$ for some $\lambda\in P_+$.
\end{enumerate}
\end{corollary}

\begin{remark}
Note that the condition $\tilde{ e}_{i}M= 0$
for any $ i \in  I$, is called {\it the highest weight} condition, which is equivalent to the condition $\varepsilon_{i}( M )=0$   
for any $ i \in  I$.
\end{remark}


\newcommand\cv{i_N}
\newcommand\er{i_1}
\newcommand\zx{\bar{1}}
\newcommand\sd{\bar{2}}
\newcommand\qw{\bar{n}}
\newcommand\tmp{\bar{i}}

\section{Monomial Realization of the Fundamental Crystals of Type $C_{n}$}
\subsection{Tableaux Realization of Fundamental Crystals of Type $C_n$}
   
\noindent\indent
\normalfont{In this subsection, we recall several results on the crystals for $U_{q}(C_{n})$-modules following \cite{TN}.\cite{MK-N}. }

        Let $I\, =\, 
                       \{1,\,\mathellipsis ,\, n\}\,$ and let $A\, =\, {(a_{ij})}_{i,j\in I}$ be the 
                        Cartan matrix of type $C_{n}$. Here, the entries of $A$ are given by 
            \[
                         a_{ij}=\left\{ {\begin{array}{rcl}
                         2&\mbox{if}& i=j, \\ 
                        -1&\mbox{if}&\left| i-j \right|=1\,\, and \,\, ( i, j )\ne(n-1, n),\\ 
                        -2&\mbox{if}&( i, j )=(n-1, n), \\ 
                        0&  &\mbox{otherwise}.
                                                  \end{array}} \right.
               \]
Let $(\epsilon_1,\cdots,\epsilon_n)$ be the orthonormal base of the dual of the Cartan subalgebra $t^*$ of $C_n$ such that $\alpha_{i}=\epsilon_{i}-\epsilon_{i+1}\,(1\le i<n)$ and $\alpha_{n}=2\epsilon_n$ are the simple roots. Hence, $\alpha_{n}$ is the long root and $\alpha_{1},\cdots,\alpha_{n-1}$ are short roots. 
$\{\Lambda_i=\epsilon_1+\cdots+\epsilon_i\,|\,1\le i\le n\}$ is the set of 
fundamental weights, which is the dual basis of the simple co-roots $\{h_i\}_{1\le i\le n}$, namely, $\lan h_i,\Lm_j\ran=\del_{i,j}$.

The crystal graph $B({\Lambda_1})$ of the vector representation $V(\Lambda_1)$ is described as follows. It is labeled by $\{\Yboxdim{16pt}\young(i), \Yboxdim{16pt}\young(\tmp);\,1\le i\le n\}$ where $\Yboxdim{16pt}\young(i)$ has weight $\epsilon_{i}$ and $\Yboxdim{16pt}\young(\tmp)$ has weight $-\epsilon_{i}$. Its crystal graph is
\begin{align}\label{c.graph Cn}
\young(1)\buildrel{1}\over\longrightarrow\young(2)\buildrel{2}\over\longrightarrow\cdots\buildrel{n-1}\over\longrightarrow\young(n)\buildrel{n}\over\longrightarrow\young(\qw)\buildrel{n-1}\over\longrightarrow\cdots\buildrel{2}\over\longrightarrow\Yboxdim{16pt}\young(\sd)\buildrel{1}\over\longrightarrow\young(\zx).
\end{align} 
The connected component of the crystal graph $B(V_{\square})^{\otimes N}$ $(1\le N\le n)$ containing $u_{\Lambda_{N}}=\Yboxdim{16pt}\young(1)\otimes\young(2)\otimes\cdots\otimes\young(N)$ is isomorphic to $B(\Lambda_{N})$.

We give the linear order on $\mathcal{ I}:=\{i, \bar{i};\,\,1\le i\le n\}$ by
\begin{align}\label{order of i}
1<2<\cdots<n<\bar{n}<\cdots<\bar{2}<\bar{1}. 
\end{align}
Here for $i\in \mathcal{ I}$, if $i\in\{1,\cdots,n\}$ (resp.$\{\bar{n},\cdots,\bar{1}\}$) we call $i$ positive (resp. negative).
By setting  $[i_{1}, i_{2},\cdots, i_{N}]=\Yboxdim{16pt}\young(\er)\otimes \cdots\otimes\Yboxdim{16pt}\young(\cv)$, we have
\begin{align}\label{tableaux con.}
B(\Lambda_{N})=\left\{[i_{1}, i_{2},\cdots, i_{N}]\in B(V_{\square})^{\otimes N} \left| {\begin{array}{l}
                (1) 1\le i_{1}<\cdots <i_{N}\le \bar{1},\\ 
                (2)\mbox{ if} \,\,\,i_k=p \,\,\,\mbox{and} \,\,\,i_{\ell}=\bar{p},\\
                       \mbox{then} \,\,\,k+(N-\ell+1)\le p.
                                 \end{array}} \right. \right\}.
\end{align}

\begin{lemma}\label{decomposition}
            We have the following decomposition for the tensor product
            of fundamental crystals $B(\Lambda_{p})$ and $B(\Lambda_{q})$ of type $C_n$:
           \begin{center}
            $B( \Lambda_{p} )\otimes  B( \Lambda_{q} )\cong \bigoplus\limits_{ (\divideontimes)}B( \Lambda_{a}  +\Lambda _{c} )$,
          \end{center}
where
\[
( \divideontimes)\left\{\begin{array}{l}
0\le a\le c\le n,\\
a+q\le p+c,\\
a+p\le q+c,\\
(p+q)-(a+c)\in 2\mathbb{ Z}_{\geq0},\\
a\le p,\,\,\frac{p+q+c-a}{2}\le n.
  \end{array}\right  .
\]
Note that we suppose that $\Lambda_{0}$ means $0$, thus, $\Lm_0+\Lm_i=\Lm_i$.
          \end{lemma}

\begin{proof}
Let $u\otimes v\in B( \Lambda_{p} )\otimes  B( \Lambda_{q} )$ be a highest weight vector. We know that  $u$ is the highest weight vector \cite{TN} and then we have  $u=[1,2,\cdots,p]$. Set $v=[j_{1}, j_{2},\cdots, j_{q}]\in B(\Lambda_{q})$. Here, 
we have the following three cases for $v$ \cite[Lemma 4.12, Theorem 6.3.1]{TN} :
\[
(1) \,\,j_1=1\quad\quad (2)\,\,j_1=p+1\quad\quad (3)\,\,j_1=\bar{p}.
\]
\begin{itemize}
\item [(1)]   If $j_1=1$, we find that $v$ should be in the form 
\[
v=[1, \ldots, a,p+1,\ldots,b,\bar{b},\ldots,\overline{c+1},\bar{a},\ldots,\bar{d}]. 
\]
But by (2) in (\ref{tableaux con.}), $\bar{a},\ldots,\bar{d}$ cannot exist. Then we can write $v$ as  
\begin{align}\label{v form}
v=[1, \ldots, a,p+1,\ldots,b,\bar{b},\ldots,\overline{c+1}], 
\end{align}
and we obtain the following conditions on $a,p,b$ and $c$
\begin{equation}
\begin{array}{c}
1\le a\le c\le b\le n,\quad a\le p\le b,\quad q\le b, 
\quad a+(b-p)+(b-c)=q,\\
(p+q)-(a+c)\in 2\mathbb{ Z}_{\geq0},
\end{array}
\label{cond-fun}
\end{equation}
where the condition $q\le b$ is derived from (2) in (\ref{tableaux con.}). 
Note that other conditions from (2) in (\ref{tableaux con.}) 
$e.g.$, $q-2\le b-1$, $q-4\le b-2,\ldots$
are satisfied a priori from the condition $q\le b$. 
The condition $(p+q)-(a+c)\in 2\mathbb{ Z}_{\geq0}$ is derived 
from $a+c\le p+q$ and $2b=p+q+c-a\in 2{\mathbb Z}_{\geq 0}$, which is obtained from 
$q=length(v)=a+(b-p)+(b-c)$. We also get $p\leq b\Leftrightarrow p+a\leq q+c$, 
$q\leq b\Leftrightarrow a+q\leq p+c$ and $b\leq n\Leftrightarrow \frac{p+q-a+c}{2}\leq n$.
Thus,  for $v$  with the condition \eqref{cond-fun},  $u\otimes v$ 
is a highest weight vector with the weight $\Lambda_{a}+\Lambda _{c}$. 
\item [(2)] If $j_1=p+1$, then we can write $v$ as $v=[p+1,\ldots,b,\bar{b},\ldots,\overline{c+1}]$. By (2) in (\ref{tableaux con.}) we can obtain the following conditions on $p,b$ and $c$
\[
p< b, \quad c\le b\le n, \quad q\le b, \quad (b-p)+(b-c)=q,
\]
for such $v$, $u\otimes v$ is a highest weight vector with the weigh $\Lambda _{c}$. Indeed, by the condition $c\leq b$ and $2b=p+q+c$, we have $p+q-c\in2\bbZ_{\geq 0}$. We also get 
\[
p< q+c, \quad q\le p+c, 
\]
by $p\leq b$ and $q\leq b$. 
Here we know that these conditions are the ones in (1) with $a=0$.
\item [(3)]  If $j_1=\bar{p}$, then we can write $v$ as $v=[\bar{p},\cdots,\overline{c+1}]$. 
Hence, we obtain the following conditions on $p$ and $c$
\[
c\le p,\quad p-c=q,
\]
for such $v$. Then if $p\geq q$, we have $v=[\bar{p},\overline{p-1},\cdots, \overline{p-q+1}]$ and that $u\otimes v$ is a highest weight vector with the weight $\Lambda _{c}$. Here we know that these conditions are the special case of (1) by setting $a=0$ and $b=p$. Now, we obtain the condition $(\divideontimes)$ in Lemma
\ref{decomposition} 
\end{itemize}
\end{proof}

\subsection{Monomial Realization of the Fundamental Crystals of Type $C_{n}$}
                 
The result in this subsection follows \cite{BCD}.
\noindent\indent
We take the set $C\, ={(c_{ij})}_{i,j\in I,\, i\ne j}$ to be 
\begin{center}
$c_{ij}=\left\{ {\begin{array}{rcl}
0& \mbox{if}& i>j , \\ 
1 &\mbox{if}& i<j,
\end{array}} \right.$
\end{center}
and set ${Y_{0}(m)}^{\pm 1}={Y_{n+1}(m)}^{\pm 1}
=1$ for all $ m\in  \mathbb{ Z}$. Then for $i\ \in  I$
and $m\in  \mathbb{ Z}$,  we have
\[
A_{i}\, (m)\, =\begin{cases}
Y_{i}(m)Y_{i}(m\, +\, 1){Y_{i-1}(m\, +\,1)}^{-1}{Y_{i+1}(m)}^{-1}&
i\ne n,\\
Y_{n}(m)Y_{n}(m\, +\, 1){Y_{n-1}(m\, +\, 1)}^{-2}&i=n.
\end{cases}
   \]

  We focus on the fundamental crystal  $B( \Lambda_k )\cong\mathcal{M}( Y_{k}( m ) )$ ($ {m\in  \mathbb{ Z}}$).

\begin{definition}\label{XXX}
For $ i \in I$  and  ${m\in \mathbb{ Z}}$, we introduce new variables
\[
\begin{array}{rcl} X_{i}(m)=& \frac{Y_{i}(m)}{Y_{i-1}(m +1)} &\mbox{if}\,\,\,\,\,  i\in\{\,1,\ldots,n\},\mbox{i.e,\,\,} i\,\,\mbox{is positive},\\ 
X_{\bar{i}}(m)=&\frac{{Y_{i-1}(m+n-i+1)}}{Y_{i}(m+n-i+1)} &\mbox{if}\,\,\,\,\,   \bar{i}\in\{\, \bar{n},\ldots,\bar{1}\},\mbox{i.e,\,\,} \bar{i}\,\,\mbox{is negative}, \end{array}
\]
which will be called $X$-variable. Note that we set
$X_{n+1}(m)=\frac {1}{Y_{n}(m+1)}$ and $X_{\overline{n+1}}(m)={Y_{n}(m)}$ for any $m\in\bbZ$. 
\end{definition}     
We have the following explicit description of the monomial realization 
${\mathcal M}(Y_k(N))$ of the fundamental crystal $B(\Lm_k)$ ($k\in I$):
\begin{proposition}[\cite{BCD}]\label{general form by X}
For $k\in I$  and $N\in{\mathbb Z}_{>0}$ we have 
\begin{align}\label{g.f}
\mathcal{M}(Y_{k}(N)) =\left\{X_{i_{1}}( k+N-1 )X_{i_{2}}( k+N-2 )
\cdots X_{i_{k}}(N)\vert i_j\in \mathcal{ I}, 
1\le i_{1} <i_{2}<\cdots < i_{k} \le \bar{1}\}\right..
\end{align}
\end{proposition}

Here note that  any monomial in $X$-variables
\begin{align}
\nonumber
M=&\prod\limits_{j=1}^L X_{b_{j}}
(  m_{j} )\prod\limits_{j=L+1}^k X_{\bar{c_{j}}}(m_{j})
\end{align}
can be written by using $Y$-variables as
\[ 
M\, = \prod\limits_{j=1}^L {Y_{a_{j}}( m_{j-1} )^{-1}Y_{b_{j}}( m_{j})} 
\prod\limits_{j=L+1}^k {Y_{c_{j}}( m_{j}+n-c_{j}+1 )^{-1}Y_{d_{j}}( m_{j}+n-d_{j} )}.
\]

For example, we have 
$Y_{k}( N)=X_{1}(  k+N-1  )X_{2}(k+N-2  )
\cdots X_{k}( N)$.
Now, it is straightforward to verify that we have 
another characterization  of the crystal $\mathcal{ M}(Y_{k}(1))$
in Y variables.
The crystal 
$\mathcal{ M}(Y_{k}(N))$ is obtained  by shifting the indices $m_j$ to $m_{j}+N-1$.    

The following result is immediate from the previous result.
\begin{proposition}\label{Y forms}
  For $k\in I$, $Y_{k}(1)$ is the monomial of weight  $\Lambda_{k}$ such that $\tilde{ e}_{i}(Y_{k}(1))=0$  for all $ i\in I$. Then the connected component $\mathcal{M}(Y_{k}(1))$ of $\mathcal{M}$  containing $Y_{k}(1)$ is characterized as 
    \begin{align}\label{exp.of Y}
                 \hspace{-5pt}\mathcal{M} ( Y_{k}(1) )=
                 \left\{ \left.\prod\limits_{j=1}^L {Y_{a_{j}}( m_{j-1} )^{-1}Y_{b_{j}}( m_{j})} \prod\limits_{j=L+1}^k {Y_{c_{j}}( m_{j}+n-c_{j}+1 )^{-1}Y_{d_{j}}( m_{j}+n-d_{j} )}\,
                 \right\vert\,(\bigstar)\right\},
                 \end{align}
where 
\[
(\bigstar)\left\{\begin{array}{l}
  0\le L\le k,\\
  0< b_{1}<b_{2}<\cdots <b_{L}\le n,\\ 
0<c_{k}<c_{k-1}<\cdots<c_{L+1}\le n,\\
 a_{j}+1=b_{j}, c_{j}=d_{j}+1,\, \, \, \forall j\le k,\\
m_{j}=k-j+1,\,\,\,0\le j\le k.
\end{array}\right  .
\]                  
\end{proposition}
\begin{remark}
Note that 
neither the expression using $X_{i}(m)$ variables in (\ref{g.f}), nor the expression in  (\ref{exp.of Y}) is unique for certain monomials since 
one has the following relations of $X$-variables (\cite{BCD}):
\begin{equation}
X_i(p)X_{\bar i}(p-n+i)=\frac{Y_i(p)}{Y_i(p+1)}=X_{i+1}(p)X_{\overline{i+1}}(p-n+i)\quad
(1\leq i\leq n-1).\label{equality}
\end{equation}

 For example,  applying \eqref{equality} by (\ref{exp.of Y}) in Proposition \ref{Y forms} and (\ref{g.f}) in Proposition \ref{general form by X} we find in type  $C_4$, 
\begin{enumerate}
\item  for $ \mathcal{M}(Y_{3}(1))$, we have
\[
X_{1}(3)X_{3}(2) X_{\bar{3}}(1)=Y_{3}(3)^{-1}Y_{3}(2)Y_{1}(3)=X_{1}(3)X_{4}(2) X_{\bar{4}}(1)\in \mathcal{M}(Y_{3}(1)).
\]
Indeed, for $(a_{1},b_{1},a_{2},b_{2},d_3,c_3)=(0,1,2,3,2,3)$ in (\ref{exp.of Y}) we get the monomial $Y_{3}(3)^{-1}Y_{3}(2)Y_{1}(3)$, and 
for $(a_{1},b_{1},a_{2},b_{2},d_3,c_3)=(0,1,3,4,3,4)$ in (\ref{exp.of Y}) also we get the monomial $Y_{3}(3)^{-1}Y_{3}(2)Y_{1}(3)$.
\item  for $ \mathcal{M}(Y_{4}(1))$, we have
$X_{1}(4)X_{2}(3) X_{3}(2)X_{\bar{1}}(1)=Y_{3}(2)Y_{1}(5)^{-1}=X_{2}(4)X_{3}(3) X_{4}(2)X_{\bar{4}}(1)$ and 
$X_{3}(4)X_{\bar{3}}(3) X_{\bar{2}}(2)X_{\bar{1}}(1)=Y_{3}(5)^{-1}Y_{3}(4)Y_{2}(5)^{-1}
=X_{4}(4)X_{\bar{4}}(3) X_{\bar{2}}(2)X_{\bar{1}}(1).$
\end{enumerate}
\end{remark}


\section{Product of Fundamental Crystals in 
Monomial Realization of type $C_n$}
                   
\subsection{The Product
$\mathcal{M}(Y_{p}(m) )\cdot \mathcal{M}(Y_{q}( 1 ))$ }

\begin{proposition}\label{crystal of product}
Let $M_{1}\in\mathcal{M}(Y_{p}( m ))$, 
$M_{2}\in\mathcal{M}(Y_{q}( 1) )$ with $m\geq 1$, and  $p,q \in\{1,\ldots ,n\}\,$.Then for any $i\in I$, we have
\[
\tilde{e}_{i}\, ( M_{1}\cdot{M}_{2} )\in  \mathcal{M}(Y_{p}( m ))
\cdot \mathcal{M}(Y_{q}( 1 ))\cup \, \{0\},
\]
\[
\tilde{f}_{i}\, ( M_{1}\cdot{M}_{2} )\in \mathcal{M}(Y_{p}( m )) 
\cdot \mathcal{M}(Y_{q}( 1 ))\cup \, \left\{ 0 \right\}.
\]
And then we know that the product $\mathcal{M}(Y_{p}( m ))\cdot \mathcal{M}(Y_{q}( 1) )$ holds the crystal structure.
           \end{proposition}

\begin{proof}
By the explicit description in Proposition \ref{general form by X} , we can write 
  \begin{align}
                \nonumber
                     M_{1}&=X_{i_{1}}( p+m-1 )X_{i_{2}}( p+m-2 
                     )\cdots X_{i_{p}}(m),\\
              \nonumber
                    M_{2}&=X_{j_{1}}( q )X_{j_{2}}( q-1 )\cdots
                    X_{j_{q}}( 1 ).
            \end{align}
First, let us consider the action of $\tilde{e}_{i}$.  The factors $Y_{i}(l)^{-1}(l\in\mathbb{Z})$ can appear at most twice
in the factors $X_{i+1}$, $X_{\bar{i}}$ of $M_{1}$ and $M_{2}$ respectively since if a  
monomial $M$ has no $Y_i(l)^{-1}$ then $\eit M=0$.
Let us see the following two cases:
\begin{enumerate}
\item  If $Y_{i}(l)^{-1}$ appear at most  once in each $M_1$ and $M_2$,  we obtain 
     \begin{center}
                  $\tilde{e}_{i}\, ( M_{1}\cdot{M}_{2} )=
                       ( \tilde{e}_{i}M_{1} )\cdot M_{2}$ or $ M_{1}\cdot ( \tilde{e}_{i}M_{2} )$ or $0$ .
                \end{center}
                  Indeed,  if both  $Y_{i}( l_{1} )^{-1}$ appears in $M_{1}$ and $Y_{i}( l_{2} )^{-1}$ appears in  $M_{2}$, by the definition  of $ \tilde{e}_{i}$ we have
             \begin{align}
                  \nonumber
                    \tilde{e}_{i}\, ( M_{1}\cdot{M}_{2} )& =A_{i}( l_{1} ) 
                     ( M_{1}\cdot{M}_{2} )\,\,\,\mathrm{or}\,\,\,A_{i}( l_{2} )  ( M_{1}\cdot{M}_{2} )\,\,\,\mathrm{or}\,\,\, 0\\
                  \nonumber
                    &= ( \tilde{e}_{i}M_{1} )\cdot M_{2}
                     \,\,\,\mathrm{or}\,\,\,  M_{1}\cdot ( \tilde{e}_{i}M_{2} )\,\,\,\mathrm{or}\,\,\, 0,
                 \end{align}
where $(\tilde{e}_{i}M_{1} )\in \mathcal{M}(Y_{p}( m ))
\cup\{0\}$
 and $(\tilde{e}_{i}M_{2} )\in\mathcal{M}(Y_{q}(1) )
 \cup\{0\}$ 
 by  Proposition \ref{Y forms}, 
 then $  \tilde{e}_{i}\, ( M_{1}\cdot{M}_{2} )
 \in  \mathcal{M}(Y_{p}( m ))  
 \cdot \mathcal{M}(Y_{q}( 1 ))\cup \, \{0\}.$
If one of them appears in $M_{1}$ or $M_{2}$,  we also have 
$\tilde{e}_{i}\, ( M_{1}\cdot{M}_{2} )=
( \tilde{e}_{i}M_{1} )\cdot M_{2}$
or $  M_{1}\cdot ( \tilde{e}_{i}M_{2} )\,\,\mathrm{or}\,\, 0$ .
                
Then we obtain $\tilde{e}_{i}\, ( M_{1}\cdot{M}_{2} )\in  \mathcal{M}(Y_{p}( m ))  \cdot \mathcal{M}(Y_{q}( 1 ))\cup \, \{0\}.$
\item If $Y_{i}(l)^{-1}$ appears twice in $M_1$, then we can write $M_1$ as
\begin{align}
\nonumber
M_{1}=X_{i_1}( p+m-1 )X_{i_2}( p+m-2 )\cdots X_{k}(l_{1}+1)X_{i+1}(l_1 ) \cdots X_{\bar{j}}(l_{2}+1) X_{\bar{i}}(l_2)\cdots X_{i_{p}}(m),
  \end{align}
where $k<i$ and $i+1<j$. Now we can obtain
\[
\tilde{e}_{i}M_1=\left\{\begin{array}{lcl}  A_{i}( l_{1} )\cdot M_1 & \mbox{if} &  l_1\geq l_{2}+n-i+1,\\
                 A_{i}( l_{2}+n-i+1 )\cdot M_1& \mbox{if} & l_1<l_{2}+n-i+1.
                   \end{array}\right. 
\]
If $l_1\geq l_{2}+n-i+1$, we have $X_{i+1}\to X_i$, that is, 
\begin{eqnarray*}
 &&M'_1:=A_{i}( l_{1} )\cdot M_1\\
 &&=X_{i_1}( p+m-1 )X_{i_2}( p+m-2 )\cdots X_{k}(l_{1}+1)X_{i}(l_1 ) \cdots X_{\bar{j}}(l_{2}+1) 
X_{\bar{i}}(l_2)\cdots\in \mathcal{M}(Y_{p}( m )).
 \end{eqnarray*}
If $ l_1<l_{2}+n-i+1$, we get
\begin{eqnarray*}
&&M''_1:= A_{i}( l_{2}+n-i+1 )\cdot M_1\\
&&=X_{i_1}( p+m-1 )X_{i_2}( p+m-2 )\cdots X_{k}(l_{1}+1)X_{i+1}(l_1 ) \cdots X_{\bar{j}}(l_{2}+1)
X_{\overline{i+1}}(l_2)\cdots \in \mathcal{M}(Y_{p}( m )).
  \end{eqnarray*}
It is a remarkable fact  that from \eqref{Y forms} even if 
$l_1<l_{2}+n-i+1$ ( resp. $l_1\geq l_{2}+n-i+1$), we have 
$M'_1\in \mathcal{M}(Y_{p}( m ))$
(resp. $M''_1\in \mathcal{M}(Y_{p}( m ))$),
which means both $M_{1}^\prime, M_{1}''$ are 
in $\mathcal{M}(Y_{p}( m ))$
simultaneously.

The case $Y_{i}(l)^{-1}$ appears twice in $M_2$ as $Y_i(\alpha)^{-1}$ and $Y_i(\beta)^{-1}$ for some $\alpha, \beta\in \mathbb{Z}$ can be done  similarly to the case of $M_1$. If we write 
\[
\tilde{e}_{i}M_2=\left\{\begin{array}{rcl}  A_{i}(\alpha )\cdot M_2:= M_{2}^\prime & \mbox{if} & \alpha\geq \beta,\\
                 A_{i}(\beta )\cdot M_2:= M_{2}''& \mbox{if} & \alpha< \beta.
                   \end{array}\right. 
\]

By Proposition \ref{Y forms},  we obtain $ \tilde{e}_{i}M_2\in \mathcal{M}(Y_{q}( 1 ))$. It is remarkable fact that even if  $\alpha< \beta$ ( resp. $\alpha\geq  \beta$), we have $ A_{i}( \alpha )\cdot M_2\in\mathcal{M}(Y_{q}( 1))$$(\mbox{resp.}\,\,\, A_{i}(\beta  ))\cdot M_2\in\mathcal{M}(Y_{q}( 1)))$,
which means both $M_{2}^\prime, M_{2}''\in \mathcal{M}(Y_{q}( 1))$.  Then we have
\begin{align}
                \nonumber
\tilde{e}_{i}(M_1\cdot M_2 )=&M_{1}^\prime\cdot M_2\quad\mbox{or}\quad M_{1}''\cdot M_2\quad\mbox{or}\quad M_1\cdot M_2^\prime\\
\nonumber
&\mbox{or}\quad M_1\cdot M_{2}''\quad\mbox{or}\quad 0\in \mathcal{M}(Y_{p}( m ))\cdot\mathcal{M}(Y_{q}( 1 ))\cup \{0\},
 \end{align}
which mean 
 $\tilde{e}_{i}(M_1\cdot M_2 )\in  \mathcal{M}(Y_{p}( m ))\cdot\mathcal{M}(Y_{q}( 1))\cup\{0\}$.
\end{enumerate}
The case of $ \tilde{f}_{i}\, ( M_{1}\cdot{M}_{2} )$ is shown 
similarly.
 \end{proof}

\begin{theorem} \label{pos.crystal}
             The product $ \mathcal{M}(Y_{p}(  m) )\cdot \mathcal{M}(Y_{q}( 1 ))$
             possesses a crystal structure and it is decomposed into a direct sum of simple crystals, that is, 
            there exist dominant integral weights $\lambda_{1},\lambda_{2},\mathellipsis,\lambda_{k}\in
P_{+}$ \,such that 
\[
\mathcal{M}(Y_{p}(m) ) 
\cdot\mathcal{M}(Y_{q}( 1 ))\cong 
B( \lambda_{1} )\oplus B( \lambda_{2} ) \oplus\cdots
\oplus  B ( \lambda_{k} ).
\]
\end{theorem}

 \begin{proof}
The former half of the statement is almost trivial by Proposition \ref{crystal of product}. 
Indeed, we know that $\mathcal{M}(Y_{p}(m) ) 
\cdot\mathcal{M}(Y_{q}( 1 ))$ is a crystal if it is closed by the actions $\eit$ and $\fit$ $(i\in I)$.
For the latter half, since we suppose $\mathfrak{g}$ is of type $C_n$, 
by Corollary \ref{M crystal } we find that $\mathcal{M}_{c}$ is 
a direct sum of the form $B(\lambda)$'s $(\lambda\in P_{+})$ 
and then  $ \mathcal{M}(Y_{p} (m) ) 
\cdot\mathcal{M}(Y_{q}( 1 ))$ is also a subcrystal of $ \mathcal{M}_c$ and then it is 
a direct sum of $B(\lambda)$'s.
\end{proof}
By considering similarly to the proof of Proposition \ref{crystal of product},
 we obtain 
\begin{corollary} \label{general result}
For any $m_1,m_2, \mathellipsis m_l\in  \mathbb{ Z}$,  $p_1,p_2, \mathellipsis, p_l
\in I\, (l>0)$, the product of crystals 
\begin{center}
$ \mathcal{M}(Y_{p_1}( m_1) ) \cdot  \mathcal{M}(Y_{p_2}(m_2) ) 
\cdots \mathcal{M}(Y_{p_l}( m_l) )$ ,
\end{center}
possesses  a crystal structure and it is decomposed into direct sum of crystal $B(\lambda)$'s $(\lambda\in P_{+})$.
\end{corollary}
\subsection{Reductions}
In the sequel, we will treat monomials in 
$\mathcal{M}(C_{n-i})$ for $i=0,1,\cdots,n-1$, 
where $U_{q}(C_{n-i})$ is the subalgebra of 
$U_{q}(C_{n})$ generated by 
$\{e_j,f_j,k^{\pm1}_j|j\in\{i+1,i+2,\cdots,n\}\}$.
Note that the set of monomial ${\mathcal M}(C_{n-i})$ is generated by 
the variables $\{Y_j(m)^{\pm1}\,|\,j\in\{i+1,\cdots,n\},m\in{\mathbb Z}\}$.
Let us introduce a morphism of multiplicative group ${\mathcal M}(C_n)$, which is called a 
``$i$-reduction of monomial",  by
\begin{align}\label{iota def}
\iota_{i}(Y_{k}(m)^{\pm1})\coloneqq
\left\{\begin{array}{lcl} Y_{k}(m)^{\pm1}&  \mbox{if} & i\ne k,\\
       1 & \mbox{if} &i=k.
\end{array}\right. 
\end{align} 
Note that $\iota_{i}\iota_{i-1}\cdots\iota_{1}:
\mathcal{M}( C_n)\longrightarrow \mathcal{M}( C_{n-i})$. 

The following lemma is a generalization of Proposition \ref{general form by X}.
\begin{lemma}\label{mkm}
For $k\in[1,2n]$, set
\begin{equation}
M_k(m):=
\left\{X_{i_{1}}( k+m-1 )X_{i_{2}}( k+m-2 )
\cdots X_{i_{k}}(m)\vert i_j\in \mathcal{ I}, 
1\le i_{1} <i_{2}<\cdots < i_{k} \le \bar{1}\}\right.
\label{k12n}
\end{equation}
Then we get
\begin{equation}
M_k(m)\cong \begin{cases}
\mathcal{M}(Y_k(m))&{\rm for }\,\,k\in[1,n]=I,\\
\mathcal{M}(Y_{2n-k}(m-n+k))&{\rm for }\,\,k\in[n+1,2n].
\end{cases}
\end{equation}
\end{lemma}
{\sl Proof.} The case $k\in I$ is obtained by Proposition \ref{general form by X}. 
Then, we assume $k\in[n+1,2n]$.  
It is almost trivial that 
the set $M_k(m)$ is closed by the actions of $\eit$ and $\fit$ for any $i\in I$
by the proof of Proposition \ref{crystal of product}. Thus, we shall show that 
a unique highest weight vector in $M_k(m)$ is $Y_{2n-k}(m-n+k)$ by using 
induction on the rank $n$.
Let $M\in M_k(m)$ be a highest weight vector.
Then it is evident that $1$-reduction $\io_1(M)$ is a highest weight vector as a monomial in 
${\cM}(C_{n-1})$, which means the following by the induction hypothesis:
\begin{equation}
\io_1(M)=\wtil X_2(l'+k'-1)\wtil X_3(l'+k'-2)\cd\wtil X_n(l'+k'-n+1)\wtil X_{\ovl n}(l'+k'-n)
\cd \wtil X_{\ovl{2n-k'}}(l')
\end{equation}  
where $l',k'\in\bbZ$ and 
\begin{equation}
\widetilde{X_{i}}(l)=\iota_{1}(X_{i}(l))
=\left\{\begin{array}{rcl}& 1&\mbox{if}\,\,\,\,  i=1\mbox{or}\bar{1},\\
& Y_{2}(l)&\mbox{if} \,\,\,\,  i=2,\\
&  \frac{1} {Y_{2}(l+n-1)}& \mbox{if} \,\,\,\,  i=\bar{2},\\
& \frac{Y_{i}(l)}{Y_{i-1}(l+1)} &\mbox{if} \,\,\,\, i\in\{3,\cdots,n\},\\
& \frac{Y_{i-1}(l+n-i+1)}{Y_{i}(l+n-i+1)} 
&\mbox{if} \,\,\,\,  \bar{i}\in\{\bar{n},\cdots,\bar{3}\}.
\end{array}\right. \label{X-til}
\end{equation}
Then, we have that $M$ is in the one of the following forms:
\begin{eqnarray}
&&X_1(l'+k')X_2(l'+k'-1)\cd X_n(l'+k'-n+1) X_{\ovl n}(l'+k'-n)
\cd X_{\ovl{2n-k'}}(l')X_{\ovl 1}(l'-1),\q\label{m1-1}\\
&&X_1(l'+k')X_2(l'+k'-1)\cd X_n(l'+k'-n+1) X_{\ovl n}(l'+k'-n)
\cd X_{\ovl{2n-k'}}(l'),\q\label{m1}\\
&&X_2(l'+k'-1)X_3(l'+k'-2)\cd X_n(l'+k'-n+1) X_{\ovl n}(l'+k'-n)
\cd X_{\ovl{2n-k'}}(l')X_{\ovl 1}(l'-1),\qq\label{m2-1}\\
&&X_2(l'+k'-1)X_3(l'+k'-2)\cd X_n(l'+k'-n+1) X_{\ovl n}(l'+k'-n)
\cd X_{\ovl{2n-k'}}(l').\q\label{m2}
\end{eqnarray}
Indeed, in the cases \eqref{m2-1} and \eqref{m2}, their weights are not dominant and then
these cases cannot occur.
For \eqref{m1-1}, if $2n-k'=2$, then $M=1$ and it is a trivial highest weight monomial.
Otherwise, $\wt(M)$ is not dominant and then it cannot happen.
For \eqref{m1}, $M$ is a highest weight monomial and it can be written explicitly as
$Y_{2n-k'-1}(l'-n+k'+1)$ where $l'=m$ and $k'=k-1$. Thus, we obtain the uniqueness of the 
highest weight vector in $M_k(m)$ and the isomorphism $M_k(m)\cong {\cM}(Y_{2n-k}(m-n+k))$. 
\qed

\begin{definition}
 \begin{enumerate}
\item Take $p\in \{1,\ldots ,2n\}$ and $M_{1}\in M_p(m)$. 
By the general form of $M_1$ given in 
\eqref{k12n}, we set
\begin{align}
\nonumber
M_{1}&=X_{i_{1}}( p+m-1 )X_{i_{2}}( p+m-2 )\cdots
   X_{i_{p}}(m)\,\,(1\le i_{1}<\cdots <i_{p}\le \bar{1}).
\end{align}
  Assume that for some $d\in \{0,1,\mathellipsis,p\}\,$ there exists 
$j\in\{1,2,\cdots,\bar{2},\bar{1}\}$ satisfying $i_d<j<i_{d+1}$ (we set $i_{0}=0$, $i_{p+1}=\bar 0$), where we say this situation that {\it  there is a gap between $i_d$ and $i_{d+1}$ in the sequence  $(i_{1},\ldots, i_{p})$.}
\item We denote a {\itshape{consecutive sequence}} by the following forms:
\begin{align}
\nonumber
[X_{i}(l) \cdots X_{j}(l+i-j)]=&\prod\limits_{k=i}^{j}X_{k}(l+i-k)
\qquad (i<j),\\
\nonumber
[X_{\bar{j}}(l) \cdots X_{\bar{i}}(l+j-i)]=
&\prod\limits_{k={i}}^{j}X_{\bar{k}}(l+j-k)\qquad (i<j),\\
\nonumber
[X_{i}(l) \cdots X_{\bar{j}}(l+i+j-2n-1)]=
&\prod\limits_{k=i}^{n}X_{k}(l+i-k)\prod\limits_{k=j}^{n}X_{\bar{k}}(l+i+k-2n-1)\qquad(i,j<n).
 \end{align} 
\item
For two monomials $L_1=X_{i_1}(m_1)\cd X_{i_k}(m_k)$ and $L_2=X_{j_1}(n_1)\cd X_{j_l}(n_l)$, if 
$\{(i_1,m_1),\cd,(i_k,m_k)\}\subset \{(j_1,n_1),\cd,(j_l,n_l)\}$, it is denoted 
$L_1\subset L_2$.
\end{enumerate}
 \end{definition}

In the sequel, we use frequently shortened notations $X_{i}, Y_{i}$ for $X_{i}(l), Y_{i}(l)$ for some $l$, if we do not  need the exact form of $l$.

\begin{definition}
For  a monomial $M=X_{i_1}(m+p-1)\cdots X_{i_p}(m)\in M_p(m)$ satisfying 
$1\le i_1<\cdots<i_p\le\overline1$, 
and $m\in \mathbb{ Z}$, we say that the {\it length} of $M$ is $p$, 
which is denoted by $l(M)$. 
And define  $h(M):=m+(p-n)_+$, which is called the  {\it height}
of $M$. Here note that the height of a monomial $M\in M_p(m)$ is the index of its 
unique highest weight monomial $X_1(m+p-1)\cd X_p(m)=Y_p(m)$ with $p\leq n$ or 
$X_1(m+p-1)\cd X_{\ovl{2n-p+1}}(m)=Y_{2n-p}(m-n+p)$ with $p>n$.
 \end{definition}
\begin{remark}
For any monomial $M=X_{j_{i+1}}(m+p-i-1)
\cdots X_{j_p}(m)\in M_p(m)_{C_{n-i}}$ 
with $i\in\{0,1,\cd,n-1\}$, $i+1\le j_{i+1}<\cdots<j_p\le\overline{i+1}$, 
$p\in\{i+1,\cdots,2n\}$ and $m\in \mathbb{ Z}$, we have $l(M)=p$
and $h(M)=m+(p-n)_+$. For example, in $i=2$-case, for $p\in\{n+1,\ldots,2n\}$
the highest weight monomial in $M_p(m)$ is $X_3(m+p-3)\cd X_n(m+p-n-2)X_{\ovl n}(p+m-n-3)\cd
X_{\ovl{2n-p+1}}(m)=Y_{2n-p}(m-n+p)$, whose height is $m-n+p$.
\end{remark}

\begin{lemma} \label{M has gaps}
For $m\geq1$ and $p,q\in \{1,\ldots ,2n-1\}\,$,  assume
\begin{equation}
m+(p-n)_+\geq 1+(q-n)_+.
\label{mpnq}
\end{equation}
If a monomial $M_1\in M_p(m)$ is in one of the following forms (1), (2) and (3), 
then a product $M_{1}\cdot{M}_{2}$ is not a highest weight vector 
for any $M_2\in M_q(1)$:
\begin{enumerate}
\item
\begin{eqnarray}
\hspace{-56pt}&&M_1=\begin{cases} [X_{2}(p+m-1)\cdots X_{p}(m+1)]X_{\bar{1}}(m)
=\frac{Y_{p}(m+1)}{Y_{1}(p+m)}\cdot \frac{1}{Y_1(m+n)}&\hbox{if }p\leq n,\\
[X_{2}(p+m-1)\cdots X_{\ovl{2n-p+1}}(m+1)]X_{\bar{1}}(m)
=\frac{Y_{2n-p}(m+p-n+1)}{Y_{1}(p+m)}\cdot \frac{1}{Y_1(m+n)}&\hbox{if }p>n.
\end{cases}\label{M1-1}
\end{eqnarray}
\item \begin{eqnarray}
&&\hspace{-100pt}M_1=\begin{cases} [X_{2}(p+m-1)\cdots X_{p+1}(m)]
=\frac{Y_{p+1}(m)}{Y_{1}(p+m)}&\hbox{if }p< n,\\
[X_{2}(p+m-1)\cdots X_{\ovl{2n-p}}(m)]
=\frac{Y_{2n-p-1}(m+p-n+1)}{Y_{1}(p+m)}&\hbox{if }p\geq n.
\end{cases}\label{M1-2}
\end{eqnarray}
\item
\begin{eqnarray}
\hspace{-40pt}&&M_1=\begin{cases} [X_{1}(p+m-1)\cdots X_{p-1}(m+1)]X_{\bar{1}}(m)
=\frac{Y_{p-1}(m+1)}{Y_{1}(m+n)}&\hbox{if }p\leq n,\\
[X_{1}(p+m-1)\cdots X_{\ovl{2n-p+2}}(m+1)]X_{\bar{1}}(m)
=\frac{Y_{2n-p+1}(m+p-n)}{Y_{1}(m+n)}&\hbox{if }p>n,
\end{cases}\label{M1-3}
\end{eqnarray}
\end{enumerate}
\end{lemma}

\begin{proof}
 Assume that $M_1\cdot M_2$ is a highest weigh vector. \\
(1) 
If $1<p<2n-1$ and $\wt(M_1\cdot M_2)$ is dominant, there are $X_1$ and $X_{\ovl 2}$ but no $X_2$ nor 
$X_{\ovl 1}$ in $M_2$ since  $\wt(M_2)$ should have 
the weight $2\epsilon_1=2\Lm_1$ as a summand by the explicit form of $M_1$.
Thus, under the condition $p>1$, $M_2$ has the factor $Y_1(q)$ and $Y_1(n)$. 
Then we have 
\[
\vep_1(M_1\cdot M_2)=\vep_1\left(\frac{Y_1(q)Y_1(n)}{Y_1(m+n)Y_1(p+m)}\right)\geq 1,
\]
since $m+n>q, n$. This derives a contradiction to that $M_1\cdot M_2$ is a highest weight vector.

If $p=1$, $M_1=\frac{1}{Y_1(m+n)}$ and if $p=2n-1$, then $M_1=\frac{1}{Y_1(m+p)}$ and then 
${\rm wt}(M_1)=-\Lm_1=-\epsilon_1$. 
In this case, if the factor $Y_1$ is in $M_2$, then it is in $X_1(q)=Y_1(q)$ or 
$Y_{\ovl 2}(1)=\frac{X_1(n)}{X_2(n)}$ and then by $q,n<n+m, (2n-1)+m$ we have
\[
\vep_1(M_1\cdot M_2)=
\vep_1\left(\frac{Y_1(q)}{Y_1(x+m)}\right)=1 \hbox{ or }
\vep_1\left(\frac{Y_1(n)}{Y_1(x+m)}\right)=1 \hbox{ or }
\vep_1\left(\frac{Y_1(q)Y_1(n)}{Y_1(x+m)}\right)=1, \q
(x=n\hbox{ or } 2n-1),
\]
which means that $M_1\cdot M_2$ cannot be a highest weight monomial.

(2) 
For a monomial  $M_2=X_{j_1}(q)\cdots X_{j_q}(1)\in M_q(1)$, assume that 
$M_1\cdot M_2$ is a highest weight vector. 
It follows from the explicit form in \eqref{M1-2} that $M_1$ has a factor $Y_1(p+m)^{-1}$. 
Since wt$(M_1\cdot M_2)$ must be a dominant weight, one finds that 
$M_2$ holds $Y_1(q)$ or $Y_1(n)$, which are from the term $X_1(q)$ or $X_{\bar 2}(1)$.
Thus, considering similarly to the case (1), 
if $q<p+m$ and $n<p+m$, then $\vep_1(M_1\cdot M_2)\geq1>0$, which implies
that $M_1\cdot M_2$ cannot be a highest weight vector.

Next, suppose that $q\geq p+m$ or $n\geq p+m$.
Let us show the following lemma.
\begin{lemma}
\begin{enumerate}
\item If $M_1=\frac{Y_{p+1}(m)}{Y_{1}(p+m)}$ in the case $p<n$ and $M_1\cdot M_2$ is a highest weight monomial,  then $\eit M_2=0$ for $i\in I\setminus\{p+1\}$, 
\item If $M_1=\frac{Y_{2n-p-1}(m+p-n+1)}{Y_{1}(p+m)}$ in the case $p\geq n$ and $M_1\cdot M_2$ is a highest weight monomial,  then $\eit M_2=0$ for $i\in I\setminus\{2n-p-1\}$. 
\end{enumerate}
\end{lemma}
\begin{proof} The case (ii) is shown similarly to the case (i), so we shall see only (i).
The case $i\ne1$ is trivial. For $i=1$, by the definition of $\vep_1$ in \eqref{varepsilon}, we have
\[
0=\vep_1(M_1\cdot M_2)=\vep_1(\frac{M_2}{Y_1(p+m)})\geq \vep_1(M_2)\geq0,
\]
which means $\vep_1(M_2)=0$.
\end{proof}
By this lemma, we find that if $M_1\cdot M_2$ is a highest weight vector 
and $M_1$ is in the form \eqref{M1-2}, then $M_2$ is in the following form:
\[
M_2=\begin{cases}[X_1(q)\cd X_l(q-l+1)][X_{p+2}(q-l)\cd]&\hbox{ if }p<n,\q{\rm (A)}\\
[X_1(q)\cd X_l(q-l+1)][X_{\ovl{p+1}}(q-l)\cd]&\hbox{ if }p<n,\q{\rm (B)}\\
[X_1(q)\cd X_l(q-l+1)][X_{2n-p}(q-l)\cd]&\hbox{ if }p\geq n,\q{\rm (C)}\\
[X_1(q)\cd X_l(q-l+1)][X_{\ovl{2n-p-1}}(q-l)\cd]&\hbox{ if }p\geq n,\ \q{\rm (D)}
\end{cases}
\]
where in (B) and (D), if $l>n$, we understand that $X_l$ is $X_{\ovl{2n-l+1}}$.
Here note that $[X_{p+2}\cd]$ and $[X_{\ovl{p+1}}\cd]$ (resp. $[X_{2n-p}\cd]$ and 
$[X_{\ovl{2n-p-1}}\cd]$) cannot occur in $M_2$ simultaneously since their 
denominators are both both $Y_{p+1}$ (resp. $Y_{2n-p-1}$) and the weight of $M_1\cdot M_2$ must be dominant.

 Assume $l>1$. Let us see (A) and (B) for $M_2$.  Since $\wt(M_1\cdot M_2)$ is dominant, 
 we find that $M_2$ should have the factor $Y_1$, then the end part of $M_2$ would be $X_{\ovl 2}(1)$. Thus, $M_2$ can be written in the following forms:
 \[
 M_2=\begin{cases}\frac{Y_l(q-l+1)Y_1(n)}{Y_{p+1}(q-l+1)},&{\rm (A)},\\
\frac{Y_l(q-l+1)Y_1(n)}{Y_{p+1}(n)},&{\rm (B)}\q{l\leq n},\\
\frac{Y_{\ovl{2n-l+1}}(q-l+1)Y_1(n)}{Y_{p+1}(n)},&{\rm (B)}\q{l> n}.
\end{cases}
\]
In case (A), the length of $M_2=[X_1\cd X_l][X_{p+2}\cd X_{\ovl 2}]$ is $q=l+(2n-p-2)$.
By $0=\vep_1(M_1\cdot M_2)=\vep_1(\frac{Y_1(n)}{Y_1(p+m)})$, we get $n\geq p+m$. 
And we also get $m\geq q-l+1$ from 
$0=\vep_{p+1}(M_1\cdot M_2)=\vep_{p+1}(\frac{Y_{p+1}(m)}{Y_{p+1}(q-l+1)})$. These induce inequality
\[
n+l-1\geq p+q=2n+l-2,
\]
which implies $1\geq n$ and then derives contradiction. Then, this case (A) never occur.

Next, let us see (B) with $1<l<n$. In this case, we get 
the inequalities $n\geq p+m$ from $0=\vep_1(M_1\cdot M_2)=\vep_1(\frac{Y_1(n)}{Y_1(p+m)})$ and 
$m\geq n$ from $0=\vep_1(M_1\cdot M_2)=\vep_{p+1}(\frac{Y_{p+1}(m)}{Y_{p+1}(n)})$. These induce
$m\geq m+p$, which contradicts $q\geq1$. Thus, this case (B) with $1<l<n$ does not occur.

Let us see (B) with $l>n$. Considering similarly to the previous case, we obtain 
the inequality $n\geq p+m$ by $\vep_1(M_1)\cdot M_2=0$ and $m\geq n$ by $\vep_{p+1}(M_1\cdot M_2)=0$,
which derives $m\geq p+m$ and then a contradiction. Thus, this case cannot  occur.
Indeed, in the cases (C) and (D) it will be shown similarly that $M_1\cdot M_2$ never be 
a highest weight monomial. 

Now, let us assume $l=1$. We have for some $\xi\in\bbZ$ and $x\in I$, 
\[
M_2=\begin{cases}Y_1(q)\cdot \frac{Y_x(\xi)}{Y_{p+1}(q)},&p<n(\Leftrightarrow \hbox{ (A) or (B)})\,
\\
Y_1(q)\cdot \frac{Y_x(\xi)}{Y_{2n-p-1}(q)},&p\geq n(\Leftrightarrow\hbox{\rm (C) or (D)}),
\end{cases}\qq
M_1=\begin{cases}
\frac{Y_{p+1}(m)}{Y_1(p+m)}&p<n,\\
\frac{Y_{2n-p-1}(m+p-n+1)}{Y_1(p+m)}&p\geq n.
\end{cases}
\]
Assume $p<n$ and $x\ne1,\,p+1$. Then, we have 
the inequalities $q\geq p+m$ and $m\geq q$ by the conditions 
$\vep_1(M_1\cdot M_2)=\vep_{p+1}(M_1\cdot M_2)=0$, which derive a contradiction $m\geq p+m$. Thus, 
the cases (A),(B) with $x\ne1,\,p+1$ cannot occur. 
Assume $x=1$. In this case, we find that $\xi=n$ since $Y_{\ovl 2}(1)=\frac{Y_1(n)}{Y_2(n)}$.
Then by the conditions $\vep_1(M_1\cdot M_2)=\vep_{p+1}(M_1\cdot M_2)=0$, we get 
($q\geq p+m$, or $n\geq p+m$) and $m\geq q$. We also find that length$(M_2)=q=2n-p-1$ and then
$q\geq n$ since $p<n$. Thus, we obtain $m\geq q\geq n\geq p+m$, which is a contradiction. Thus, 
(A), (B) with $x=1$ never occur. Let us see the case $x=p+1$. In the case (A),  we have 
$M_2=Y_1(q)\frac{Y_{p+1}(n-p)}{Y_{p+1}(q)}$ and then we have the inequality $q\geq p+m$ and 
($m\geq q$ or $q-n\geq q$), which derive a contradiction $m\geq p+m$.
In the case (B), we have 
$M_2=Y_1(1)$ with $q=1$. Thus, from the condition $\vep_1(M_1\cdot M_2)=0$, we have  
$q=1\geq p+m$, which derives a contradiction.  Thus, the cases (A),(B) never happen. By arguing similar 
to the cases (A), (B), we know that the cases (C),(D) never happen.

\medskip
\noindent
(3) By considering similarly to the case (1), 
we find that
in almost all cases $M_1\cdot M_2$ cannot be a highest weight monomial, except the case $p=2$.
Let us see the case $p=2$. In this case, $M_1=\frac{Y_1(m+1)}{Y_1(m+n)}$ and then we get 
\[
\vep_1(M_1\cdot M_2)=
\vep_1\left(\frac{Y_1(m+1)Y_1(q)}{Y_1(n+m)}\right)=1 \hbox{ or }
\vep_1\left(\frac{Y_1(m+1)Y_1(n)}{Y_1(n+m)}\right)=1 \hbox{ or }
\vep_1\left(\frac{Y_1(m+1)Y_1(q)Y_1(n)}{Y_1(n+m)}\right)=1, 
\]
since $n+m>q,\,m+1,\,n$, which imply that $M_1\cdot M_2$ cannot be a highest weight monomial. 
\end{proof}

\subsection{Examples}
\noindent\indent
Here we list the decomposition of the monomial product and the highest weight monomials 
for type $C_2$, which are not only examples but also play a role of the first step of induction in the proof of Theorem \ref{induction on n} below.
\begin{example}\label{example C2}
Let us see the case of type $C_2$. Then $I=\{1,2\}$.
We see the monomial product $M_p(m)\cdot M_q(1)$ for $p,q\in\{1,2,3\}$.
Note that $M_4(m)=\{1\}$ is isomorphic to $B(0)$ and then,  there is nothing to do.
\begin{enumerate}
\item For $p=1$ and $q=1$, we have $m+(p-n)_+=m$ and $1+(q-n)_+=1$. Then we get 
\begin{enumerate}
\item $m\geq3$: $M_1(m)\cdot M_1(1)=\mathcal{M}(Y_{1}( m ))\cdot 
      \mathcal{M}(Y_{1}( 1) )
      \cong B(2\Lambda _1)\oplus B(\Lambda _2)\oplus B(0)$ and  
      the highest weight vector of $B(2\Lambda _1)$ is $ Y_1(m)\cdot Y_1(1)$, 
       of $ B(\Lambda _2)$ is $ Y_1(m)\cdot Y_2(1)\cdot Y_1(2)^{-1}$ and 
       of $B(0)$ is $ Y_1(m)\cdot Y_1(3)^{-1}$.
\item $m=2$: $M_1(2)\cdot M_1(1)=\mathcal{M}(Y_{1}( 2 ))\cdot 
      \mathcal{M}(Y_{1}( 1) )\cong 
      B(2\Lambda _1)\oplus B(\Lambda _2)$,  the highest weight vector of  
      $B(2\Lambda _1)$ is $ Y_1(2)\cdot Y_1(1)$ and of $ B(\Lambda _2)$ is 
      $ Y_1(2)\cdot Y_2(1)\cdot Y_1(2)^{-1}$.
\item $m=1$: $\mathcal{M}(Y_{1}( 1 ))\cdot \mathcal{M}(Y_{1}( 1) )\cong 
      B(2\Lambda _1)$,  the highest weight vector of $B(2\Lambda _1)$ is $ Y_1(1)\cdot Y_1(1)$.
\end{enumerate}

\item For $p=1$ and $q=2$, we have $m+(p-n)_+=m$ and $1+(q-n)_+=1$. Then we get
\begin{enumerate}
\item  $m\geq3$: $M_1(m)\cdot M_2(1)=\mathcal{M}(Y_{1}(m ))\cdot \mathcal{M}
       (Y_{2}(1) )\cong B(\Lambda _2+\Lambda _1)\oplus B(\Lambda _1)$ and 
      the highest weight vector of $B(\Lambda _2+\Lambda _1)$ is  $Y_1(m)\cdot Y_2(1)$ 
      and of $B(\Lambda _1)$ is $ Y_1(m)\cdot Y_1(3)^{-1}\cdot Y_1(2)$.
\item  $m=2$: $M_1(2)\cdot M_2(1)=\mathcal{M}(Y_{1}(2 ))\cdot 
       \mathcal{M}(Y_{2}( 1) )
      \cong B(\Lambda _2+\Lambda _1)$, the highest weight vector of $B(\Lambda _2+\Lambda _1)$ 
      is  $Y_1(2)\cdot Y_2(1)$
\item $m=1$: $M_1(1)\cdot M_2(1)=\mathcal{M}(Y_{1}( 1 ))\cdot 
      \mathcal{M}(Y_{2}( 1) )\cong 
       B(\Lambda _2+\Lambda _1)$,   the highest weight vector of 
       $B(\Lambda _2+\Lambda _1)$ is  $Y_1(1)\cdot Y_2(1)$.
\end{enumerate}
\item For $p=2$ and $q=1$, we also have $m+(p-n)_+=m$ and $1+(q-n)_+=1$. Then we get
\begin{enumerate}
\item $m\geq2$: $M_2(m)\cdot M_1(1)=\mathcal{M}(Y_{2}( m ))\cdot  
      \mathcal{M}(Y_{1}( 1) )\cong B(\Lambda _2+\Lambda _1)\oplus B(\Lambda _1)$ 
      and  the highest weight vector of $B(\Lambda _2+\Lambda _1)$ is $ Y_2(m)\cdot Y_1(1)$ 
      and of $B(\Lambda _1)$ is $ Y_2(m)\cdot Y_2(2)^{-1}\cdot Y_1(1)$.
\item $m=1$: $M_2(1)\cdot M_1(1)=\mathcal{M}(Y_{2}(1 ))\cdot \mathcal{M}(Y_{1}( 1) )\cong B(\Lambda _2+\Lambda _1)$,  the highest weight vector of $B(\Lambda _2+\Lambda _1)$ is $ Y_2(1)\cdot Y_1(1)$.
\end{enumerate}

\item For $p=2$ and $q=2$, we also have $m+(p-n)_+=m$ and $1+(q-n)_+=1$. Then we get
\begin{enumerate}
\item $m\geq 3$:$\mathcal{M}(Y_{2}( m ))\cdot \mathcal{M}(Y_{2}( 1) )\cong B(2\Lambda _2)\oplus B(2\Lambda _1)\oplus B(0)$ and the highest weight vector of $B(2\Lambda _2)$ is $ Y_2(m)\cdot Y_2(1)$, of $ B(2\Lambda _1)$ is $ Y_2(m)\cdot Y_2(2)^{-1}\cdot Y_1(2)^{2}$ and of $B(0)$ is $ Y_2(m)\cdot Y_2(3)^{-1}$.
\item $m=2$: $\mathcal{M}(Y_{2}( 2 ))\cdot \mathcal{M}(Y_{2}( 1) )\cong B(2\Lambda _2)\oplus B(2\Lambda _1)$,  the highest weight vector of $B(2\Lambda _2)$ is $ Y_2(2)\cdot Y_2(1)$ and of $ B(2\Lambda _1)$ is $ Y_2(2)\cdot Y_2(2)^{-1}\cdot Y_1(2)^{2}$.
\item $m=1$: $\mathcal{M}(Y_{2}( 1 ))\cdot \mathcal{M}(Y_{2}( 1) )\cong B(2\Lambda _2)$, the highest weight vector of $B(2\Lambda _2)$ is $ Y_2(1)\cdot Y_2(1)$.
\end{enumerate} 

\item For $p=1$ and $q=3$, we have $m+(p-n)_+=m$ and $1+(q-n)_+=1+1=2$. 
By Lemma \ref{mkm}, $M_3(1)={\mathcal M}(Y_1(2))=M_1(2)$.  Then, by the result of (1), we get 
\begin{enumerate}
\item $m\geq4$: $M_1(m)\cdot M_3(1)=\mathcal{M}(Y_{1}( m ))\cdot 
      \mathcal{M}(Y_{1}(2) )
      \cong B(2\Lambda _1)\oplus B(\Lambda _2)\oplus B(0)$ and  
      the highest weight vector of $B(2\Lambda _1)$ is $ Y_1(m)\cdot Y_1(2)$, 
       of $ B(\Lambda _2)$ is $ Y_1(m)\cdot Y_2(2)\cdot Y_1(3)^{-1}$ and 
       of $B(0)$ is $ Y_1(m)\cdot Y_1(4)^{-1}$.
\item $m=3$: $M_1(3)\cdot M_3(1)=\mathcal{M}(Y_{1}( 2 ))\cdot 
      \mathcal{M}(Y_{1}(2) )\cong 
      B(2\Lambda _1)\oplus B(\Lambda _2)$,  the highest weight vector of  
      $B(2\Lambda _1)$ is $ Y_1(3)\cdot Y_1(2)$ and of $ B(\Lambda _2)$ is 
      $ Y_1(3)\cdot (Y_2(2)\cdot Y_1(3)^{-1})$.
\item $m=2$: $M_1(2)\cdot M_3(1)=\mathcal{M}(Y_{1}( 2 ))\cdot 
       \mathcal{M}(Y_{1}(2) )\cong 
      B(2\Lambda _1)$,  the highest weight vector of $B(2\Lambda _1)$ is $ Y_1(2)\cdot Y_1(2)$.
\end{enumerate}

\item For $p=2$ and $q=3$, we have $m+(p-n)_+=m$ and $1+(q-n)_+=1+1=2$. 
By Lemma \ref{mkm}, $M_3(1)={\mathcal M}(Y_1(2))=M_1(2)$.  Then, by the result of (3), we get 
\begin{enumerate}
\item $m\geq3$: $M_2(m)\cdot M_3(1)=\mathcal{M}(Y_{2}( m ))\cdot  
      \mathcal{M}(Y_{1}(2) )\cong B(\Lambda _2+\Lambda _1)\oplus B(\Lambda _1)$ 
      and  the highest weight vector of $B(\Lambda _2+\Lambda _1)$ is $ Y_2(m)\cdot Y_1(1)$ 
      and of $B(\Lambda _1)$ is $ Y_2(m)\cdot Y_2(3)^{-1}\cdot Y_1(2)$.
\item $m=2$: $M_2(2)\cdot M_3(1)=\mathcal{M}(Y_{2}(2))\cdot \mathcal{M}(Y_{1}(2) )\cong B(\Lambda _2+\Lambda _1)$,  the highest weight vector of $B(\Lambda _2+\Lambda _1)$ is $ Y_2(2)\cdot Y_1(2)$.
\end{enumerate} 

\item For $p=3$ and $q=1$, we have $m+(p-n)_+=m+1$ and $1+(q-n)_+=1$. 
By Lemma \ref{mkm}, $M_3(m)={\mathcal M}(Y_1(m+1))=M_1(m+1)$.  Then, by the result of (1), we get 
\begin{enumerate}
\item $m\geq2$: $M_3(m)\cdot M_1(1)=\mathcal{M}(Y_{1}( m+1 ))\cdot 
      \mathcal{M}(Y_{1}( 1) )
      \cong B(2\Lambda_1)\oplus B(\Lambda_2)\oplus B(0)$ and  
      the highest weight vector of $B(2\Lambda _1)$ is $ Y_1(m+1)\cdot Y_1(1)$, 
       of $ B(\Lambda_2)$ is $ Y_1(m+1)\cdot Y_2(1)\cdot Y_1(2)^{-1}$ and 
       of $B(0)$ is $ Y_1(m+1)\cdot Y_1(3)^{-1}$.
\item $m=1$: $M_3(1)\cdot M_1(1)=\mathcal{M}(Y_{1}( 2 ))\cdot 
      \mathcal{M}(Y_{1}( 1) )\cong 
      B(2\Lambda _1)\oplus B(\Lambda _2)$,  the highest weight vector of  
      $B(2\Lambda _1)$ is $ Y_1(2)\cdot Y_1(1)$ and of $ B(\Lambda _2)$ is 
      $ Y_1(2)\cdot Y_2(1)\cdot Y_1(2)^{-1}$.
\item $m=0$: $M_3(0)\cdot M_1(1)=\mathcal{M}(Y_{1}( 1 ))\cdot \mathcal{M}(Y_{1}( 1) )\cong 
      B(2\Lambda _1)$,  the highest weight vector of $B(2\Lambda _1)$ is $ Y_1(1)\cdot Y_1(1)$.
\end{enumerate} 

\item For $p=3$ and $q=2$, we have $m+(p-n)_+=m+1$ and $1+(q-n)_+=1$. 
By Lemma \ref{mkm}, $M_3(m)={\mathcal M}(Y_1(m+1))=M_1(m+1)$.  Then, by the result of (2), we get
\begin{enumerate}
\item  $m\geq2$: $M_3(m)\cdot M_2(1)=\mathcal{M}(Y_{1}(m+1))\cdot \mathcal{M}
       (Y_{2}(1))\cong B(\Lambda _2+\Lambda _1)\oplus B(\Lambda _1)$ and 
      the highest weight vector of $B(\Lambda _2+\Lambda _1)$ is  $Y_1(m+1)\cdot Y_2(1)$ 
      and of $B(\Lambda _1)$ is $ Y_1(m+1)\cdot Y_1(3)^{-1}\cdot Y_1(2)$.
\item  $m=1$: $M_3(1)\cdot M_2(1)=\mathcal{M}(Y_{1}(2 ))\cdot 
       \mathcal{M}(Y_{2}( 1) )
      \cong B(\Lambda _2+\Lambda _1)$, the highest weight vector of $B(\Lambda _2+\Lambda _1)$ 
      is  $Y_1(2)\cdot Y_2(1)$
\item $m=0$: $M_3(0)\cdot M_2(1)=\mathcal{M}(Y_{1}( 1 ))\cdot 
      \mathcal{M}(Y_{2}( 1) )\cong 
       B(\Lambda _2+\Lambda _1)$,   the highest weight vector of 
       $B(\Lambda _2+\Lambda _1)$ is  $Y_1(1)\cdot Y_2(1)$.
\end{enumerate}

\item For $p=3$ and $q=3$, we have $m+(p-n)_+=m+1$ and $1+(q-n)_+=2$. 
By Lemma \ref{mkm}, $M_3(m)={\mathcal M}(Y_1(m+1))=M_1(m+1)$.  Then, by the result of (1), we get
\begin{enumerate}
\item $m\geq3$: $M_3(m)\cdot M_3(1)=\mathcal{M}(Y_{1}( m+1 ))\cdot 
      \mathcal{M}(Y_{1}(2) )
      \cong B(2\Lambda _1)\oplus B(\Lambda _2)\oplus B(0)$ and  
      the highest weight vector of $B(2\Lambda _1)$ is $ Y_1(m+1)\cdot Y_1(2)$, 
       of $ B(\Lambda _2)$ is $ Y_1(m+1)\cdot Y_2(2)\cdot Y_1(3)^{-1}$ and 
       of $B(0)$ is $ Y_1(m+1)\cdot Y_1(4)^{-1}$.
\item $m=2$: $M_3(2)\cdot M_3(1)=\mathcal{M}(Y_{1}( 3 ))\cdot 
      \mathcal{M}(Y_{1}( 2) )\cong 
      B(2\Lambda _1)\oplus B(\Lambda _2)$,  the highest weight vector of  
      $B(2\Lambda _1)$ is $ Y_1(3)\cdot Y_1(2)$ and of $ B(\Lambda _2)$ is 
      $ Y_1(2)\cdot Y_2(2)\cdot Y_1(3)^{-1}$.
\item $m=1$: $M_3(1)\cdot M_3(1)=\mathcal{M}(Y_{1}( 2 ))\cdot 
      \mathcal{M}(Y_{1}(2) )\cong 
      B(2\Lambda _1)$,  the highest weight vector of $B(2\Lambda _1)$ is $ Y_1(2)\cdot Y_1(2)$.
\end{enumerate} 
Here note that in these cases, if $M_1\cdot M_2\in M_p(m)\cdot M_q(1)$ is a highest weight 
monomial, then we have $M_1=\begin{cases}Y_p(m)&\hbox{ if }p=1,2,\\Y_{4-p}(m-2+p)&\hbox{ if }p=3.
\end{cases}$

\end{enumerate} 
\end{example}
\subsection{Properties of the product 
$M_p(m)\cdot M_q(1)$}
  
In the sequel, as before we will treat monomials in 
$\mathcal{M}(C_{n-i})$ for $i=0,1,\cdots,n-1$, 
where $U_{q}(C_{n-i})$ is the subalgebra of 
$U_{q}(C_{n})$ generated by 
$\{e_j,f_j,k^{\pm1}_j|j\in\{i+1,i+2,\cdots,n\}\}$.
Note that the set of monomial ${\mathcal M}(C_{n-i})$ is generated by 
the variables $\{Y_j(m)^{\pm1}\,|\,j\in\{i+1,\cdots,n\},m\in{\mathbb Z}\}$

\begin{theorem}\label{induction on n}
Let $M_{1}\in M_p(m)\subset \mathcal{M}(C_{n})$, 
$M_{2}\in M_q(1)
\subset \mathcal{M}(C_{n})$    
with $h(M_1)\geq h(M_2)$ and $p=l(M_1),q=l(M_2)\in \{1,\ldots, 2n\}\,$. 
If $M_{1}\cdot{M}_{2}$ is a highest weight vector,  
then we have 
\begin{equation}
M_1=\begin{cases}Y_p(m)&\hbox{ if }p\in\{1,\cd,n\},\\Y_{2n-p}(m-n+p)&\hbox{ if }p\in\{n+1,\cd,2n\}.\end{cases}
\label{hwv-M1}
\end{equation}
\end{theorem}

\begin{proof}
 We will use the induction on $n$:
First, by Example \ref{example C2}, 
we find that the claim is correct in the $C_2$-case, 
which is the first step of induction.
Next, assume that the case $C_{n-1}$ holds, 
which means that we consider monomials in 
$\mathcal{M}(C_{n-1})$, 
where we take simple roots $\{\alpha_2,\cdots,\alpha_n\}$ 
from the one for $C_n$ and for 
$M_{1}\in M_p(m)_{C_{n-1}}$, $M_{2}\in M_q(1)_{C_{n-1}}$ with 
$m+ (p-n)_+\geq 1+(q-n)_+$ and $p,q\in \{2,\ldots,2n-1\}$ if 
$M_{1}\cdot{M}_{2}$ is a highest weight vector, that is, 
$\varepsilon_{i}( M_{1}\cdot M_{2})=0$ for any $i\in\{ 2, \ldots, n\}$ then  we get
\[
M_{1}=\begin{cases}Y_{p}(m)&\hbox{ if }p\in\{2,\cd, n\},\\
Y_{2n-p}(m+p-n)&\hbox{ if } p\in\{n+1,\cd,2n-1\}, 
\end{cases}
\]
where note that 
\[
M_p(m)_{C_{n-1}}=\{X_{i_2}(m+p-1)X_{i_3}(m+p-2)\cd X_{i_p}(m)\mid 2\leq i_2<\cd<i_p\leq \bar 2\}
\,\,(p\in\{2,\cd,2n-1\}).
\]
Now, let us show that the case $C_{n}$ holds, that is, 
if $M_{1}\cdot{M}_{2}$ is a highest weight vector, then $M_1$ is in the form of 
\eqref{hwv-M1}.  Here we assume that 
\begin{equation}
p<2n,\label{p<2n}
\end{equation}
since if $p=2n$, then $M_1=1$, which is a trivial case.
By the general forms of $M_{1}$ and $M_2$  given in \eqref{k12n}, we set
\begin{align}
\nonumber
M_{1}&=X_{i_{1}}( p+m-1 )X_{i_{2}}( p+m-2 )\cdots X_{i_{p}}(m)\,\,\,\,(1\le i_{1}<\cdots <i_{P}\le \bar{1}),\\
\nonumber
M_{2}&=X_{j_{1}}(q)X_{j_{2}}(q-1)\cdots X_ {j_{q}}(1)\,\,\,\,(1\le j_{1}<\cdots<j_{q}\le \bar{1}).
\end{align}
Here we recall the $1$-reduction 
$\iota_{1}:
\mathcal{M}( C_n)\longrightarrow \mathcal{M}( C_{n-1})$
as in 
\eqref{iota def}
\[
\iota_{1}(Y_{k}(m)^{\pm1})\coloneqq
\left\{\begin{array}{lcl} Y_{k}(m)^{\pm1}&  \mbox{if} & k\ne1,\\
       1 & \mbox{if} &k=1.
\end{array}\right. 
\]  
Then we have
\begin{align}
\nonumber
\iota_{1}(  M_{1})&=\widetilde{M_1}=\iota_{1}(X_{i_{1}}( p+m-1 ))\iota_{1}(X_{i_{2}}( p+m-2 ))\cdots\iota_{1}( X_{i_{p}}(m)),\\
\nonumber
\iota_{1}(  M_{2})&=\widetilde{M_2}=\iota_{1}(X_{j_{1}}(q))\iota_{1}(X_{j_{2}}(q-1))\cdots \iota_{1}(X_ {j_{q}}(1)),
\end{align}
where we set $\widetilde{X_{i}}(l)$ as in \eqref{X-til}. 
Then we obtain 
\begin{align*}
&\widetilde{M_1}=
\left\{
\begin{array}{lll}
\widetilde{X}_{i_{1}}( p+m-1 )\widetilde{X}_{i_{2}}
( p+m-2 )\cdots \widetilde{X}_{i_{p}}(m)
&  \mbox{if}\,\,\,1,\bar{1}\notin\{i_1,\cdots,i_p\}&\hbox{(1-a)}\\
\widetilde{X}_{i_{2}}(p+m-2 )\widetilde{X}_{i_{3}}
(p+m-3 )\cdots \widetilde{X}_{i_p}(m)
&  \mbox{if} \,\,\,1\in\{i_1,\cdots,i_p\}\,\mbox{and}\,\bar{1}\notin\{i_1,\cdots,i_p\}
&\hbox{(1-b)}\\ 
\widetilde{X}_{i_1}(p+m-1)\widetilde{X}_{i_2}
(p+m-2 )\cdots \widetilde{X}_{i_{p-1}}(m+1)
&  \mbox{if} \,\,\,1\notin\{i_1,\cdots,i_p\}\,\mbox{and}\,\bar{1}\in\{i_1,\cdots,i_p\}&\hbox{(1-c)}\\
\widetilde{X}_{i_2}( p+m-2 )\widetilde{X}_{i_3}
( p+m-3 )\cdots \widetilde{X}_{i_{p-1}}(m+1)
&  \mbox{if}\,\,\,1,\bar{1}\in\{i_1,\cdots,i_p\}&\hbox{(1-d)}
\end{array}\right.
\end{align*}

\begin{align}
&\widetilde{M_2}=
\left\{
\begin{array}{lll}
\widetilde{X}_{j_{1}}( q )\widetilde{X}_{j_{2}}
( q-1)\cdots \widetilde{X}_{j_q}(1)
&  \mbox{if}\,\,\,1,\bar{1}\notin\{j_1,\cdots,j_p\}
&\q\qq\qq\hbox{(2-a)}\\
\widetilde{X}_{j_{2}}(q-1 )\widetilde{X}_{j_{3}}
(q-2 )\cdots \widetilde{X}_{j_p}(1)
&  \mbox{if} \,\,\,1\in\{j_1,\cdots,j_p\}\,\mbox{and}\,\bar{1}\notin\{j_1,\cdots,j_p\}
&\q\qq\qq\hbox{(2-b)}\\ 
\widetilde{X}_{j_1}(q)\widetilde{X}_{j_2}
(q-1 )\cdots \widetilde{X}_{j_{p-1}}(2)
&  \mbox{if} \,\,\,1\notin\{j_1,\cdots,j_p\}\,\mbox{and}\,\bar{1}\in\{j_1,\cdots,j_p\}
&\q\qq\qq\hbox{(2-c)}\\
\widetilde{X}_{j_2}( q-1 )\widetilde{X}_{j_3}
( q-2 )\cdots \widetilde{X}_{j_{p-1}}(2)
&  \mbox{if}\,\,\,1,\bar{1}\in\{j_1,\cdots,j_p\}
&\q\qq\qq\hbox{(2-d)}
\end{array}\right.
\label{4-casesM2}
\end{align}
By these explicit forms, we obtain 
\begin{align}
\begin{array}{l}(l(\wtil{M_1}),h(\wtil{M_1}))=\begin{cases}
(p,m+(p-n+1)_+)&\hbox{(1-a)}\\
(p-1,m+(p-n)_+)&\hbox{(1-b)}\\
(p-1,m+1+(p-n)_+)&\hbox{(1-c)}\\
(p-2,m+1+(p-n-1)_+)&\hbox{(1-d)}
\end{cases}\\
(l(\wtil{M_2}),h(\wtil{M_2}))=\begin{cases}
(q,1+(q-n+1)_+)&\hbox{(2-a)}\\
(q-1,1+(q-n)_+)&\hbox{(2-b)}\\
(q-1,2+(q-n)_+)&\hbox{(2-c)}\\
(q-2,2+(q-n-1)_+)&\hbox{(2-d)}
\end{cases}
\end{array}
\end{align}
Under the assumption $h(M_1)\geq h(M_2)$, in almost all cases we have $h(\wtil{M_1})\geq h(\wtil{M_2})$
except the following cases:
\begin{enumerate}
\item ((1-a), (2-a)) with $h(M_1)=h(M_2)$,  $p-n<0$ and $q-n\geq 0$.
\item ((1-a),(2-c))  with $h(M_1)=h(M_2)$,  $p-n<0$.
\item ((1-a),(2-d))  with $h(M_1)=h(M_2)$,  $p-n<0$ and $q-n\leq 0$.
\item ((1-b),(2-a))  with $h(M_1)=h(M_2)$ and $q-n\geq 0$.
\item ((1-b),(2-c))  with $h(M_1)=h(M_2)$.
\item ((1-b),(2-d))  with $h(M_1)=h(M_2)$ and $q-n\leq 0$.
\item ((1-d),(2-a)) with $h(M_1)=h(M_2)$, $p-n>0$ and $q-n\geq0$.
\item ((1-d),(2-c)) with $h(M_1)=h(M_2)$, $p-n>0$.
\item ((1-d),(2-d)) with $h(M_1)=h(M_2)$, $p-n>0$ and $q-n\leq 0$.
\end{enumerate}
Considering except the cases (1)--(9) above, we have $h(\wtil{M_1})\geq h(\wtil{M_2})$. 
By the assumption $\vep_i({M_1}{M_2})=0$ for any $i\in \{1,2,\cd,n\}$, 
we have 
$\vep_i(\widetilde{M_1}\widetilde{M_2})=0$ for any $i\in \{2,3,\cd,n\}$. Thus, 
by the induction hypothesis we have $\vep_i(\widetilde{M_1})=0$ for $i\in\{2,3,\ldots,n\}$ and then 
\begin{align}
&\widetilde{M_1}=
\left\{
\begin{array}{lll}
&\widetilde{X}_{2}( p+m-1 )\widetilde{X}_{3}
( p+m-2 )\cdots \widetilde{X}_{p+1}(m)
&  \mbox{if}\,\,\,1,\bar{1}\notin\{i_1,\cdots,i_p\},\\
&\widetilde{X}_{2}(p+m-2 )\widetilde{X}_{3}
(p+m-3 )\cdots \widetilde{X}_{p}(m)
&  \mbox{if} \,\,\,1\in\{i_1,\cdots,i_p\}\,\mbox{and}\,\bar{1}\notin\{i_1,\cdots,i_p\},\\ 
&\widetilde{X}_{2}(p+m-1)\widetilde{X}_{3}
(p+m-2 )\cdots \widetilde{X}_{p}(m+1)
&  \mbox{if} \,\,\,1\notin\{i_1,\cdots,i_p\}\,\mbox{and}\,\bar{1}\in\{i_1,\cdots,i_p\},\\
&\widetilde{X}_{2}( p+m-2 )\widetilde{X}_{3}
( p+m-3 )\cdots \widetilde{X}_{p-1}(m+1)
&  \mbox{if}\,\,\,1,\bar{1}\in\{i_1,\cdots,i_p\},
\end{array}\right.
\label{4-2-cases}
\end{align}
where note that in the first case of \eqref{4-2-cases}, $p+1\leq 2n$ by the assumption \eqref{p<2n}.
In these four cases, by their conditions we obtain
the explicit form of $M_1$ as follows:
\begin{align}
&M_1=
\left\{
\begin{array}{lll}
&{X}_{2}( p+m-1 ){X}_{3}
( p+m-2 )\cdots {X}_{p+1}(m)
&  \mbox{if}\,\,\,1,\bar{1}\notin\{i_1,\cdots,i_p\},\\
&X_1(p+m-1){X}_{2}(p+m-2 )\cdots {X}_{p}(m)
&  \mbox{if} \,\,\,1\in\{i_1,\cdots,i_p\}\,\mbox{and}\,\bar{1}\notin\{i_1,\cdots,i_p\},\\ 
&{X}_{2}(p+m-1){X}_{3}
(p+m-2 )\cdots {X}_{p}(m+1)X_{\bar 1}(m)
&  \mbox{if} \,\,\,1\notin\{i_1,\cdots,i_p\}\,\mbox{and}\,\bar{1}\in\{i_1,\cdots,i_p\},\\
&X_1(p+m-1){X}_{2}( p+m-2 )\cdots {X}_{p-1}(m+1)X_{\bar 1}(m)
&  \mbox{if}\,\,\,1,\bar{1}\in\{i_1,\cdots,i_p\}.
\end{array}\right.
\label{4-3-cases}
\end{align}
Then by Lemma \ref{M has gaps}, we find that except the second case, 
$M_1\cdot M_2$ is not a highest weight monomial.
In the second case, we really get that
\begin{equation}
M_1=X_1(p+m-1){X}_{2}(m+p-2 )\cdots {X}_{p}(m)=
\begin{cases}
Y_p(m)&\hbox{ if }p\in[1,n],\\
Y_{2n-p}(m-n+p)&\hbox{ if }p\in[n+1,2n].
\end{cases}
\label{M1-hwv}
\end{equation}

Finally, in the cases (1)--(9) above, we get $h(\wtil{M_1})< h(\wtil{M_2})$. 
Then, by arguing similarly to the above cases, one has
\begin{align}
&\widetilde{M_2}=
\left\{
\begin{array}{lll}
&\widetilde{X}_{2}( q )\widetilde{X}_{3}
(q-1 )\cdots \widetilde{X}_{q+1}(1)
&  \mbox{if}\,\,\,1,\bar{1}\notin\{j_1,\cdots,j_q\},\\
&\widetilde{X}_{2}(q)\widetilde{X}_{3}
(q-1 )\cdots \widetilde{X}_{q}(2)
&  \mbox{if} \,\,\,1\notin\{j_1,\cdots,j_q\}\,\mbox{and}\,\bar{1}\in\{j_1,\cdots,j_q\},\\
&\widetilde{X}_{2}( q-1 )\widetilde{X}_{3}
(q-2 )\cdots \widetilde{X}_{q-1}(2)
&  \mbox{if}\,\,\,1,\bar{1}\in\{j_1,\cdots,j_q\},
\end{array}\right.
\end{align}
where note that in \eqref{4-casesM2}, the second case (2-b) does not occur since 
in the cases (1)--(9) above the case (2-b) does not appear. Since in the cases 
(1)--(9) we have $h(M_1)=h(M_2)$, by considering similarly to 
the one for \eqref{4-3-cases} and applying Lemma \ref{M has gaps} to them, 
$M_1\cdot M_2$ cannot be a highest weight vector, 
which completes the proof of  Theorem \ref{induction on n}.
\end{proof}

\begin{corollary}\label{l<=m} 
For $u\in \mathcal{M}(Y_{p}(m ))$,  $v\in  \mathcal{M}(Y_{q}( l) )$ 
with $p,q\in[1,n]$ and $m\geq l$, if $u\cdot v$ is the highest weight vector in 
$\mathcal{M}(Y_{p}( m ))\cdot \mathcal{M}(Y_{q}( l) )$, 
then $u=Y_p(m)$.
\end{corollary}

 \begin{proof} 
 For $p,q\in[1,n]$, we have $\mathcal{M}(Y_{p}( m ))=M_p(m)$ and 
 $\mathcal{M}(Y_{q}( l) )=M_q(l)$. Under the condition of $p\in[1,n]$, we obtain 
 $h(u)=m+(p-n)_+=m$ for any $u\in M_p(m)$, $h(v)=1+(q-n)_+=1$ for any $v\in M_q(1)$
  and then by Theorem \ref{induction on n}, we have 
 that $u=Y_p(m)$ if $m\geq1$ and $u\cdot v\in M_p(m)\cdot M_q(1)$ is a highest weight vector. 
 The crystal structure of $\mathcal{M}(Y_{p}( m ))\cdot \mathcal{M}(Y_{q}( l) )$ depend only on the difference $m-l$ by the definitions in Section.3, that, is, 
 for a monomial $M=Y_{i_1}(m_1)^{\epsilon_1}\cdots Y_{i_p}(m_p)^{\epsilon_p}$
 and $M'=Y_{i_1}(m_1+a)^{\epsilon_1}\cdots Y_{i_p}(m_p+a)^{\epsilon_p}$ 
 ($a\in\bbZ,\,\,\epsilon_i=\pm$), we obtain 
 $wt(M)=wt(M')$, $\varepsilon_i(M)=\varepsilon_i(M')$, 
 $\varphi_i(M)=\varphi_i(M')$,  
 $\widetilde f_i(M)=A_i(k)^{-1}M$ iff $\widetilde f_i(M')=A_i(k+a)^{-1}M'$ and 
 $\widetilde e_i(M)=A_i(j)M$ iff $\widetilde e_i(M')=A_i(j+a)M'$.
\end{proof}

\begin{corollary}\label{l=m} For $p,q\in[1,n]$ and $m\in\bbZ$, 
the crystal $\mathcal{M}(Y_{p}( m ))\cdot \mathcal{M}(Y_{q}( m) )$ 
is connected as a crystal graph.
 \end{corollary}

\begin{proof}
Assume for $u\in \mathcal{M}(Y_{p}( m ))$ and $v\in  \mathcal{M}(Y_{q}(m) )$, $u\cdot v$ is the highest weight vector. Then by Corollary \ref{l<=m},  we find that $u=Y_{p}( m )$ and $v=Y_{q}(m)$, which means $Y_{p}( m )\cdot Y_{q}( m)$ is the unique highest weight vector in $\mathcal{M}(Y_{p}( m ))\cdot \mathcal{M}(Y_{q}(m) )$.
\end{proof}

\begin{remark}
Our main aim of this section is Corollary \ref{l<=m} since we would like to know how 
the product $\mathcal{M}(Y_{p}( m ))\cdot \mathcal{M}(Y_{q}(1) )$, with 
$p,q\in[1,n]$ and $m\geq1$, is decomposed into connected components.
The reason why in Theorem \ref{induction on n} we consider the range of $p,q$ as  not 
$[1,n]$ but $[1,2n]$ is as follows: in the  proof of the theorem we used the induction on the rank of 
the Lie algebra of type $C$. In the rank $n$ case, the length of a monomial 
$M=X_{i_1}(m+p-1)\cd X_{i_p}(m)\in M_p(m)$ with $p\in[1,2n]$ is at most $2n$ since 
$1\leq i_1<\cd<i_p\leq \ovl{1}$. Thus, we find that if $l(M)=2n$, then $1,\bar 1\in\{i_1\cd,i_p\}$ and 
if $l(M)=2n-1$, then $1$ or $\bar 1\in\{i_1\cd,i_p\}$. Therefore, we have  
$l(\io_1(M))\leq 2n-2$, which means that the induction will proceeds. 
But, if we restrict the range of $p$ to $[1,n]$, the length of a monomial 
$M=X_{i_1}(m+p-1)\cd X_{i_p}(m)\in M_p(m)$ is at most $n$. Thus, the claim of the theorem 
holds under the condition $l(M)\leq n$. And if $l(M)=n$
 and unless $1,\bar1\in\{1,2,\cd,\bar1\}$,
then $l(\io_1(M))=n$. In this case, we cannot proceed the induction since the length of 
monomials applied $\io_1$ must be less than or equal to $n-1$.

\end{remark}

\section{Decomposition of Monomial Product}
We have seen that any monomial product $\mathcal{M}(Y_{p}(m))\cdot\mathcal{M}(Y_{q}(1))$ holds
the crystal structure and then we consider their decomposition into connected components.
\subsection{Decomposition of Monomial Product}

\begin{lemma}
For $m\geq1$, any connected component in $ \mathcal{M}(Y_{p}(  m) ) 
\cdot\mathcal{M}(Y_{q}( 1 ))$ appears in 
$ \mathcal{M}(Y_{p}( m ) )\otimes  \mathcal{M}(Y_{q}( 1 ))$ and   
$ \mathcal{M}(Y_{p}(  m) ) \cdot\mathcal{M}(Y_{q}( 1 ))$ is a subcrystal of  
$ \mathcal{M}(Y_{p}( m ) )\otimes  \mathcal{M}(Y_{q}( 1 ))$. 
More precisely, for $M_1\in\mathcal{M}(Y_{p}(  m) ), 
M_2 \in\mathcal{M}(Y_{q}( 1 ))$  if $M_1\cdot M_2$ 
is a highest weight vector then $M_1\otimes  M_2$ 
is a highest weight vector in 
$ \mathcal{M}(Y_{p}( m ) )\otimes  \mathcal{M}(Y_{q}( 1 ))$.
\end{lemma}

 \begin{proof}
 For $M_{1}\in\mathcal{M}(Y_{p}( m ))$ and  $M_{2}\in\mathcal{M}(Y_{q}( 1) )$ assume that $M_{1}\cdot{M}_{2}$ is a highest weight vector. By Theorem \ref{induction on n}, we know that $M_1=Y_p(m)$ is the 
 unique highest weight vector in $\cM(Y_p(m))$. Next let us see the condition on $M_2$.
 
  (1) In the case $i\ne p$, by (\ref{epsilon def}), we  have
\begin{align} \label{for tensor}
\nonumber
 \varepsilon_{i}(M_1 \otimes M_2)&=\max\,(\varepsilon_{i}(M_1 ),\varepsilon_{i}(M_2 )-\langle h_{i},wt(M_1)\rangle)
=\max\,(\varepsilon_{i}(M_1 ),\varepsilon_{i}(M_2 )-\langle h_{i},\Lambda_p\rangle)\nonumber\\
&=\max\,(0,\varepsilon_{i}(M_2 )-0)
=\max\,(0,\varepsilon_{i}(M_2 ))=\varepsilon_{i}(M_2),
 \end{align}
  and by the definition of $\varepsilon_{i}$ in \eqref{varepsilon} for $i\ne p$, we have
\begin{align} \label{for M1.M2}
\varepsilon_{i}(M_1 \cdot M_2)=\varepsilon_{i}(M_2).
 \end{align}
Therefore, by (\ref{for tensor}) and (\ref{for M1.M2}), we obtain $\varepsilon_{i}(M_1 \otimes M_2)=\varepsilon_{i}(M_1 \cdot M_2)$.  

(2) In the case $i=p$, we set 
$M_2=\prod\limits_{\substack{i\in I \\ n\in \mathbb{ Z}}}Y_{i}(n)^{y_i(n)}$.  
It follows form (\ref{varepsilon}) that 
 \begin{align} \label{varepsilon p}
\varepsilon_{p}( M_2)= \max_{n\in  \mathbb{ Z}}\,\bigg\{-\sum\limits_{k> n} {y_{p}(k)}\bigg\}.
 \end{align}
By (\ref{varepsilon p}) we can calculate 
 \begin{align} \label{varepsilon p for tensor}
\nonumber
 \varepsilon_{p}(M_1 \otimes M_2)=&\max\,( \varepsilon_{p}(M_1 ), \varepsilon_{p}(M_2)-\langle h_{p},wt(M_1)\rangle)\\
=&\max\,(0,\max_{n\in  \mathbb{ Z}}\,\bigg\{-1-\sum\limits_{k> n} {y_{p}(k)}\bigg\}),
 \end{align}
and
 \begin{align}\label{varepsilon p for M1.M2}
\varepsilon_{p}(M_1 \cdot M_2)=\max( \max_{\substack{n\in  \mathbb{ Z}\\m\le n}}\bigg\{-\sum\limits_{k> n} {y_{p}(k)}\bigg\},\max_{\substack{n\in  \mathbb{ Z}\\m> n}}\bigg\{-1-\sum\limits_{k> n} {y_{p}(k)}\bigg\}).
 \end{align}
By (\ref{varepsilon p for M1.M2}) and (\ref{varepsilon p for tensor}), we get $\varepsilon_{p}(M_1 \cdot M_2)\geq  \varepsilon_{p}(M_1 \otimes M_2)$ since $-\sum\limits_{k> n} {y_{p}(k)=0}$ for $n\gg0$. 
Then we find that if $\varepsilon_{i}(M_1 \cdot M_2)=0$ for any $i$, 
then $ \varepsilon_{i}(M_1 \otimes M_2)=0$ for any $i$, 
which means that if $M_1\cdot M_2$ is a highest weight vector then $M_1 \otimes M_2$ 
is also a highest weight vector.
\end{proof}

In Lemma \ref{decomposition}, $ \mathcal{M}(Y_{p}( m ) )\otimes  \mathcal{M}(Y_{q}( 1 ))$ is decomposed into the direct sum of $B(\Lambda_a+\Lambda_c)$'s, where $(a,c)$ satisfies the condition $( \divideontimes)$ in Lemma \ref{decomposition}.

Here, we will classify all the monomials  $M_{2}\in\mathcal{M}(Y_{q}( 1) )$ such that $wt(M_2)=\Lambda_a+\Lambda_c-\Lambda_p$ and $Y_{p}( m )\cdot M_2$ is the highest weight vector.

\begin{lemma} \label{gaps in M2}
For $M_{1}\in\mathcal{M}(Y_{p}( m ))$,  $M_{2}\in\mathcal{M}(Y_{q}( 1) )$ with $m\geq 1$, $p,q\in \{1,\ldots ,n\}$, suppose $M_1=Y_{P}(m)$ and $M_2$ is in the forms:
 \begin{enumerate}
\item If \,\,$0\le a\le c\le p$, then
$M_2= [X_{1}\cdots X_{a}]L_1 [X_{\bar{p}}\cdots X_{\overline{c+1}}] L_2.$
\item If \,\,$0\le a\le p<c$, then
$M_2=[ X_{1}\cdots X_{a}]L_1[X_{p+1}\cdots X_c]L_2 $.
\item If \,\,$0\le a\le c$, $a<p$, then 
$M_2=[X_{1}\cdots X_{c}]L_1 L_2[X_{\bar{p}}\cdots X_{\overline{a+1}}]$.
\end{enumerate}
Here $L_1,L_2$ are products of $X_i$'s and $X_{\bar{i}}$'s 
respectively such that $wt(L_1)+wt(L_2)=0$ and $L_1,L_2$ 
include subsequences which are consecutive 
but separated from other sequences, such as 
$[X_l\cdots X_k]\subset L_1, [X_{\bar{k}}\cdots X_{\bar{l}}]\subset L_2$ with $l\leq k$.
Then $M_{1}\cdot{M}_{2}$ is not a highest weight vector.
\end{lemma}

 \begin{proof}
Assume that   $M_{1}\cdot{M}_{2}$ is a highest weight vector. 
We know that $wt(M_2)=\Lambda_a+\Lambda_c-\Lambda_p$, where
\begin{align}
\nonumber
\Lambda_a+\Lambda_c-\Lambda_p=
\left\{\begin{array}{rcl}  
&\epsilon_1+\cdots+\epsilon_a-\epsilon_{c+1}-\cdots-\epsilon_p & \mbox{if} \,\, a\le c\le p,\\
&\epsilon_1+\cdots+\epsilon_a+\epsilon_{p+1}+\cdots+\epsilon_c & \mbox{if}\,\, a\le p<c.
\end{array}\right.
\end{align}  
By Proposition \ref{general form by X}, 
for some $\xi, \,\,\eta,\,\,\zeta$ with $\xi>\eta>\zeta$
we have:
 \begin{enumerate}
\item If \,\,$0\le a\le c\le p$,
  \begin{align}\label{when c<p}
\nonumber
M_{2}= &[X_{1}( q )\cdots X_{a}( q-a+1 )]\cdots[X_l(\xi )\cdots X_k(\xi+l-k )]\cdots \\
\nonumber
&[X_{\bar{p}}(\eta) \cdots X_{\overline{c+1}}(\eta+c-p+1)]\cdots \\
\nonumber
&[X_{\bar{k}}(\zeta)\cdots X_{\bar{l}}(\zeta+l-k)]\cdots.\\
\nonumber
M_1\cdot M_{2}= &Y_{p}(m)\cdot Y_{a}(q-a+1)\cdots\frac{Y_{k}(\xi+l-k)}{Y_{l-1}(\xi+1)}\cdots\frac{Y_{c}(\eta-p+n+1)}{Y_{p}(\eta-p+n+1)}\\
&\cdots \frac{Y_{l-1}(\zeta-k+n+1)}{Y_{k}(\zeta-k+n+1)}\cdots,
\end{align} 
where dots$\cdots$ means the non-consecutive separated sequences from the consecutive sequences.
By the highest weight condition, comparing $Y_{k}$ to see $\vep_k=0$ we get
\begin{align}\label{Xk}
0\geq(\zeta-k+n+1)-(\xi+l-k)=\zeta-\xi+n+1-l,
\end{align}
and comparing $Y_{l-1}$ we get
\begin{align}\label{Xl}
0\geq(\xi+1)-(\zeta-k+n+1)=\xi-\zeta+k-n.
\end{align}
By (\ref{Xk}) and (\ref{Xl}), we have $l-1\geq  k$, which contradicts the assumption $l\le k$. Then this case cannot happen and then $M_1\cdot M_2$ is not a highest weight vector.
\item  If \,\,$0\le a\le p<c$,
\begin{align}
\nonumber
M_{2}= &[X_{1}( q )\cdots X_{a}( q-a+1 )]\cdots[X_l(\xi)\cdots X_k(\xi+l-k)]\cdots\\
\nonumber
&[ X_{p+1}(\eta)\cdots X_c(\eta+p-c+1  )]\cdots\\
\nonumber
&[X_{\bar{k}}(\zeta )\cdots X_{\bar{l}}(\zeta+l-k )]\cdots.\\
\nonumber
M_1\cdot M_{2}= &Y_{p}(m)\cdot Y_{a}(q-a+1)\cdots \frac{Y_{k}(\xi+l-k)}{Y_{l-1}(\xi+1)}\cdots \frac{Y_{c}(\eta+p-c+1 )}{Y_{p}(\eta)}\\
\nonumber
&\cdots \frac{Y_{l-1}(\zeta-k+n+1)}{Y_{k}(\zeta-k+n+1)}\cdots,
\end{align} 
where dots$\cdots$ means the same  as in (\ref{when c<p}). 
By the highest weight condition, comparing $Y_{k}$ to see $\vep_k=0$ we get
\begin{align}\label{Xk2}
0\geq (\zeta-k+n+1)-(\xi+l-k)=\zeta-\xi+n+1-l,
\end{align} 
and comparing $Y_{l-1}$ to see $\vep_{k-1}=0$ we get
\begin{align}\label{Xl2}
0\geq(\xi+1)-(\zeta-k+n+1)=\xi-\zeta+k-n.
\end{align}
By (\ref{Xk2}) and (\ref{Xl2}), we have $l-1\geq k$, which contradicts the assumption $l\le k$. 
Then this case cannot happen and then $M_1\cdot M_2$ is not a highest weight vector.

\item  If \,\,$0\le a\le c$,
\begin{align}
\nonumber
M_{2}= &[X_{1}( q )\cdots X_{c}( q-c+1 )]\cdots[X_l(\xi)\cdots X_k(\xi+l-k)]\cdots\\
\nonumber
&[X_{\bar{k}}(\eta )\cdots X_{\bar{l}}(\eta+l-k )]\cdots
[ X_{\bar p}(\zeta)\cdots X_{\overline{a+1}}(\zeta+p-a+1  )]\cdots\\
\nonumber
M_1\cdot M_{2}= &Y_{p}(m)\cdot Y_{c}(q-c+1)\cdots \frac{Y_{k}(\xi+l-k)}{Y_{l-1}(\xi+1)}\cdots \\
\nonumber
& \frac{Y_{l-1}(\zeta-k+n+1)}{Y_{k}(\zeta-k+n+1)}\cdots
\frac{Y_{c}(\eta+p-c+1 )}{Y_{p}(\eta)}\cdots
\nonumber
\end{align} 
where dots$\cdots$ means the same  as in (\ref{when c<p}). 
By the highest weight condition, comparing $Y_{k}$ we get
\begin{align}\label{Xk3}
0\geq (\eta-k+n+1)-(\xi+l-k)=\eta-\xi+n+1-l,
\end{align} 
and comparing $Y_{l-1}$ we get
\begin{align}\label{Xl3}
0\geq(\xi+1)-(\eta-k+n+1)=\xi-\eta+k-n.
\end{align}
By (\ref{Xk3}) and (\ref{Xl3}), we have $l-1\geq k$, which contradicts the assumption $l\le k$. 
Then this case cannot happen and then $M_1\cdot M_2$ is not a highest weight vector
\end{enumerate}
Thus, we find that $M_{1}\cdot{M}_{2}$ is not a highest weight vector.
\end{proof}

\begin{theorem}  \label{forms thm}
Let $(a,c)$ be a pair satisfying the condition  $( \divideontimes)$ in Lemma \ref{decomposition}. For $M_{1}\in\mathcal{M}(Y_{p}( m ))$,  $M_{2}\in\mathcal{M}(Y_{q}( 1) )$ with $m\geq 1$ and $p,q\in \{1,\ldots ,n\}$,   assume that $M_{1}\cdot{M}_{2}$ is a highest weight vector with a weight $\Lambda_a+\Lambda_c$ as in Lemma \ref{decomposition}. 
Then we obtain $M_1=Y_p(m)$ 
and  as for $M_2$ we have:
 \begin{enumerate}
\item If\,\, $0<a\le c\le p$,  $M_{2}$ is in one of the following forms (I)--(III):
\begin{itemize}
 \item[(I)] 
 \begin{align}
\nonumber
 M_{2}=&[X_{1}( q )\cdots X_{a}( q-a+1 )]
 [X_{a+1}( q-a )\cdots X_{d}(q-d+1)][X_{\bar{p}}(q-d)\cdots\\
\nonumber
& X_{\overline{c+1}}(q-p-d+c+1) ][X_{\bar{d}}(q-p-d+c)\cdots X_{\overline{a+1}}(q-p-2d+c+a+1)],
                \end{align}
where $d=\frac{q+a+c-p}{2}$ and with $(a+p=c+q,\,c<p,\, m\geq \frac{q-p-a-c}{2}+n+1)$ or 
 $(a+q=c+p,\, a<q, \,m\geq n-a+1)$ or  $(a=q, c=p)$.
 \item[({II})] 
\begin{align}
\nonumber
  M_{2}=& [X_{1}( q )\cdots X_{a}( q-a+1 )][X_{p+1}(q-a)\cdots X_{f}(q+p-a-f+1)]\\
\nonumber
& [X_{\bar{f}}(q+p-a-f)\cdots X_{\overline{p+1}}(q+2p-a-2f+1)][X_{\bar{p}}(q+2p-a-2f)\cdots\\
\nonumber
& X_{\overline{c+1}}(q+p-a-2f+c+1)].
 \end{align} 
 where {$f=\frac{q+p+c-a}{2}$} and
with  $(a<c{\leq }p<d=n, \, m\geq q-a+1)$. 
 \item [(III)]
\begin{align}
\nonumber
  M_{2}=&[X_{1}( q )\cdots X_{a}( q-a+1 )][ X_{e}( q-a )\cdots X_{c}(q-a-c+e)][X_{\bar{p}}(q-a-c+e-1)\cdots\\
\nonumber
& X_{\overline{c+1}}(q-p-a+e) X_{\bar{c}}(q-p-a+e-1)\cdots 
X_{\bar{e}}(q-p+2e-a-1)],
                \end{align}
where $e=\frac{p-q+a+c+2}{2}$, and with $(c\leq p,\,a+p<c+q,\,\,m\geq n+\frac{q-p-a-c+2}{2})$.
\end{itemize}
\item If\,\,   $0<a\le p<c$,   $M_{2}$ is in one of the following forms (IV)--(VII):
\begin{itemize}
 \item[(IV)] \begin{align}
\nonumber
 M_{2}=&[X_{1}( q )\cdots X_{a}( q-a+1 )][X_{p+1}( q-a )\cdots X_{c}(q+p-a-c+1)X_{c+1}(q+p-a-c)\\
\nonumber
&\cdots X_{f}(q+p-a-f+1)][X_{\bar{f}}(q+p-a-f)\cdots X_{\overline{c+1}}(q+p+c-a-2f+1)],
 \end{align}
where $f=\frac{p+q+c-a}{2}$, and with   $(c<f=n, \, m\geq q-a+1)$ 
$(a<p<c=f, \,\,m\geq q-a+1)$ { or $(a=p,\, c=q)$}.
\item[(V)] \begin{align}
\nonumber
 M_{2}=&[X_{1}( q )\cdots X_{a}( q-a+1 )X_{a+1}( q-a )\cdots X_{h}(q-h+1)][X_{p+1}(q-h)\\
\nonumber
&\cdots X_{c}(q+p-h-c+1)][X_{\bar{h}}(q+p-h-c)\cdots X_{\overline{a+1}}(q+p+a-2h-c+1)],
 \end{align}
where $h=\frac{q+a+p-c}{2}\le p<c$, and with $(a< p,\,p+c=q+a,\,m\geq n-a+1)$ or 
$(a+c=p+q,\,p+c>a+q, m\geq q-a+1)$ or $(a=p,\, c=q)$.
 \item[(VI)] \begin{align}
\nonumber
 M_{2}=&[X_{1}( q )\cdots X_{a}( q-a+1 )]
 [X_{e}( q-a )\cdots X_{p}(q+e-a-p)][X_{p+1}(q+e-a-p-1)\\
\nonumber
&\cdots X_{c}(q+e-a-c)][X_{\bar{p}}(q+e-a-c-1)\cdots 
X_{\bar{e}}(q+2e-a-c-p-1)],
 \end{align}
where $e=\frac{p+a+c-q+2}{2}$, and with  $(a+q\leq p+c,\,a+c<p+q,\,m\geq \frac{q-a-c-p+2}{2}+n)$ or
$(a+q<p+c,\,a+c=p+q,\,m\geq q-a+1)$ or $(a=p,c=q)$.
 \item [(VII)]
 \begin{align}
\nonumber
M_2=&[X_{1}( q )\cdots X_{d}( q-d+1 )]
[X_{\bar{d}}(q-d)\cdots X_{\overline{c+1}}(q-2d+c+1)]\\
\nonumber
&[X_{\bar{p}}(q-2d+c)\cdots X_{\overline{a+1}}(q-p-2d+c+a+1)]
\end{align}
where $d=\frac{q-p+a+c}{2}$, with $(a<p<c,\, a+q=c+p,\,\,m\geq \frac{q-p-a-c}{2}+n+1)$
or $(a=p,c=q)$.
\end{itemize}
\item If \,\, $0=a< c\le p$,  $M_{2}$ is in one of the following forms (VIII)--(IX):
\begin{itemize}
 \item[(VIII)]
 \begin{align}
\nonumber
 M_{2}=&[X_{p+1}( q )\cdots X_{g}( q+p-g+1 )][X_{\bar{g}}( q+p-g )\cdots X_{\overline{p+1}}(q+2p-2g+1)]\\
\nonumber
&[X_{\bar{p}}(q+2p-2g)\cdots X_{\overline{c+1}}( q+p+c-2g+1)],
 \end{align} 
where $g=\frac{p+q+c}{2}$, and with $(m\geq q+1, \,g=n,\,p<q+c)$ or $(p=q+c,\,m\geq n-c+1)$.
\item[({IX})] \begin{align}
\nonumber
 M_{2}=&[X_{b}( q )\cdots X_{c}( q+b-c )]
 [X_{\bar{p}}( q+b-c-1 )\cdots X_{\overline{c+1}}(q+b-p)]\\
\nonumber
&[X_{\bar{c}}(q+b-p-1)\cdots X_{\bar{b}}( q+2b-c-p-1)],
 \end{align}
where $b=\frac{p+c-q+2}{2}$, and with $(m\geq\frac{q-c-p+2}{2}+n)$.\end{itemize}
\item If   $a=0\le p< c$, $M_{2}$ is in one of the following forms (X)--(XI):
\begin{itemize}
\item[({X})] \begin{align}
\nonumber
 M_{2}=&[X_{b}( q )\cdots X_{p}( q+b-p )]
 [X_{p+1}( q+b-p-1  )\cdots X_{c}( q+b-c)]\\
\nonumber
&[X_{\bar{p}}( q+b-c-1)\cdots X_{\bar{b}}( q+2b-c-p-1)],
 \end{align}
where   $b=\frac{p+c-q+2}{2}$, and with   $(c<p+q,\,m\geq n+\frac{q-c-p+2}{2})$ or 
$(c=p+q,\,m\geq q+1)$.
\item[({XI})] \begin{align}
\nonumber
 M_{2}=&[X_{p+1}( q )\cdots X_{c}( q+p-c +1)]
 [X_{c+1}( q+p-c  )\cdots X_{g}( q+p-g +1)]\\
\nonumber
&[X_{\bar{g}}(q+p-g)\cdots X_{\overline{c+1}}( q+p+c-2g+1)],
 \end{align}
where $g=\frac{p+q+c}{2}$, with $m\geq q+1$. Note that if $c<p+q$, we have $g=n$ and 
$m\geq q+1(=\frac{q-p-a-c+2}{2}+n)$, and if $c=p+q$, we get $m\geq q+1(=q-a+1)$.
\end{itemize}
\item If  $a=c=0$,  $M_{2}$ is in the following form:
\begin{itemize}
\item[({XII})] \begin{align}
\nonumber
 M_{2}=[X_{\bar{p}}(q)\cdots X_{\bar{1}}(1)],
 \end{align}
with $p=q$ and $m\geq n+1$.
\end{itemize}
\end{enumerate}
\end{theorem}

 \begin{proof}
 First, we assume that $M_{1}\cdot{M}_{2}$ is a highest weight vector and then
 we know that by Theorem \ref{induction on n}, $M_{1}=Y_{p}(m)$ and 
 the form of $M_2$ is in one of {(I)}--{(XII)} in the theorem by 
 Lemma \ref{gaps in M2} since if there are sequences $X_l\cd X_k$ and $X_{\bar l}\cd X_{\bar l}$ in
 $M_2$, { then one of then must has no gap between its neighbor sequences}.

  By Lemma \ref{decomposition}, we know that $wt(M_2)=\Lambda_a+\Lambda_c-\Lambda_p$,  where
\begin{align}
\nonumber
\Lambda_a+\Lambda_c-\Lambda_p=\left\{\begin{array}{rcl}  &\epsilon_1+\cdots+\epsilon_a-\epsilon_{c+1}-\cdots-\epsilon_p & \mbox{if} \,\, a\le c\le p,\\
&\epsilon_1+\cdots+\epsilon_a+\epsilon_{p+1}+\cdots+\epsilon_c & \mbox{if}\,\, a\le p<c.
\end{array}\right.
\end{align}  
Then $M_2$ has the factors 
\begin{align}
\nonumber
\left\{\begin{array}{ll}  X_1\cdots X_a X_{\bar{p}}\cdots X_{\overline{c+1}}& \mbox{if} \,\, a\le c\le p,\\
X_1\cdots X_a X_{p+1}\cdots X_c & \mbox{if}\,\, a\le p<c.
\end{array}\right.
\end{align}  
By Proposition \ref{general form by X} and Lemma \ref{gaps in M2}, we have
 \begin{enumerate}
\item If $0<a\le c\le p$,  $M_1\cdot M_{2}$ is in one of the following forms:
\begin{itemize}
\item [({I})] $a\leq d=\frac{q+a+c-p}{2}\leq c\leq p$ and 
\begin{align}\label{case I}
M_1\cdot M_{2}=Y_{p}(m)\cdot Y_{d}(q-d+1)\cdot  \frac{Y_{c}(q-p-d+n+1)}{Y_{p}(q-p-d+n+1)}\cdot \frac{Y_{a}(q-p-2d+c+n+1)}{Y_{d}(q-p-2d+c+n+1)}.
                \end{align}
\begin{enumerate}
\item[(i)] If $ a<d<c<p$, then 
by the highest weight condition, comparing $Y_{d}$ we get
\begin{align}
\nonumber
0\geq(q-p-2d+c+n+1)-(q-d+1)=n-\frac{q+p+a-c}{2}.
\nonumber
 \end{align}
By the condition $( \divideontimes)$ in Lemma \ref{decomposition}, 
 we obtain $a=c=d, p=q=n$, which means this case cannot occur since we assume { $c> d$}. 
\item[(ii)] If $a\leq d<c=p$, then in (\ref{case I}) 
we find $\frac{Y_{c}(q-p-d+n+1)}{Y_{p}(q-p-d+n+1)}=1$. 
If $a=d$, we have $a=q$ and $c=p$, which contradicts the condition $d< c$.
Then, we assume $a< d$.
By the highest weight condition, comparing $Y_{d}$ we get
\begin{align}
\nonumber
 0\geq(q-p-2d+c+n+1)-(q-d+1)=n-d.
\end{align}
By the condition $( \divideontimes)$ in Lemma \ref{decomposition},  
$n\geq d$ then the only possibility is $d=n$. Since  $c=p$ we obtain $\frac{q+a}{2}\geq n$ then 
we have $a=c=p=q=n$ which contradict  our setting. Then this case cannot happen.
\item[(iii)] If $a=d\leq  c<p$, we get $a+p=c+q$. By the highest weight conditions, comparing $Y_p$, we also 
obtain 
\[
m\geq q-p-d+n+1=\frac{q-p-a-c}{2}+n+1.
\]
\item[(iv)] If $a< d=c=p$, 
in (\ref{case I}) we find  $\frac{Y_{c}(q-p-d+n+1)}{Y_{p}(q-p-d+n+1)}=1$. By the highest weight condition, comparing $Y_p(=Y_{d})$ we get $m\geq q-p-2d+c+n+1=n-a+1$ or $q-d+1\geq q-p-2d+c+n+1$.
Note that suppose $a=q$, then we have $p=c=d=q=a$, which contradicts 
$a<p$. Then we have $a<q$, then we obtain $m\geq n-a+1$ or $d=n$. If $d=n$ then, $c=p=d=n=a=q$, which is
a contradiction. Thus, we have $m\geq n-a+1$, $a+q=c+p, {\,c=p}$ and $a<q$. 
\item[(v)] If $a\leq d=c<p$, in (\ref{case I}) 
we find $  \frac{Y_{c}(q-p-d+n+1)}{Y_{p}(q-p-d+n+1)}\cdot \frac{Y_{a}(q-p-2d+c+n+1)}{Y_{d}(q-p-2d+c+n+1)}=\frac{Y_{a}(q-p-2d+c+n+1)}{Y_{p}(q-p-d+n+1)}$.
By the highest weight condition, comparing $Y_{p}$ we get $m\geq q-p-d+n+1=n-a+1$  
with $a+q=c+p$ and $c<p$. Here note that if $c<p$ and $a+p\leq c+q$, then we get $a<q$.
{Thus, we have $m\geq n-a+1$, $a+q=c+p$, $c<p$ and $a<q$}.
\item[(vi)] If $a=d=c=p$, we get $M_1\cdot M_{2}=Y_{p}(m)\cdot Y_{p}(1)$, which is always the highest weight vector.
\end{enumerate}

\item [({II})] We have $a\leq c\leq p\leq f=\frac{p+q+c-a}{2}$.
If $f=p$, then this case occurs in the previous case (I) with $a=d$. Then, we assume that 
 $p<f$. 
\begin{align}\label{case II}
M_1\cdot M_{2}=&Y_{p}(m)\cdot Y_{a}(q-a+1)\cdot \frac{Y_{f}(q+p-a-f+1)}{Y_{p}(q-a+1)}\cdot \frac{Y_{c}(q+p-a-2f+n+1)}{Y_{f}(q+p-a-2f+n+1)}.
\end{align}  
\begin{itemize}
\item [(i)]If $a\leq c<p<f$, by the highest weight condition, comparing $Y_{f}$ we get
\begin{align}
\nonumber
0\geq(q+p-a-2f+n+1)-(q+p-a-f+1)=n-f.
\end{align}
By the condition $( \divideontimes)$ in Lemma \ref{decomposition},  $f\le n$, then the only possibility is $f=n$. 
{ Thus, we obtain $m\geq q-a+1$ and $f=n$.}
\item [(ii)] If $a<c=p<f$, by the highest weight condition, comparing $Y_{f}$ we get $f=n$ 
with the similar way to the cases above.  And comparing $Y_{p}$, we get $m\geq q-a+1$ or
\begin{align}\label{compare Yp}
0\geq(q-a+1)-(q+p-a-2f+n+1)=p+q-a-n.
\end{align}
Since $f=n$, we get $q=a=n$ and then (\ref{compare Yp}) means $f=p$, which contradicts
our assumption $p<f$, then (\ref{compare Yp}) cannot happen. 
Then we have $f=n$ and $m\geq q-a+1(=\frac{q-p-a-c+2}{2}+n)$.
\item [(iii)] If $a=c=p<f$, by the highest weight condition, comparing $Y_f$, we get 
$q+p-a-f+1\geq q+p-a-2f+n+1$. Then we have $f=n$, which means $q=c=n$ by the condition in Lemma \ref{decomposition}. Thus, we have $c=n=f$ which derives a contradiction.

\end{itemize}

 \item [({III})] We  have $a\leq e-1\leq c\leq p$. If $e-1=c$, then this case occurs in 
 the case (I) with $a=d$.  Then we assume $e-1<c$. 
 \begin{align}
    M_1\cdot M_{2}=Y_{p}(m)\cdot Y_{a}(q-a+1)\cdot  \frac{Y_{c}(q-a-c+e)}{Y_{e-1}(q-a+1)}
    \cdot\frac{Y_{e-1}(q-p-a-c+e+n)}{Y_{p}(q-p-a-c+e+n)}.
 \end{align}
\begin{itemize}
\item[(i)] If $a<e-1<c< p$,
by the highest weight condition, comparing $Y_{e-1}$ we get 
\begin{align}\label{satisfy result}
0\geq(q-a+1)-(q-p-a-c+e+n)=\frac{p+q+c-a}{2}- n,
\end{align}
and comparing $Y_{p}$, we find  
\begin{align}
\nonumber
m\geq q-p-a-c+e+n=n+\frac{q-p-a-c+2}{2}.
\end{align}
By the condition $( \divideontimes)$ in Lemma \ref{decomposition}, $\frac{p+q+c-a}{2}\le n$ is always satisfied.
 Then we have $m\geq n+\frac{q-p-a-c+2}{2}$.
\item[(ii)] If $a=e-1<c<p$, we only need to see $Y_p$. Then by the highest weight condition,
we get 
$m\geq q-p-a-c++n+1\Leftrightarrow m\geq \frac{q-p-a-c}{2}+n+1$.
\item[(iii)] If $a\leq e-1<c=p$, by the highest weight condition, comparing $Y_{e-1}$ we get (\ref{satisfy result}) similarly  to the previous case. And comparing $Y_{p}$, we find 
 $m\geq n+\frac{q-p-a-c+2}{2}$ or 
\begin{align}\label{compare e-1}
0\geq(q-p-a-c+e+n)-(q-a-c+e)= n-p.
\end{align}
By the condition $( \divideontimes)$ in Lemma \ref{decomposition}, $n\leq p$ then
the only possibility is $p=n$ under (\ref{compare e-1}). 
And since we have (\ref{satisfy result}),  
we obtain $p=e-1$ and that contradict  our assumption that $p>e-1$,  (\ref{compare e-1}) 
cannot happen. Then we have  $m\geq n+\frac{q-p-a-c+2}{2}$ only.
\end{itemize}
Note that $e-1<c$ is equivalent to $a+p<q+c$. By using this and $c\leq p$
one gets $a+c\leq a+p<q+c\leq p+q$.
\end{itemize}
\item If  $0<a\le p<c$,   $M_1\cdot M_{2}$ is in one of the following forms:
\begin{itemize}
 \item[({IV})] \label{IV case}  We have $a\leq p<c\leq f=\frac{p+q+c-a}{2}$ and 
  \begin{align}
 M_1\cdot M_{2}=Y_{p}(m)\cdot Y_{a}(q-a+1)\cdot \frac{Y_{f}(q+p-a-f+1)}{Y_{p}(q-a+1)}\cdot 
\frac{Y_{c}(q+p-a-2f+n+1)}{Y_{f}(q+p-a-2f+n+1)}.
 \end{align}
\begin{itemize}
\item[(i)]\label{f=n} If $a<p<c<f$,  
by the highest weight condition, comparing $Y_{f}$ we get  
\begin{align}
\nonumber
0\geq(q+p-a-2f+n+1)-(q+p-a-f+1)=n-f.
\end{align}
By the condition $( \divideontimes)$ in Lemma \ref{decomposition}, one has $f\le n$, then the only possibility is  $n=f$. And comparing $Y_{p}$, we find $m\geq q-a+1\left(=\frac{q-p-a-c}{2}+n+1\right)$. 
\item[(ii)] If $a=p<c<f$, 
{ by the same argument as the previous one, we have $f=n$, which means $n=\frac{p+q+c-a}{2}
=\frac{q+c}{2}$. This derives a contradiction to the fact $c<n$ and $q\leq n$}.
Thus, this case cannot happen.
\item[(iii)] If $a<p<c=f$, we obtain $m\geq q-a+1$.
\item[(iv)] If $a=p<c=f$,  we get $c=q$ and then $M_1\cdot M_{2}=Y_{p}(m)\cdot Y_{q}(1)$, which is always a highest weight vector.
\end{itemize}
 \item[({V})] We have $a\leq h=\frac{p+q-a+c}{2}\leq p<c$.
\begin{align}\label{case V}
M_1\cdot M_{2}=Y_{p}(m)\cdot Y_{h}(q-h+1)\cdot 
\frac{Y_{c}(q+p-h-c+1)}{Y_{p}(q-h+1)}\cdot 
\frac{Y_{a}(q+p-2h-c+n+1)}{Y_{h}(q+p-2h-c+n+1)}.
 \end{align}
\begin{itemize}
\item[(i)] If $a<h<p<c$, 
{ comparing $Y_h$, we obtain
\[
0\geq (q+p-2h-c+n+1)-(q-h+1)=p-h-c+n=\frac{p-q-a-c}{2}+n.
\]
Then, we get $\frac{q+a+c-p}{2}\geq n\geq \frac{p+q+c-a}{2}$ by the condition in Lemma~\ref{decomposition}, which means $a\geq p$ and then derives a contradiction to $a<p$. Thus, 
this case cannot happen.}
\item[(ii)] If $a<h=p<c$, by the condition $( \divideontimes)$ in Lemma \ref{decomposition} we have $p+c=q+a$.  In (\ref{case V}), 
we find $ Y_{h}(q-h+1)\cdot \frac{Y_{c}(q+p-h-c+1)}{Y_{p}(q-h+1)}=Y_{c}(q+p-h-c+1)$. 
By the highest weight condition, comparing $Y_{p}$ we get 
$m\geq q+p-2h-c+n+1=n-a+1$. Then we have $p+c=q+a$ and $m\geq n-a+1$.
\item[(iii)] If $a=h< p<c$, then we have $m\geq q-h+1=q-a+1$.
\item[(iv)] If $a=h=p<c$,  we get $c=q$ and 
$M_1\cdot M_{2}=Y_{p}(m)\cdot Y_{q}(1)$, which is always the highest weight vector.
\end{itemize}
Note that $a=h$ (resp. $h=p$) is equivalent to $a+c=p+q$ (resp. $a+q=c+p$).
 \item[(VI)] We have $a\leq e-1=\frac{p+a+c-q}{2}\leq p<c$ and 
 \begin{align}
 M_1\cdot M_{2}=Y_{p}(m)\cdot Y_{a}(q-a+1)\cdot \frac{Y_{c}(q+e-a-c)}{Y_{e-1}(q-a+1)}\cdot \frac{Y_{e-1}(q+e+n-a-c-p)}{Y_{p}(q+e+n-a-c-p)}.
 \end{align}
\begin{itemize}
\item[(i)]   If $a<e-1<p<c$, 
by the highest weight condition,  comparing $Y_{e-1}$, we get 
\begin{align}
\nonumber
0\geq(q-a+1)-(q+e+n-a-c-p)=\frac{p+q+c-a}{2}- n.
\end{align}
by the condition $( \divideontimes)$ in Lemma \ref{decomposition}, $\frac{p+q+c-a}{2}\le n$ is always satisfied and comparing $Y_{p}$, we obtain $m\geq q+e+n-a-c-p= \frac{q-a-c-p+2}{2}+n$. 
\item[(ii)] If $a<e-1=p<c$, comparing $Y_p(=Y_{e-1})$, we get $m\geq q-a+1$.
\item[(iii)] If $a=e-1<p<c$, we have $a+q=c+p$, $a+c<p+q$ and $m\geq \frac{q-p-a-c}{2}+n+1$ by comparing
$Y_p$. 
\item[(iv)] If $a=e-1=p<c$, we get $a=p$ and $c=q$. In this case, $M_1\cdot M_2$ is always a highest weight
vector, which is included in the previous case.
\end{itemize}
Note that $a=e-1\Leftrightarrow a+q=p+c$ and $e-1=p\Leftrightarrow a+c=p+q$.
 \item [({VII})] Since $c-d=\frac{c+p-a-q}{2}\geq 0$, we have $c\geq d$. 
 The explicit form of 
 $M_2$ we have $d\geq c$. Thus, we have $c=d$ and then $a\leq p<c=d$.
\begin{align}
\nonumber
M_1\cdot M_{2}=&Y_{p}(m)\cdot Y_{d}(q-d+1)\cdot  
\frac{Y_{c}(q-2d+n+1)}{Y_{d}(q-2d+n+1)}
\frac{Y_{a}(q-p-2d+c+n+1)}{Y_{p}(q-p-2d+c+n+1)}
\end{align}
\begin{itemize}
\item[(i)]  If $a< p<c=d$, by the highest weight condition, comparing $Y_p$ we get
$m\geq q-p-2d+c+n+1=\frac{q-p-a-c}{2}+n+1$. Note that $c=d\Leftrightarrow a+q=c+p$. Thus 
we get 
\[
p+q>a+q=c+p>c+a.
\]
\item[(ii)] If $a=p<c=d$, we have $M_2=Y_d(q-d+1)=Y_q(1)$. Then $M_1\cdot M_2$ is always highest weight
vector.  Then, we get $a=p$ and $c=q$, which is included in the 
previous case.

\end{itemize}
\end{itemize}
\item If  $a=0< c\le p$,  $M_1\cdot M_2$ is in one of the following forms:
\begin{itemize}
 \item[({VIII})] We get $c\leq p\leq g=\frac{p+q+c}{2}$. 
 \begin{align}
\nonumber
M_1\cdot M_{2}=&Y_{p}(m)\cdot\frac{Y_{g}(q+p-g+1)}{Y_{p}(q+1)}\cdot \frac{Y_{c}(q+p+n-2g+1)}{Y_{g}(q+p+n-2g+1 )}.
 \end{align} 
\begin{itemize}
\item[(i)] If $c<p<g$, 
by the highest weight condition, comparing $Y_{g}$ we get 
\begin{align}
\nonumber
0\geq(q+p+n-2g+1 )-(q+p-g+1)=n-g.
\end{align}
By the condition $( \divideontimes)$ in Lemma \ref{decomposition}, one has $g\le n$, then the only possibility is $g=n$. 
 And comparing  $Y_{p}$, we obtain $m\geq q+1(=\frac{q-p-a-c}{2}+n+1)$. 
\item[(ii)] If $c<p=g$, we have $p=q+c$ and 
\begin{align}
\nonumber
M_1\cdot M_{2}=&Y_{p}(m)\cdot \frac{Y_{c}(q+p+n-2g+1)}{Y_{p}(q+p+n-2g+1)}
 \end{align} 
By the highest weight condition, comparing $Y_{p}$ we get 
$m\geq q+p+n-2g+1=n-c+1(=\frac{q-p-a-c}{2}+n+1$).
\item[(iii)] If $c=p<g$, by the highest weight condition, comparing $Y_{g}$ we have $g=n$ 
by the similar way to the previous case.  And comparing  $Y_{p}$, we get  $m\geq q+1$ or 
\begin{align}\label{p=n cannot happen}
0\geq(q+1 )-(q+p+n-2g+1)=n-p.
\end{align}
By the condition $( \divideontimes)$ in Lemma \ref{decomposition}, one has $p\le n$, then the only possibility is $p=n$, which contradicts  { $p<g=n$} then (\ref{p=n cannot happen}) cannot happen. Thus, we have $g=n$ and $m\geq q+1(=\frac{q-p-a-c}{2}+n+1)$.
\item[(iv)] If $c=g=p$, { then we have $c=p+q$ and $p=q+c$ and then $q=0$, 
which contradicts $q\geq1$. Thus, this case cannot occur.}
\end{itemize}
Note that $c=g\Leftrightarrow c=p+q$ and $g=p\Leftrightarrow p=q+c$.
\item[({IX})] We have $b-1=\frac{p+c-q}{2}\leq c\leq p$. 
\begin{align}
\nonumber
M_1\cdot M_{2}=&Y_{p}(m)\cdot \frac{Y_{c}(q+b-c)}{Y_{b-1}(q+1)}\cdot 
\frac{Y_{b-1}(q+b+n-c-p)}{Y_{p}(q+b+n-c-p )}.
 \end{align}
\begin{itemize}
\item[(i)] If $b-1<c<p$, then by the highest weight condition, comparing $Y_{b-1}$ if $c\ne p$ we get 
\begin{align}
\nonumber
0\le(q+b+n-c-p)-(q+1)=n-\frac{p+q+c}{2}.
\end{align}
By the condition $( \divideontimes)$ in Lemma \ref{decomposition}, $\frac{p+q+c-a}{2}\le n$ is always satisfied and  comparing $Y_{p}$, we obtain $m\geq\frac{q-c-p+2}{2}+n$. 
\item[(ii)] If $b-1<c=p$, by the highest weight condition, comparing $Y_p$, we get 
$q+b-c\geq q+b+n-c-p\Leftrightarrow p\geq n$ or $m\geq q+b+n-c-p$. 
The condition $p\geq n$ implies $p=n$. 
{ Now, we have $c=p=n$. But, by the condition $( \divideontimes)$ in Lemma \ref{decomposition}, 
we have $\frac{p+q+c-a}{2}=\frac{p+q+c}{2}\leq n$. These results mean $q\leq 0$, which contradicts $q\geq1$. Thus, $p=n$ cannot occur.}
Comparing $Y_{b-1}$ we have the same result as above.
Thus, we get $m\geq \frac{q-c-p+2}{2}+n$.
\item[(iii)] If $b-1=c<p$, then by the highest weight condition, comparing $Y_{p}$ we have 
$m\geq \frac{q-c-p+2}{2}+n$.
\item[(iv)] If $b-1=c=p$, then we get $q=0$ and then this case does not happen.
\end{itemize}
\end{itemize}
\item If   $a=0\le p< c$, $M_{2}$ has the following forms:
\begin{itemize}
\item[({X})]  
we have $b-1\leq p<c$. 
\begin{align}   
M_1\cdot M_{2}=&Y_{p}(m)\cdot \frac{Y_{c}(q+b-c)}{Y_{b-1}(q+1)}\cdot 
\frac{Y_{b-1}(q+b+n-c-p)}{Y_{p}(q+b+n-c-p )}.
 \end{align}
\begin{itemize}
\item[(i)] If $b-1<p<c$, 
by the highest weight condition, comparing $Y_{b-1}$ we get 
\begin{align}
\nonumber
0\geq(q+1)-(q+b+n-c-p)=\frac{q+p+c}{2}-n.
\end{align}
By the condition $( \divideontimes)$ in Lemma \ref{decomposition}, $\frac{p+q+c-a}{2}\le n$ is always satisfied and comparing $Y_{p}$ we find  $m\geq n+\frac{q-c-p+2}{2}$. 
\item[(ii)] If $b-1=p<c$,  
$M_2=X_{p+1}(q)\cdots X_c(1)=\frac{Y_c(1)}{Y_{p}(q+1)}$ and then comparing $Y_p$ 
  we get $m\geq q+1$. Note that $b-1=p$ is equivalent to $c=p+q$.
\end{itemize}
\item[(XI)] 
We obtain $p<c\leq g$. 
\begin{align}     
\nonumber
 M_{2}=&X_{p+1}( q )\cdots X_{c}( q+p-c +1)X_{c+1}( q+p-c  )\cdots X_{g}( q+p-g +1)\\
\nonumber
&X_{\bar{g}}(q+p-g)\cdots X_{\overline{c+1}}( q+p+c-2g+1).\\
M_1\cdot M_{2}=&Y_{p}(m)\cdot \frac{Y_{g}(q+p-g+1)}{Y_{p}(q+1)}\cdot 
\frac{Y_{c}(q+p+n-2g+1)}{Y_{g}(q+p+n-2g+1 )},
 \end{align}
where $g=\frac{q+p+c}{2}$.
\begin{itemize}
\item[(i)] If $p<c<g$, by the highest weight condition, comparing $Y_{g}$ we get
\begin{align}
\nonumber
0\geq(q+p+n-2g+1 )-(q+p-g+1)=n-g.
\end{align}
By the condition $( \divideontimes)$ in Lemma \ref{decomposition}, one has $g\le n$, then the only possibility is $n=g$. And comparing $Y_{p}$, we obtain   $m\geq q+1(=\frac{q-p-a-c+2}{2}+n)$.  
\item[(ii)] If $p<c=g$, we find $\frac{Y_{c}(q+p+n-2g+1)}{Y_{f}(q+p+n-2g+1 )}=1$. By the highest weight condition, comparing $Y_{p}$ we get  $m\geq q+1(=q-a+1)$.  Note that $c=g$ is equivalent to $c=p+q$.
\end{itemize}
\end{itemize}
\item If  $a=c=0$, then 
we have $p=q$ by the condition $a+p\le c+q$ and $a+q\le c+p$. Thus, 
$M_{2}$ is in the following form:
\begin{itemize}
\item[({XII})] 
$ M_{2}=
[X_{\bar{p}}(p)\cdots X_{\bar{1}}(1)]=\frac{1}{Y_p(n+1)}$. Then we have 
$M_1\cdot M_{2}=Y_{p}(m)\cdot \frac{1}{Y_{p}(n+1)}$ and  
by the highest weight condition, comparing $Y_{p}$ we get $m\geq n+1(=\frac{q-p-a-c+2}{2}+n)$. 
 \end{itemize}
 \end{enumerate}
\end{proof}
For a statement $P$, let $\mathcal{L}(P)=
       \left\{\begin{array}{rcl}
       1 & \mbox{if} & P\,\, is \,true,\\ 0 & \mbox{if} & P\,\, is \,false.
        \end{array}\right.$
                
\begin{theorem} \label{smy.deco.}
                         \begin{align}
                   \nonumber
                  \mathcal{M}(Y_{p}(m) ) \cdot\mathcal{M}(Y_{q}( 1 ))
                   &\cong   B(\Lambda_{p}+\Lambda_{q})\\
\nonumber
&\oplus\bigoplus \limits_{\substack{(a,c)\ne(p,q),(q,p)\\
0\le a\le p,\,\,a+p\le c+q,\,\,(p+q)-(a+c)\in 2\mathbb{ Z}_{>0}\\a\le c\le n,\,\,a+q\le c+p,\,\,\frac{p+q+c-a}{2}\le n.} }B(\Lambda_{a}+\Lambda_{c})\cdot  \mathcal{L}(m\geq \frac{q-p-a-c+2}{2}+n)\\
              \nonumber
                &\oplus\bigoplus\limits_{\substack{(a,c)\ne(p,q),(q,p)\\0\le a\le p,\,\,a+p\le c+q,\,\,p+q=a+c,\\
a\le c\le n,\,\,a+q\le c+p,\,\,\frac{p+q+c-a}{2}\le n.}}
B( \Lambda_{a} +\Lambda _{c} )\cdot \mathcal{L}(m\geq q-a+1), 
                    \end{align}
where note that $\Lm_0$ means $0$.               
\end{theorem}

\begin{proof}
By Theorem \ref{forms thm}, we can obtain the following:
\begin{enumerate}
\item For $ B(\Lambda_{p}+\Lambda_{q})$ we find $M_1\cdot M_2=Y_p(m)\cdot Y_q(1)$, which is always highest weight vector with
\begin{itemize}
\item[(i)]  if $0<a\le c\le p$, we have $a=q$ and $c=p$ or $a=p$ and $c=q$.
\item[(ii)]  if $0<a\le p<c$,  we have $a=p$ and $c=q$ with $p<q$.
\end{itemize}
\item For $(p+q)-(a+c)\in 2\mathbb{ Z}_{>0}$ we find that the condition on $m$ 
is $m\geq \frac{q-p-a-c+2}{2}+n$ from the proof of the Theorem \ref{forms thm}
({I})(iii)(iv)(v), 
({II}){ (i)}(ii), 
({III})(i)(ii)(iii), 
({IV})(i), 
({V})(ii), 
({VI})(i)(iii),  
({VII})(i), 
({VIII})(i)(ii)(iii), 
({IX})(i)(ii)(iii), 
({X})(i), 
({XI})(i), 
({XII}).
Thus, all these cases cover the corresponding range for the second direct sum 
of the formula in the theorem.

Let us check that there is no multiplicity on all $(a; c)$’s as above, namely, there is no
intersection on the above conditions in 
(1)--(5), or in the case that there exists an intersection
among them, they give the same monomial $M_2$. 
Then, let us see the following cases, where there might be possibly intersection among them.
\begin{itemize}
\item[(a)] (I)(v), (III)(ii): In (I)(v), we have $c=d=\frac{q-p+a+c}{2}=n$ and in (III)(ii) we get 
$a=e-1=\frac{p-q+a+c}{2}=d$. Then in both cases, $M_2$ is equal to the monomial
$\frac{Y_a(q-p-c+n+1)Y_c(q-c+1)}{Y_p(q-p-c+n+1)}$. 
\item[(b)] { (II)(i), (III)(i): From (II)(i), we have $f=\frac{p+q+c-a}{2}=n$. 
Then, we have $q+p-a-2f+n+1=q+p-a-f+1$, which implies that in (II)(i) 
$M_2=Y_a(q-a+1)\frac{Y_c(q-a+1)}{Y_p(q-a+1)}$. 
We also get $q-p-a-c+e+n=q-a+1$, which shows that in (III)(i), $M_2=Y_a(q-a+1)\frac{Y_c(q-a+1)}{Y_p(q-a+1)}$. Thus, in both cases, we obtain the same $M_2$.}
\item[(c)] (IV)(i) (VI)(i): Under the condition of (IV)(i), we have $f=\frac{p+q+c-a}{2}=n$ and then 
$q+e+n-a-c-p=q-a+1$ and $q+p-a-n+1=q+e-a-c$, where $e$ is in the same form as the previous case. Therefore, one obtains
that in both cases $M_2$ is in the form $\frac{Y_a(q-a+1)Y_c(q-p-a-n+1)}{Y_p(q-a+1)}$.
\item[(d)] (V)(ii), (VI)(iii), (VII)(i):
Under the conditions of all three cases, we have 
$p=h=\frac{p+q+a-c}{2}$, $a=e-1=\frac{p-q+a+c}{2}$ and $c=d=\frac{q-p+a+c}{2}$. Thus, 
one gets that in all three cases $M_2$ is in the same form
$M_2=\frac{Y_c(q-c+1)Y_a(q-p-c+n+1)}{Y_p(q-p-c+n+1)}$. 
\item[(e)] (VIII)(i), (IX)(i): 
In the case (VIII)(i) we have $g=\frac{p+q+c}{2}=n$ and then we get $q+b+n-c-p=q+1$ and $q+b-c=q+p-n+1$, 
where $b=\frac{p-q+c+2}{2}$. Therefore, 
one  obtains that in both cases, $M_2$ is in the same form 
$ \frac{Y_c(q+p-n+1)}{Y_p(q+1)}$.
\item[(f)] (VIII)(ii), (IX)(iii): 
Under the conditions of the both cases, we have $p=g=\frac{p+q+c}{2}=b-1=\frac{p-q+c}{2}$ and 
{ then 
$q+b+n-c-p=q+b-c=q+p-n+1=q+1$. 
Thus, one gets that the monomial $M_2$ is in the same form
$\frac{Y_c(q+1)}{Y_p(q+1)}$}. 
\item[(g)] (VIII)(iii), (IX)(ii): 
Under the conditions of both cases, we have $c=p$, $g=\frac{p+q+c}{2}$ and $b-1=\frac{p-q+c}{2}$ and then 
{ we get $q+p+n-2g+1=q+b+n-c-p=q+1$ and $q+b-c=\frac{q}{2}+1$. 
Thus, one obtains that in both cases, 
$M_2=\frac{Y_c(\frac{q}{2}+1)}{Y_p(q+1)}$. 
Note that by the condition of 
$(\divideontimes)$ in Lemma \ref{decomposition}, $q=p+q-a-c\in2\bbZ_{\geq0}$  and then
$q$ is even.}
\item[(h)] (X)(i), (XI)(i): Under the conditions of (XI)(i), we have 
$g=\frac{p+q+c}{2}=n$ and then $q+b+n-c-p=q+1$ and $q+b-c=q+p-n+1$.
Therefore, one gets that in both cases, $M_2$ is in the same form as 
$\frac{Y_c(q+p-n+1)}{Y_p(q+1)}$.
\end{itemize} 
Therefore, in these cases there is no multiplicity on any $(a,c)$. 

\item For $p+q=a+c$ we find that the condition on $m$ is $m\geq q-a+1$ from the proof of the Theorem \ref{forms thm}
({IV})(iii), ({V})(iii), (VI)(ii), ({X})(ii), ({XI})(ii).
Here, indeed, one finds that in these cases always $a<p<c$, which is derived from the condition 
$(a,c)\ne(p,q),\,(q,p)$, $a\leq p$, $a+c=p+q$ and $a+q\leq c+p$. Thus, all these cases cover the corresponding range for the last sum of the formula in the theorem.

Let us check that there is no multiplicity on all $(a,c)$'s in this case (3).
There would be the following cases:
\begin{itemize}
\item[(a)] (IV)(iii), (V)(iii), (VI)(ii): 
Under the conditions of these three cases, we have $c=f=\frac{p+q-a+c}{2}$, $a=h=\frac{p+q+a+c}{2}$ and
$p=e-1=\frac{p-q+a+c}{2}$, which are equivalent to $a+c=p+q$. And then we get 
\[
q+p-a-f+1=q+p-h-c+1=q+e-a-c=1
\]
Therefore, one obtains that in all three cases, $M_2$ is in the same form as 
$\frac{Y_a(q-a+1)Y_c(1)}{Y_p(q-a+1)}$. 
\item[(b)] (X)(ii), (XI)(ii): Under the conditions $a=0$, 
$p=b-1=\frac{p+c-q}{2}$ and $c=g=\frac{p+q+c}{2}$, 
we get $q+b-c=q+p-g+1=1$. Thus, one gets that in both cases
$M_2$ is in the same form as $\frac{Y_c(1)}{Y_p(q+1)}$.
\end{itemize}
\end{enumerate}
Here we know that there is no multiplicity on any $(a,c)$'s.
\end{proof}    

\begin{remark}
This theorem implies that the parameter $m$ seems to be some `distance' 
between two crystals ${\mathcal M}(Y_p(m))$
and ${\mathcal M}(Y_q(1))$. Indeed, if $m$ is sufficient large, then the product 
${\mathcal M}(Y_p(m))\cdot {\mathcal M}(Y_q(1))$ is the same as their tensor product 
and if $m=1$, then the product is the same as a single crystal $B(\Lm_p+\Lm_q)$.

\end{remark}

\begin{example}
For type $C_5$, 
take  $p=q=3$. By Lemma~\ref{decomposition}, we get 
\begin{align}
                   \nonumber
B(\Lambda_{3})\otimes B(\Lambda_{3})&\cong B(2\Lambda_{3})\oplus 
B(\Lambda_{2}+\Lambda_{4})\oplus B(\Lambda_{1}+\Lambda_{5})\oplus 
B(\Lambda_{1}+\Lambda_{3})\oplus B(\Lambda_{4})\oplus B(2\Lambda_{2})\\
   \nonumber
&\oplus B(\Lambda_{2})\oplus B(2\Lambda_{1})\oplus B(0),
\end{align}
( dimension  is $110\times 110=4004+5005+1155+891+780+165+55+44+1=12100)$.
\begin{itemize}
\item The component $B(2\Lambda_{3})$ 
always appears in $\mathcal{M}(Y_{3}(m) ) \cdot\mathcal{M}(Y_{3}( 1 ))$.
\item For $B(\Lambda_{2}+\Lambda_{4})$, $a=2, c=4$ satisfy $p+q=a+c$ 
then the condition for appearance in $\mathcal{M}(Y_{3}(m) ) \cdot\mathcal{M}(Y_{3}( 1 ))$ 
is $m\geq q-a+1=3-2+1=2$.
\item For $B(\Lambda_{1}+\Lambda_{5})$, $a=1, c=5$ satisfy $p+q=a+c$ 
then the condition for appearance in $\mathcal{M}(Y_{3}(m) ) \cdot\mathcal{M}(Y_{3}( 1 ))$ 
is $m\geq q-a+1=3-1+1=3$. 
\item For $B(\Lambda_{1}+\Lambda_{3})$, $a=1, c=3$ satisfy $p+q>a+c$ 
then the condition for appearance in $\mathcal{M}(Y_{3}(m) ) \cdot\mathcal{M}(Y_{3}( 1 ))$ 
is $m\geq \frac{q-p-a-c+2}{2}+n=\frac{3-3-1-3+2}{2}+5=4$.
\item For $B(\Lambda_{4})$, $a=0, c=4$ satisfy $p+q>a+c$ satisfy $p+q>a+c$ 
then the condition for appearance in $\mathcal{M}(Y_{3}(m) ) \cdot\mathcal{M}(Y_{3}( 1 ))$
is $m\geq \frac{q-p-a-c+2}{2}+n=\frac{3-3-0-4+2}{2}+5=4$. 
\item For $B(2\Lambda_{2})$, $a=c=2$  satisfy $p+q>a+c$ 
then the condition for appearance in $\mathcal{M}(Y_{3}(m) ) \cdot\mathcal{M}(Y_{3}( 1 ))$ 
is  $m\geq \frac{q-p-a-c+2}{2}+n=\frac{3-3-2-2+2}{2}+5=4$. 
\item For $B(\Lambda_{2})$, $a=0,c=2$ satisfy $p+q>a+c$ 
then the condition for appearance in $\mathcal{M}(Y_{3}(m) ) \cdot\mathcal{M}(Y_{3}( 1 ))$ 
is $m\geq \frac{q-p-a-c+2}{2}+n=\frac{3-3-0-2+2}{2}+5=5$.
\item For $B(2\Lambda_{1})$, $a=c=1$ satisfy $p+q>a+c$  
then the condition for appearance in $\mathcal{M}(Y_{3}(m) ) \cdot\mathcal{M}(Y_{3}( 1 ))$ 
is  $m\geq \frac{q-p-a-c+2}{2}+n=\frac{3-3-1-1+2}{2}+5=5$.
\item For $B(0)$, $a=c=0$ satisfy $p+q>a+c$ 
then the condition for appearance in $\mathcal{M}(Y_{3}(m) ) \cdot\mathcal{M}(Y_{3}( 1 ))$ 
is $m\geq \frac{q-p-a-c+2}{2}+n=\frac{3-3-0-0+2}{2}+5=6$.
\end{itemize}
Then we have
\begin{itemize}
\item[(a)] If $m\geq6$,  $\mathcal{M}(Y_{3}(m) ) \cdot\mathcal{M}(Y_{3}( 1 ))
\cong  B(2\Lambda_{3})\oplus B(\Lambda_{2}+\Lambda_{4})\oplus 
B(\Lambda_{1}+\Lambda_{5})\oplus B(\Lambda_{1}+\Lambda_{3})\oplus B(\Lambda_{4})\oplus 
B(2\Lambda_{2})\oplus B(\Lambda_{2})\oplus B(2\Lambda_{1})\oplus B(0)$.
\item[(b)] If $m=5$,  $\mathcal{M}(Y_{3}(m) ) \cdot\mathcal{M}(Y_{3}( 1 ))
\cong  B(2\Lambda_{3})\oplus B(\Lambda_{2}+\Lambda_{4})\oplus B(\Lambda_{1}+\Lambda_{5})
\oplus B(\Lambda_{1}+\Lambda_{3})\oplus B(\Lambda_{4})\oplus B(2\Lambda_{2})
\oplus B(\Lambda_{2})\oplus B(2\Lambda_{1})$.
\item[(c)] If $m=4$,  $\mathcal{M}(Y_{3}(m) ) \cdot\mathcal{M}(Y_{3}( 1 ))
\cong  B(2\Lambda_{3})\oplus B(\Lambda_{2}+\Lambda_{4})\oplus B(\Lambda_{1}+\Lambda_{5})
\oplus B(\Lambda_{1}+\Lambda_{3})\oplus B(\Lambda_{4})\oplus B(2\Lambda_{2})$.
\item[(d)] If $m=3$,  $\mathcal{M}(Y_{3}(m) ) \cdot\mathcal{M}(Y_{3}( 1 ))
\cong  B(2\Lambda_{3})\oplus B(\Lambda_{2}+\Lambda_{4})\oplus B(\Lambda_{1}+\Lambda_{5})$.
\item[(e)] If $m=2$,  $\mathcal{M}(Y_{3}(m) ) \cdot\mathcal{M}(Y_{3}( 1 ))
\cong B(2\Lambda_{3})\oplus B(\Lambda_{2}+\Lambda_{4})$.
\item[(f)] If $m=1$, $\mathcal{M}(Y_{3}(m) ) \cdot\mathcal{M}(Y_{3}( 1 ))
\cong B(2\Lambda_{3})$.
\end{itemize}
\end{example}
\begin{example}
By considering similarly to the previous example, for type $C_5$ and $p=4,\,q=5$, we obtain:
\begin{itemize}
\item[(a)] If $m\geq 5$, we have $ \mathcal{M}(Y_{4}(m) ) 
\cdot\mathcal{M}(Y_{5}( 1 ))\cong B(\Lambda_{4}+\Lambda_{5})\oplus 
B(\Lambda_{3}+\Lambda_{4})\oplus B(\Lambda_{2}+\Lambda_{3})\oplus B(\Lambda_{1}+\Lambda_{2})\oplus B(\Lambda_{1})$.
\item[(b)] If $m=4$,  $ \mathcal{M}(Y_{4}(m) ) \cdot\mathcal{M}(Y_{5}( 1 ))
\cong B(\Lambda_{4}+\Lambda_{5})\oplus B(\Lambda_{3}+\Lambda_{4})\oplus B(\Lambda_{2}+\Lambda_{3})\oplus B(\Lambda_{1}+\Lambda_{2})$.
\item[(c)] If $m=3$, $ \mathcal{M}(Y_{4}(m))\cdot\mathcal{M}(Y_{5}(1))
\cong B(\Lambda_{4}+\Lambda_{5})\oplus B(\Lambda_{3}+\Lambda_{4})\oplus B(\Lambda_{2}+\Lambda_{3})$.
 \item[(d)] If $m=2$,  $ \mathcal{M}(Y_{4}(m) ) \cdot\mathcal{M}(Y_{5}( 1 ))
 \cong B(\Lambda_{4}+\Lambda_{5})\oplus B(\Lambda_{3}+\Lambda_{4})$.
\item[(e)] If $m=1$, $ \mathcal{M}(Y_{4}(m) ) \cdot\mathcal{M}(Y_{5}( 1 ))
\cong B(\Lambda_{4}+\Lambda_{5})$.
\end{itemize}
\end{example}

\create{thebibliography}{9}
\addcontentsline{toc}{section}{References}
\vspace{-1ex}
\bibitem{AN} M. Alshuqayr and T.Nakashima, 
Decomposition Theorem for Product of Fundamental Crystals in Monomial Realization, 
Tokyo J.Math. 43,  no.1, (2020), 239--258.

\bibitem {BCD}S-J.Kang,  J-A.Kim,  D-U.Shin, Crystal Bases for quantum classical algebras and
Nakajima's monomials, Publ.RIMS, Kyoto Univ. 40 (2004), 757--791.

\bibitem{MK} M.Kashiwara, Realization of Crystals, 
in: Contemp. Math., vol. 325,Amer. Math. Soc., 2003, pp. 133-139.

\bibitem{M} M.Kashiwara, Crystal base and Littelmann's refined Demazure character formula, 
 Duke Math. J. 71(1993)-839-858.

\bibitem{M1} M.Kashiwara, Crystallizing the $q$-analogue of universal enveloping  algebras,  
Comm. Math. Phys. 133(1990), 249-260.

\bibitem{M2} M.Kashiwara,  On crystal bases of the $q$-analogue of 
universal enveloping  algebras,  Duke Math. J. 63(1991), 456-516.

\bibitem{MK-N}  M.Kashiwara, T.Nakashima, Crystal Graphs for Representations of the 
$q$-Analogue of Classical Lie Algebras,  J. of Algebra, 165(1994), 295-345.

\bibitem{KJ} J.Kamnitzer, P.Tingley, B.Webster, A.Weekes, O.Yacobi, 
Highest weights for truncated shifted Yangians and product monomials, 
J.comb. Algebra 3(2019), no3, 237-303.

\bibitem{TN} T.Nakashima, Crystal base a generalization of the Littlewood-Richardson 
rule for the classical Lie algebras,  Commun. Math. Phys. 154 (1993), no. 2, 215-243.

\bibitem{N} H.Nakajima, $t$-Analogs of $q$-Characters of Quantum affine  Algebras of 
Type $A_{N}$, $D_{N}$,  in: Contemp. Math., vol.325, Amer. Math. Soc., (2003), 141-160.

\delete{thebibliography}

\end{document}